\documentclass[11pt,reqno]{NumPDEsArticle}

\usepackage[T1]{fontenc}
\usepackage[english]{babel}

\usepackage{NumPDEsMacros}

\usepackage{csquotes}
\usepackage{enumitem}
\usepackage{amsmath,amssymb,mathrsfs}
\usepackage{empheq}
\usepackage{graphicx}
\usepackage{tikz}
\usepackage{caption}
\usepackage{subcaption}
\usepackage{booktabs}
\usepackage{graphicx}
\usepackage[capitalize,nameinlink]{cleveref}

\usepackage{cleveref} 

\usepackage{pgfplots}
\usepackage{pgfplotstable}
\usepackage{placeins}
\usepgfplotslibrary{groupplots}
\pgfplotsset{compat=1.18}

\DeclareMathAlphabet\mathbfcal{OMS}{cmsy}{b}{n}

\title[BDF2-type integrator for LLG, part III]{BDF2-type integrator for Landau--Lifshitz--Gilbert equation in micromagnetics:\\ IMEX time-stepping and full effective field}

\author{Michele Ald\'e}
\email{Michele.Alde@asc.tuwien.ac.at \quad\rm (corresponding author)}

\author{Dirk Praetorius}
\email{Dirk.Praetorius@asc.tuwien.ac.at}

\author{Agus L. Soenjaya}
\email{Agus.Soenjaya@asc.tuwien.ac.at}

\address{TU Wien, Institute of Analysis and Scientific Computing, Wiedner Hauptstra\ss{}e 8--10, 1040 Wien, Austria}

\thanks{This research was funded by the Austrian Science Fund (FWF) projects \href{https://www.doi.org/10.55776/F65}{10.55776/F65} (SFB F65 ``Taming complexity in PDE systems''), and \href{https://www.doi.org/10.55776/I6802}{10.55776/I6802} (international project I6802 ``Functional error estimates for PDEs on unbounded domains''). Additionally, the \href{https://www.vsmath.at}{Vienna School of Mathematics} supports Michele Ald\'e.}

\newtheorem*{proposition*}{Proposition}

\newcommand{\mm}{\boldsymbol{m}}
\newcommand{\heff}{\boldsymbol{h}_{\rm eff}}
\newcommand{\aaa}{\boldsymbol{a}}

\newcommand{\bb}{\boldsymbol{b}}
\newcommand{\ff}{\boldsymbol{f}}
\newcommand{\jj}{\boldsymbol{j}}
\newcommand{\hh}{\boldsymbol{h}}
\newcommand{\vv}{\boldsymbol{v}}

\newcommand{\pp}{\boldsymbol{p}}

\newcommand{\rr}{\boldsymbol{r}}
\newcommand{\ee}{\boldsymbol{e}}
\newcommand{\sss}{\boldsymbol{s}}
\newcommand{\nuu}{\boldsymbol{\nu}}

\newcommand{\bpsi}{\boldsymbol{\psi}}
\newcommand{\bpi}{\boldsymbol{\pi}}

\newcommand{\lamex}{\lambda_{\rm ex}}
\newcommand{\lDM}{\lambda_{\rm DM}}
\newcommand{\bT}{\boldsymbol{T}}
\newcommand{\bL}{\mathbf{L}}
\newcommand{\bH}{\mathbf{H}}
\newcommand{\bW}{\mathbf{W}}
\newcommand{\OT}{\Omega_T}

\newcommand{\calQ}{\mathcal{Q}}
\newcommand{\mhatp}{\widehat{\mm}_{h\tau}^+}

\newcommand{\scalarproductO}[2]{\langle #1, #2 \rangle_{\Omega}}

\newcommand{\normO}[2][]{\norm[#1]{#2}_{\Omega}}
\newcommand{\normOT}[2][]{\norm[#1]{#2}_{\Omega_T}}
\newcommand{\HDM}{\mathcal{H}_{\rm DM}}
\newcommand{\zz}{\boldsymbol{z}}

\newcommand{\bR}{\mathbf{R}}

\newcommand{\Dtau}{\mathrm{D}_t}

\definecolor{TUPlotBlue}{HTML}{006699}
\definecolor{TUPlotRed}{HTML}{A92E25}
\definecolor{TUPlotGreen}{HTML}{35948B}
\definecolor{TUPlotOrange}{HTML}{E18922}
\definecolor{TUPlotPurple}{HTML}{BA4682}
\definecolor{TUPlotYellow}{HTML}{D8C35A}
\definecolor{TUPlotGrey}{HTML}{9D9D9C}
\definecolor{TUPlotCyan}{HTML}{72ADD5}
\definecolor{TUPlotBrown}{HTML}{9B7140}
\definecolor{TUPlotPink}{HTML}{E4A9C7}

\makeatletter
\newsavebox{\@brx}
\newcommand{\llangle}[1][]{\savebox{\@brx}{\(\m@th{#1\langle}\)}%
  \mathopen{\copy\@brx\kern-0.5\wd\@brx\usebox{\@brx}}}
\newcommand{\rrangle}[1][]{\savebox{\@brx}{\(\m@th{#1\rangle}\)}%
  \mathclose{\copy\@brx\kern-0.5\wd\@brx\usebox{\@brx}}}
\makeatother

\renewcommand{\set}[3][]{#1\{#2 \,:\, #3#1\}}
\newcommand{\xx}{\boldsymbol{x}}
\newcommand{\uu}{\boldsymbol{u}}
\def\III{\boldsymbol{\II}}
\def\SSS{\boldsymbol{\SS}}
\def\pphi{\boldsymbol{\phi}}
\def\vvarphi{\boldsymbol{\varphi}}

\begin{document}
	
\maketitle

\begin{abstract}
We extend the BDF2-type finite-element integrator proposed in \cite{part1}
to the Landau--Lifshitz--Gilbert (LLG) equation with lower-order effective-field contributions, including nonlinear anisotropies of arbitrary polynomial degree and current-induced non-conservative torques arising in micromagnetic applications. To incorporate these additional terms while retaining a linear solve at each time step, we employ an implicit--explicit (IMEX) time splitting. Our main theorem proves unconditional weak convergence of the fully discrete approximations towards weak solutions of the LLG equation. The analysis removes the mild but artificial CFL-type restrictions required at the first and final time steps in \cite{part1}. Furthermore, in the presence of the Dzyaloshinskii--Moriya interaction (DMI), the convergence is likewise unconditional and does not require the CFL-type restriction imposed in \cite{hpprss2019}. Numerical experiments support the theoretical findings.
\end{abstract}

\section{Introduction}
\textbf{State of the art and related works.}
Time-dependent micromagnetic phenomena are usually modeled by the
phenomenological Landau--Lifshitz--Gilbert (LLG) equation
~\cite{LandauA1935,gilbert2004}, a nonlinear time-dependent PDE
describing the evolution of the magnetization
$\mm(t,\xx)\in\R^3$ in a ferromagnetic body $\Omega\subset\R^d$, $d=2,3$. For constant temperatures below the Curie temperature, the magnetization has constant modulus $|\mm(t,\xx)|=M_s>0$, which after normalization reads $|\mm|=1$. The understanding of LLG underlies
important technical applications such as sensors, actuators, and
magnetic storage devices~\cite{BrownB1962,BrownB1963,hubert2008magnetic}.
More recently, magnetic skyrmions and other chiral spin textures have attracted considerable attention as potential information carriers in next-generation spintronic devices, including racetrack memories and non-conventional computing architectures~\cite{FertReyrenCros2017}.
Their nanoscale dimensions, metastability, and response to applied fields and spin-polarized currents make the understanding and reliable simulation of their dynamics particularly important.
For instance, the controlled current-driven motion of individual skyrmions has recently been demonstrated experimentally in FeGe nanotracks of approximately $100\,\mathrm{nm}$ width~\cite{SongEtAl2024}.

The mathematical theory of LLG, at least when the exchange field is dominant, is by now well developed from both the analytical and numerical perspectives; see Section~\ref{sec: physical background} for a precise formulation of the LLG model considered here. Regarding the analytical aspects, we refer
to~\cite{Alouges_Soyeur} for global existence and nonuniqueness of
weak solutions, to~\cite{csg1998,cf2001,cimrak2007,ft2017b,QiYamin2023}
for local-in time existence and regularity of strong solutions, and
to~\cite{ds2014,dip2020} for weak-strong uniqueness: for fixed data, a sufficiently
smooth strong solution coincides with every weak 
solution as long as the strong solution exists.

As far as the numerical analysis is concerned, the main difficulties
stem from the strong nonlinearity of the equation and the need to
preserve the physical structure of the problem in the discretizing
scheme, in particular the pointwise constraint $|\mm|=1$ and the
inherent dissipative energy structure. 
These difficulties become especially relevant when simulating
metastable magnetic configurations or current-driven dynamics, where the competition between different effective-field contributions can lead to substantially different magnetization
patterns and dynamical responses.
This motivates the development of computationally efficient higher-order schemes with rigorous stability and convergence guarantees, rather than relying solely on their observed accuracy
in numerical experiments.

Existing numerical works can roughly be divided
into schemes which converge with rates to sufficiently smooth strong
solutions~\cite{prohl2001,bs2006,cimrak2005,Gao2014,An2016,anGaoSun2021,
feischl,AnSun2022,CaiChenWangXie2023,LiuHuangHe2023,AnLiSun2024,
liliwangyang2026,huang2026lumped}
and schemes which guarantee convergence, along subsequences, towards
weak solutions~\cite{aj2006,bp2006,bkp2008,Alouges_2008,
SpinPolarized,kimwilkening2018}.
The latter approach does not presuppose a smooth solution and usually
relies on discrete energy estimates and compactness arguments.
Higher-order time discretizations include Galerkin and
finite-difference schemes~\cite{YangJiang2021,CaiChenWangXie2022,
CaiChenWangXie2023,xiewang2025bdf3,xiewang2025fd},
Lagrange-multiplier approach~\cite{ChengShen2023}, auxiliary-variable
formulations~\cite{lishenzheng2026}, and implicit--explicit (IMEX)
Runge--Kutta methods~\cite{guiduwang2024}.

Among the approaches to weak convergence, the seminal
works~\cite{bp2006,Alouges_2008} established unconditional weak
convergence for first-order finite elements. While~\cite{bp2006}
employs a midpoint scheme, which preserves the nodal modulus
constraint and is formally second-order in time but requires the
solution of a nonlinear system at every time step, the
work~\cite{Alouges_2008} proposes a tangent plane scheme which is
first-order in time and requires only one linear system per time
step. It exploits $\mm\cdot\partial_t\mm=0$ to formulate a linear
discrete problem for $\vv=\partial_t\mm$ in a nodal tangent space,
followed by a nodal normalization of the updated magnetization.

Being a linear method for a nonlinear equation, the tangent plane
scheme has inspired several generalizations. Its projection-free
variant~\cite{SpinPolarized,bartels2016} omits nodal normalization of
the discrete magnetization at every time step, thereby avoiding the
angle condition for the associated mesh, and recovers the modulus constraint in the limit. For the first-order tangent plane scheme, both weak
convergence~\cite{Alouges_2008} and optimal \textit{a-priori} error estimates
for sufficiently smooth strong solutions~\cite{AnLiSun2024} have
been established. The same results are available for its
projection-free variant, with weak convergence proved
in~\cite{SpinPolarized} and \textit{a-priori} error estimates derived
in~\cite{ft2017}.

For higher-order tangent plane methods, the two convergence questions
have been addressed separately and, according to our knowledge, up to now there were no schemes
for which the full picture is available. In particular,~\cite{akst2014}
proposed a variant of the original algorithm that remains
unconditionally convergent towards weak solutions while being formally almost
second-order in time. Subsequently,~\cite{dpprs2020} formulated and analyzed
an IMEX variant of this scheme, although no corresponding \emph{a-priori} error analysis was provided. Conversely,~\cite{feischl} combines linearly implicit BDF-$k$
methods of orders up to $k=5$ with finite elements of arbitrary
degree and proves optimal-order error estimates for sufficiently
regular solutions. Convergence towards weak solutions, however, is not
addressed and still remains open. Adaptive strategies and \textit{a-posteriori}  
error estimates are derived in~\cite{bdfk2023adaptive,karch2026}.
Related discretizations of the harmonic map heat flow are studied
in~\cite{abp2024,bkw2024,abrw2026,baiguili2025}; see
also~\cite{dong2024bdf} for BDF schemes for accelerated constrained
gradient flows.

We emphasize that much of the existing analysis is restricted to a simplified core model in which the micromagnetic energy contains only the leading-order exchange contribution, possibly together with an external applied field. Physically relevant simulations, however, generally require additional energy contributions.
Besides the leading-order exchange contribution and the external applied field,
the total Gibbs free energy also includes anisotropy and the nonlocal stray field energy. In chiral magnetic materials, it may further contain the bulk Dzyaloshinskii--Moriya interaction
(DMI)~\cite{dzyaloshinskii1958,moriya1960}, which favors twisted magnetization textures.
More generally, the interplay between exchange, DMI, anisotropy,
and magnetostatic interactions is central to the stabilization
and dynamics of chiral magnetic textures~\cite{bogdan1,bogdan2}. We refer
to~\cite{davoli2020} for the derivation of thin-film models with
DMI, and to Section~\ref{sec: physical background} for a more
detailed discussion of the physical energy contributions.

For an efficient numerical treatment, IMEX method treat the
exchange contribution implicitly while evaluating lower-order field
contributions explicitly. Such strategies have been analyzed for tangent-plane
schemes~\cite{akt2012,bsffgppr2014,SpinPolarized,hpprss2019,dpprs2020},
midpoint-type schemes~\cite{mrs2018,dppr2023}, and predictor-corrector
methods~\cite{mauser2022}. These methods are either first-order or only
formally almost second-order in time, require nonlinear solves, or do not simultaneously accommodate the general effective field and torque contributions considered here. We note that DMI has been incorporated in~\cite{hpprss2019,dppr2023}, while an \emph{a-priori} error analysis for a mass-lumped midpoint scheme with DMI was established in~\cite{soenjaya2026dmi}.

Related extensions include coupling with Maxwell's
equations~\cite{bpp2015}, eddy currents~\cite{ft2017,dpprs2020},
elastodynamics through magnetoelastic
effects~\cite{bs2006,bppr2014,normington2025,kruzik2026}, and
spin-polarized transport~\cite{SpinPolarized,Abertetal2015,Abertetal2016};
see also~\cite{ruggeri2016coupling}. Finally, extensions to inertial dynamics LLG equation are analyzed in~\cite{ruggeri2022}.
\\

\textbf{Contributions of the present work.}
The present paper is the natural continuation of \cite{part1} and, together with \cite{part2}, 
constitutes the third part of a series of papers. This series aims to extend the tangent-plane approach to a BDF$2$ discretization in time and to prove weak convergence
to weak solutions as well as optimal-order error estimates if a sufficiently smooth strong solution exists.
In particular,~\cite{part1} develops an integrator which is a synthesis of \cite{Alouges_2008,abp2024,feischl}. Indeed, it employs first-order 
finite elements for the spatial discretization, 
and combines a nodal discrete tangent space as 
in~\cite{bp2006,Alouges_2008}, a projection-free 
update as in~\cite{SpinPolarized,bartels2016,abp2024}, 
a predictor-corrector structure with an
extrapolated term 
as in~\cite{feischl}, and the projection-free constructions 
for harmonic maps heat flow as in~\cite{bkw2024,abrw2026}.
Each time step requires the solution of only one linear system.

In the first part~\cite{part1}, restricted to a reduced effective
field consisting of the exchange and Zeeman contributions only, we
proved unconditional stability and, for arbitrary sequences
$(h,\tau)\to(0,0)$, unconditional weak convergence of a subsequence of the discrete
space-time magnetization in $\bH^1(\OT)$ towards a function satisfying
the modulus constraint and the variational formulation of LLG.
However, the energy inequality additionally required that certain
correction terms at the initial and terminal time levels vanished. 
This was, for instance,
ensured by the CFL-type condition $\tau=o(h^2)$, which entered only
the estimates of these endpoint terms. The numerical experiments
suggested that this restriction was an artifact of the proof, but
its removal in the analysis remained open.

For the same reduced effective field, in the second part~\cite{part2}
we proved optimal-order \textit{a-priori} error estimates of the form
$\max_n \norm{\mm(t_n)-\mm_h^n}_{\bH^1(\Omega)}\lesssim h+\tau^2$
for sufficiently smooth strong solutions, under the mild coupling
condition $\tau\lesssim h^{(1+\varepsilon)/4}$ for some
$0<\varepsilon<1$. The analysis is inspired by~\cite{feischl}, but uses
a nodal rather than a variational orthogonality constraint to define
the discrete tangent space. As observed in~\cite{part2}, the two
works together provide, to the best of our knowledge, the first
linear second-order-in-time integrator for LLG that converges both
towards weak solutions and, with optimal rates, towards sufficiently
smooth strong solutions. Both papers restrict to a reduced effective field consisting of exchange and an applied external field only.

In the present paper, we extend the weak convergence analysis 
of~\cite{part1} to a more physically relevant full effective field 
which includes, besides the leading-order exchange contribution and the external
applied field, also other contributions such as nonlinear anisotropy of arbitrary polynomial degree, the stray field, and bulk DMI. At the same time, we generalize 
the equation to also account for current-driven dynamics, introducing
a non-conservative torque
$-\mm\times\calQ(\mm)$, covering for example Zhang--Li and Slonczewski-type
contributions~\cite{zhangli2004,slonczewski1996}. The BDF$2$ scheme treats the exchange contribution of the effective field
implicitly and the other magnetization-dependent terms explicitly by second-order
extrapolation (except for the initial time-step where DMI is also treated implicitly).
Thus, the method still requires only the solution of one linear system per time step.

Most importantly, we strengthen the convergence result and analysis of~\cite{part1} to a fully unconditional one by removing the space-time coupling $\tau=o(h^2)$ imposed on the first and last time steps in the proof of the discrete energy inequality. This confirms the numerical observation of~\cite{part1} that the condition was merely technical. 
In particular, our analysis provides an unconditionally weakly convergent, projection-free BDF$2$ tangent-plane scheme that is formally second-order in time, even in the presence of DMI. This improves upon the convergence results of~\cite{hpprss2019}, where the formally second-order scheme requires a CFL-type coupling between $h$ and $\tau$ and a mesh angle condition.
Neither restriction is required for our scheme.

Moreover, our framework accommodates nonlinear lower-order effective-field contributions, including magnetocrystalline anisotropies of arbitrary polynomial degree, while retaining the physical energy inequality in the limit.
This extends the analysis of~\cite{bsffgppr2014}, which already accommodates nonlinear anisotropies but establishes the physical energy inequality under the additional assumption that the lower-order operator is linear and self-adjoint. More precisely, for arbitrary sequences $(h,\tau)\to(0,0)$ without any space--time coupling, our main result establishes unconditional weak subsequential convergence towards a function satisfying the modulus constraint and the variational formulation of LLG. Under additional strong consistency assumption \eqref{enum: punto B5} on the discrete torque, the limit
also satisfies the physical energy inequality.

We stress that the removal of the CFL condition is not a consequence of the additional energy contributions, 
but of a new argument in the proof of the energy inequality: 
Instead of the discrete physical exchange energy $\frac12 \normO{\nabla\mm_h^n}^2$ 
with an endpoint correction term $\eta_n$, we use the quadratic energy associated with
the $G$-stability identity~\cite{Dahlquist,
nevanlinna_odeh,baiocchi_crouzeix,hairer}
\[
	\mathcal{G}_h^n
	=\frac14\Big(
	\normO{\nabla\mm_h^n}^2
	+\normO{2\nabla\mm_h^n-\nabla\mm_h^{n-1}}^2
	\Big)
\]
for BDF2
and pass directly to its integrated lower limit by weak lower semicontinuity.
This avoids the separate control of the terminal exchange correction $\eta_n$ as in \cite{part1},
and the spatial inverse estimate used for that purpose. Notice also that, for the reduced effective field with only exchange contribution and external field ($\lDM=0$, $\bpi=\boldsymbol{0}$, and $\calQ=\boldsymbol{0}$), the integrator exactly coincides with that of~\cite{part1}, and our proof applies without changes; see Remark~\ref{rem: part1 unconditional}. 
Thus, the present analysis also yields an unconditional energy inequality
in the setting of~\cite{part1}, and consequently upgrades the convergence
result there to one without any coupling between $h$ and $\tau$.

The corresponding \textit{a-priori} analysis for the
full effective field will be addressed in a subsequent work~\cite{part4}, 
complementary to this one, creating a complete picture of a second-order 
in time unconditional and linear scheme for the LLG equation with physically relevant effective field and non-conservative torque.
\\

\textbf{Outline.}
The outline of the present paper reads as follows: after recalling some general notation,
Section~\ref{sec:energy cont} introduces the main contributions
to the micromagnetic energy and their physical interpretation. Section~\ref{section:torque examples} discusses relevant
examples of admissible torque contributions.
Section~\ref{section:model} presents the mathematical model, including
the assumptions on the effective field contributions and the formulation
of the LLG equation with general nonconservative torques.
Sections~\ref{section:discretization}-\ref{section:integrator}
introduce the spatial discretization and the fully discrete BDF2
integrator in Algorithm~\ref{alg: full discr}. 
Finally, Section~\ref{section:theorem} states the main convergence result.
The proof of this main theorem is given in the following sections:
In Section~\ref{section:energy identity}, we establish the discrete
stability estimates and energy identity, while the compactness argument
and the identification of the weak limit are carried out in
Sections~\ref{section:extraction} and~\ref{section:identification},
respectively. Finally, Section~\ref{section:numerics} presents the numerical
experiments that support our theoretical findings.
\\

\textbf{General notation.}
%
Throughout this paper, we adhere to the standard notation for Lebesgue, Sobolev, and Bochner spaces, using bold letters to indicate spaces of vector valued functions, e.g., $\bL^2(\Omega) \coloneqq L^2(\Omega,\R^3)$ and $\bH^1(\Omega)\coloneqq H^1(\Omega,\R^3)$.
Given $\uu, \vv \in \bL^2(\Omega)$, we will denote the $\bL^2$-scalar product in $\Omega\subset\R^d$, $d=2,3,$ by 
\begin{align*}
 \langle\uu,\vv\rangle_\Omega 
 \coloneqq 
 \langle\uu,\vv\rangle _{\bL^2(\Omega)}= \int_\Omega\uu(\xx)\cdot\vv(\xx)\,\d\xx,
\end{align*}
and we will use $\normO{\,\cdot\,}$ to indicate the induced norm, i.e., $\normO{\uu} \coloneqq \langle\uu,\uu\rangle_\Omega^{1/2}$.
Similar notation is applied for the space-time domain $\OT\coloneqq \Omega \times [0,T]$ and the corresponding scalar product $\langle\cdot,\cdot\rangle_{\OT}$ and norm $\normOT{\,\cdot\,}$.
All other norms will be explicitly stated.

By $C > 0$, we will denote a generic constant, which is always independent of the discretization parameters, but not necessarily the same at each occurrence.
We will use the notation $A\lesssim B$ to indicate that $A\leq C \, B$, where the hidden constant $C > 0$ is clear from the context.
Moreover, $A \simeq B$ abbreviates $A \lesssim B \lesssim A$.

\section{Physical background and mathematical model}\label{sec: physical background}
\subsection{Energy contributions in micromagnetism}\label{sec:energy cont}
In micromagnetism, a key quantity in the study of magnetic processes is the so-called total magnetic Gibbs free energy of a ferromagnetic body, which is a functional of the magnetization $\mm$ and a possible applied external field $\ff$. Since our analysis refers to a non-dimensionalized version of the model, we will present the rescaled reduced version of the energy functionals and refer to, e.g., \cite{ruggeri2016coupling} for a detailed presentation involving all physical parameters. In the present setting the total energy comprises the following contributions (a)-(e), which will all be covered by our current analysis, while \cite{part1,part2} restrict only to (a) and (b):
\begin{itemize}
	\item[\rm (a)] The \textbf{Heisenberg exchange contribution} penalizes non-uniformities of the magnetization, with $\lamex>0$ being the so-called \textit{exchange length}:
	\begin{align*}
	\EE_{\rm ex}(\mm) = \frac{\lamex^2}{2}\int_\Omega |\nabla\mm|^2\,\d\xx;
	\end{align*}
	\item[\rm (b)] The \textbf{Zeeman external contribution} penalizes any deviation of the magnetization from a given applied external magnetic field $\ff$, i.e.,
	\begin{align*}
	\EE_{\rm ext}(\mm,\ff) = -\int_\Omega\ff\cdot\mm\d\xx;
	\end{align*}
	\item[\rm (c)] The \textbf{magnetostatic contribution} is defined as 
	\begin{align*}
	\EE_{\rm stray}(\mm) = -\frac{1}{2}\int_\Omega \boldsymbol{h}_\text{s}(\mm)\cdot\mm\,\d\xx = \frac{1}{2}\int_{\R^3}|\boldsymbol{h}_\text{s}(\mm)|^2\d\xx,
	\end{align*}
	where the stray field $\boldsymbol{h}_\text{s}(\mm)$ solves the following rescaled Maxwell equations:
	\begin{subequations}
		\begin{empheq}[left=\empheqlbrace]{align}
			\nabla \cdot \boldsymbol{h}_\text{s}(\mm) & = -\nabla \cdot (\chi_\Omega \mm) & & \text { in } \mathbb{R}^3, \\
			\nabla \times \boldsymbol{h}_\text{s}(\mm) & = 0 & & \text { in } \mathbb{R}^3.	
		\end{empheq}
	\end{subequations}
	Since $\boldsymbol{h}_\text{s}(\mm)$ is an irrotational field in a simply connected domain, there exists a scalar magnetostatic potential $u$ such that $\boldsymbol{h}_\text{s}(\mm) = -\nabla u$. We denote its restrictions to $\Omega$ and $\R^3\setminus\overline{\Omega}$ by
	$u^{\rm int}\coloneqq u|_{\Omega}$ and $u^{\rm ext}\coloneqq u|_{\R^3\setminus\overline{\Omega}}$, respectively. The magnetostatic potential $u$ is then characterized by the full-space transmission problem
	\begin{subequations}
		\begin{empheq}[left=\empheqlbrace]{align} 	-\Delta u^{\mathrm{int}} & =-\nabla \cdot \boldsymbol{m} & & \text { in } \Omega, \\ 
			-\Delta u^{\mathrm{ext}} & =0 & & \text { in } \mathbb{R}^3 \backslash \overline{\Omega}, \\ 
			u^{\mathrm{ext}}-u^{\mathrm{int}} & =0 & & \text { on } \Gamma, \\ 
			\left(\nabla u^{\mathrm{ext}}-\nabla u^{\mathrm{int}}\right) \cdot \boldsymbol{n} & =-\boldsymbol{m} \cdot \boldsymbol{n} & & \text { on } \Gamma, \\ 
			u(\xx) & =\mathcal{O}(1 /|\xx|) & & \text { as }|\xx| \rightarrow \infty .
		\end{empheq}
	\end{subequations}
	The mapping $\mm\mapsto\boldsymbol{h}_{\rm s}(\mm)|_\Omega$
	is linear, bounded, and self-adjoint on $\bL^2(\Omega)$; see~\cite{Pra04}.
	This contribution represents the magnetostatic self-energy associated with the interaction of the magnetization with its own stray (or demagnetizing) field. In the presence of an electric current, an additional magnetic field may arise through the \emph{Oersted field}. In the present model, current-induced effects are incorporated through the non-conservative torque $\calQ$, while the Oersted field is neglected. If prescribed independently, such a field could alternatively be included as part of the applied field $\ff$.
	\item[\rm (d)] The \textbf{magnetocrystalline anisotropy contribution} penalizes deviations of the magnetization from a set of given easy axes. The anisotropy energy is given by
	\begin{align*}
	\EE_{\rm ani}(\mm) = \int_\Omega\psi(\mm)\d\xx.
	\end{align*}
	In case of the uniaxial anisotropy with a given easy axis $\boldsymbol{a}\in \S^2$, we have $\psi(\mm) =  \frac{q^2}{2} \bigl( 1-(\boldsymbol{a}\cdot \mm)^2 \bigr)$, where $q>0$ is a constant. Other types of anisotropy can be considered. For planar anisotropy, $\psi(\mm) =  \frac{q^2}{2} (\boldsymbol{a}\cdot \mm)^2$. Some materials exhibit cubic anisotropy with three mutually orthogonal easy directions $\boldsymbol{a}_j\in \S^2$, $j=1,2,3$. In that case,
	\[
		\psi(\mm)= q_1(m_1^2 m_2^2 + m_2^2 m_3^2 + m_3^2 m_1^2) + q_2 m_1^2 m_2^2 m_3^2,
	\]
	where $q_1,q_2\in \R$ and $m_j=\boldsymbol{a}_j\cdot \mm$, $j=1,2,3$;
	see e.g. \cite{Abert19,ruggeri2016coupling} for more details.
	Higher-order anisotropy contributions have also been considered in the literature~\cite{Collings_etal23}.
	\item[\rm (e)] The \textbf{Dzyaloshinskii-Moriya interaction (DMI)} or \textbf{antisymmetric exchange energy} tries to twist the magnetization, in contrast with the exchange energy, which tries to align it. The DMI is the fundamental ingredient in the formation of the so-called magnetic skyrmions, which are topologically protected vortex-like magnetic structures. The DMI energy is defined as a linear combination of the coefficients of the \textit{chirality tensor} $\boldsymbol{\nabla}\mm\times\mm\in\R^{3\times3}$ (see \cite{bogdan1,bogdan2}), i.e.,
	\begin{align*}
	\EE_{\rm DMI}(\mm) = \int_\Omega \boldsymbol{D}\colon (\boldsymbol{\nabla}\mm\times\mm)\,\d\xx = \sum_{j=1}^3\int_\Omega \boldsymbol{d}_j\cdot(\partial_j\mm\times\mm)\,\d\xx,
	\end{align*}
	where the tensor $\boldsymbol{D}$ depends on the crystallographic class of the material.
	For simplicity, our analysis restricts to the case of \textit{bulk} DMI, which in the rescaled case takes the form
	\begin{align*}
	\EE_{\rm DMI}(\mm) = \frac{\lDM}{2}\int_\Omega (\nabla \times \mm)\cdot \mm\,\d\xx,
	\end{align*}
	where $\lDM>0$ is the DMI constant. To see other possible expressions of the tensor $\boldsymbol{D}$, we refer to \cite{ruggeri2016coupling,hpprss2019} and the references therein.
\end{itemize}

\subsection{Physical examples of non-conservative torque $\calQ$}\label{section:torque examples}

We briefly record two important classes of torque contributions.
	
	First, we consider the Zhang--Li torque, \cite{zhangli2004} which satisfies~\eqref{enum: punto B4}, while~\eqref{enum: punto B5} can fail. Given the current density $\nuu\in C([0,T];L^\infty(\Omega;\R^d))$, the Zhang--Li torque reads
	\[
	\calQ_{\rm ZL}(\pphi)
	\coloneqq
	\pphi\times(\nuu\cdot\nabla)\pphi
	+
	\xi(\nuu\cdot\nabla)\pphi,\quad \text{for }|\pphi|=1 \quad \text{ a.e. in } \,\Omega,
	\]
	Since the discrete extrapolation
	$\widehat{\mm}_h^{j+1}$ does not satisfy an exact pointwise modulus
	constraint, we introduce the pointwise cutoff
	\begin{equation}\label{eq: cutoff R}
		\bR\colon\R^3\to\overline{\boldsymbol B_1(\boldsymbol 0)},
		\qquad
		\bR(\rr)
		\coloneqq
		\frac{\rr}{\max\{1,|\rr|\}}.
	\end{equation}
	The mapping $\bR$ is the metric projection onto the closed unit ball and
	therefore satisfies
	\begin{equation}\label{eq: cutoff properties}
		|\bR(\rr)|\le1,
		\qquad
		\bR(\rr)=\rr
		\quad\text{if }|\rr|\le1,
		\qquad
		|\bR(\rr)-\bR(\sss)|
		\le|\rr-\sss|
	\end{equation}
	for all $\rr,\sss\in\R^3$. For a vector field, $\bR$ is understood
	pointwise.
	
	We define
	$\nuu^{j+1}\coloneqq\nuu(t_{j+1})$ and
	\begin{equation}\label{eq: discrete Zhang Li}
		\calQ_h^{j+1}(\pphi_h)
		\coloneqq
		\bR(\pphi_h)\times
		(\nuu^{j+1}\cdot\nabla)\pphi_h
		+
		\xi(\nuu^{j+1}\cdot\nabla)\pphi_h.
	\end{equation}
	Then,~\eqref{eq: discrete torque bound} follows from
	\begin{equation}\label{eq: discrete Zhang Li bound}
		\normO{\calQ_h^{j+1}(\pphi_h)}
		\le
		(1+|\xi|)
		\norm{\nuu^{j+1}}_{L^\infty(\Omega)}
		\normO{\nabla\pphi_h}.
	\end{equation}
	Moreover,
	$\nuu_\tau^+\to\nuu$ in
	$L^\infty(0,T;L^\infty(\Omega;\R^d))$, and the nonexpansiveness of
	$\bR$ shows that the weak consistency
	\eqref{eq: weak torque consistency} is satisfied. Since
	$\bR(\mm)=\mm$ whenever $|\mm|=1$, the cutoff does not modify the
	continuous Zhang--Li torque. Hence, \eqref{enum: punto B4} is guaranteed.
	In general, however,
	\eqref{eq: strong torque consistency} does not follow from the compactness
	available below because the Zhang--Li torque depends on $\nabla\mm$, i.e., \eqref{enum: punto B5} can fail.
	
	Second, Slonczewski-type torques \cite{slonczewski1996} satisfy
	assumptions~\eqref{enum: punto B4}--\eqref{enum: punto B5}. Let
	$\pp\in\bL^\infty(\Omega)$, with $|\pp|\le1$ a.e.\ in $\Omega$, be a
	prescribed polarization field, and let
	$A,B\colon\R\to\R$ be bounded and Lipschitz continuous. For $|\pphi|=1$,
	consider
	\begin{equation}\label{eq: Slonczewski torque Q}
		\calQ_{\rm SL}(\pphi)
		\coloneqq
		A(\pphi\cdot\pp)\,\pphi\times\pp
		+
		B(\pphi\cdot\pp)\,\pp,
	\end{equation}
	so that $-\pphi\times\calQ_{\rm SL}(\pphi)$ contains the usual
	damping-like and field-like torque contributions. As above, we extend
	this definition away from the unit sphere by replacing $\pphi$ with
	$\bR(\pphi)$ and define
	\begin{equation}\label{eq: discrete Slonczewski}
		\calQ_h^{j+1}(\pphi_h)
		\coloneqq
		A\bigl(\bR(\pphi_h)\cdot\pp\bigr)\,
		\bR(\pphi_h)\times\pp
		+
		B\bigl(\bR(\pphi_h)\cdot\pp\bigr)\,\pp.
	\end{equation}
	Since $|\bR(\pphi_h)|\le1$ and $|\pp|\le1$, the bound
	\eqref{eq: discrete torque bound} follows immediately. Moreover, the
	nonexpansiveness of $\bR$ and the Lipschitz continuity of $A$ and $B$
	imply the strong consistency~\eqref{eq: strong torque consistency},
	and hence also~\eqref{eq: weak torque consistency}. Finally,
	$\bR(\mm)=\mm$ whenever $|\mm|=1$, so that the cutoff does not modify
	the continuous Slonczewski-type torque. This proves \eqref{enum: punto B3}--\eqref{enum: punto B4}.


\subsection{Mathematical model}\label{section:model}
Since this work extends the results of \cite{part1, part2}, the results and proofs here presented can be regarded as a complementary extension of their counterparts in these two papers, providing a more complete description of the model. In particular, the cases without DMI or without non-conservative torque are recovered by setting $\lDM=0$ or $\calQ=0$, respectively. As mentioned earlier, however, we also give an improved analysis which avoids any CFL-type coupling in space-time.

We consider a ferromagnetic body occupying a bounded Lipschitz domain $\Omega \subset \R^d$, $d=2,3$, with boundary $\Gamma \coloneqq \partial\Omega$. We recall that the two-dimensional setting corresponds to so-called thin-film limits, which describe micromagnetic phenomena in samples where one dimension (the thickness) is much smaller than the other two; see, e.g., \cite{davoli2020} for a detailed derivation of the thin-film model with $d = 2$ from the general setting with $d=3$. Moreover, the physical stray-field problem described in Section~\ref{sec: physical background} corresponds to $d=3$. For $d=2$, the abstract analysis below applies to the corresponding reduced thin-film lower-order operator, provided it satisfies \eqref{enum: punto A1}. When the curl operator is used for $d=2$, vector fields are understood as three-dimensional fields that are independent of the third spatial variable, i.e., 
\[
\operatorname{curl}\mm\coloneqq
\begin{pmatrix}
\partial_2m_3\\
-\partial_1m_3\\
\partial_1m_2-\partial_2m_1
\end{pmatrix}\qquad \text{for}\quad d=2.
\]

The dynamics of the normalized magnetization $\mm\in\S^2 \coloneqq \{\xx \in \R^3 \colon |\xx| = 1\}$ is governed by the Landau--Lifshitz--Gilbert (LLG) equation. In the present setting, the total micromagnetic energy reads
\begin{equation}\label{eq: energy}
	\begin{split}
		\EE(\mm,\ff)
		&\coloneqq
		\frac{\lamex^2}{2}\int_\Omega |\nabla\mm|^2\,\d\xx
		+\int_\Omega\psi(\mm)\,\d\xx
		-\frac{1}{2}\int_\Omega \boldsymbol{h}_{\rm s}(\mm)\cdot\mm\,\d\xx
		-\int_\Omega\ff\cdot\mm\,\d\xx
		\\
		&\qquad
		+\frac{\lDM}{2}\int_\Omega(\nabla\times\mm)\cdot\mm\,\d\xx.
	\end{split}
\end{equation}
The driving force associated with the micromagnetic energy is the so-called effective field $\heff$, which is defined as the negative variational derivative of the energy with respect to the magnetization. Its bulk contribution reads
\begin{equation}\label{eq: heff}
	\begin{split}
	\heff(\mm,\ff)
	\coloneq
	-\frac{\delta\EE(\mm,\ff)}{\delta\mm}
	&=
	\lamex^2\Delta\mm
	-\lDM\nabla\times\mm
	-\nabla_{\mm}\psi(\mm)
	+\boldsymbol{h}_{\rm s}(\mm)
	+\ff
	\\
	&\eqcolon
	\lamex^2\Delta\mm
	-\lDM\nabla\times\mm
	+\bpi(\mm)
	+\ff,
\end{split}
\end{equation}
where $\bpi(\mm)= -\nabla_{\mm}\psi(\mm)+\boldsymbol{h}_{\rm s}(\mm)$ denotes the $\mm$-dependent lower-order contributions due to the anisotropy and stray field.

In addition to the conservative contributions generated by \eqref{eq: energy}--\eqref{eq: heff}, we allow for a general non-conservative torque $-\mm\times\calQ(\mm)$. The operator $\calQ$ may depend on prescribed space-time quantities, such as a spin current or a polarization field. Whenever this explicit time dependence needs to be displayed, we write $\calQ(t;\pphi)$; otherwise, we retain the shorter notation $\calQ(\pphi)$.

We assume that:
\begin{enumerate}[label=(A\arabic*), ref=A\arabic*]
	\item\label{enum: punto A1}
	$\bpi\colon\bL^2(\Omega)\to\bL^2(\Omega)$ satisfies $\bpi(\boldsymbol{0})=\boldsymbol{0}$ and is globally
	Lipschitz continuous: there exists $C_{\bpi}>0$ such that
	\[
	\normO{\bpi(\pphi)-\bpi(\bpsi)}
	\le
	C_{\bpi}\normO{\pphi-\bpsi}
	\qquad
	\text{for all }\pphi,\bpsi\in\bL^2(\Omega).
	\]
	Moreover, there exists a functional $\EE_{\rm low}\colon\bL^2(\Omega)\to\R$ that is Lipschitz continuous on bounded subsets of $\bL^2(\Omega)$, satisfying the chain rule
	\[
	\frac{\d}{\d t}\EE_{\rm low}(\mm(t))
	=
	-\scalarproductO{\bpi(\mm(t))}{\partial_t\mm(t)}
	\]
	for a.e.\ $t\in(0,T)$ and every
	$\mm\in H^1(0,T;\bL^2(\Omega))$ with
	$|\mm|=1$ a.e.\ in $\OT$.
	\item\label{enum: punto A2}
	For almost every $t\in(0,T)$, the possibly time-dependent torque
		operator is a mapping $\calQ(t;\cdot) \colon \bH^1(\Omega) \to \bL^2(\Omega)$.
		For $\pphi\in L^2\bigl(0,T;\bH^1(\Omega)\bigr)$, we define
		\[
		[\calQ(\pphi)](t)
		\coloneqq
		\calQ(t;\pphi(t))
		\qquad
		\text{for a.e. }t\in(0,T).
		\]
		We assume that the mapping
		$t\mapsto[\calQ(\pphi)](t)$ is measurable and that there exists a
		constant $C_{\calQ}>0$, independent of $\pphi$, such that
		\begin{equation}\label{eq: continuous torque bound}
			\norm{
				\calQ(\pphi)
			}_{\bL^2(\OT)}
			\le
			C_{\calQ}
			\bigl(
			1+
			\norm{\pphi}_{L^2(0,T;\bH^1(\Omega))}
			\bigr).
		\end{equation}
\end{enumerate}

	In particular, nonlinear magnetocrystalline anisotropies of arbitrary
	polynomial degree, including the cubic anisotropy, can be incorporated
	under assumption~\eqref{enum: punto A1} by using the metric projection $\bR$ defined in~\eqref{eq: cutoff R}.
	To retain the original treatment of uniaxial anisotropy contribution,
	we write $\psi=\psi_{\rm lin}+\psi_{\rm nl}$, where $\psi_{\rm lin}$ is a quadratic polynomial with no linear term, 
	$\psi_{\rm nl}\in C^2(\R^3)$, and both satisfy $\nabla\psi_{\rm lin}(\boldsymbol{0})=\nabla\psi_{\rm nl}(\boldsymbol{0})=\boldsymbol{0}$.
	Notice that, if $\psi_{\rm lin}$ contains also a linear term, then the corresponding
	contribution in the gradient $\nabla\psi_{\rm lin}$ can be absorbed into the external field $\ff$.
	We extend the lower-order field to arbitrary arguments by setting
	\begin{equation}\label{eq:phi lin nl}
	\bpi(\pphi)
	=
	-\nabla_{\pphi} \psi_{\rm lin}(\pphi)
	-(\nabla_{\pphi} \psi_{\rm nl})\bigl(\bR(\pphi)\bigr)
	+
	\boldsymbol{h}_{\rm s}(\pphi),
	\end{equation}
	and define the corresponding extension of the physical lower-order energy by
	\[
	\EE_{\rm low}(\pphi)
	=
	\int_\Omega
	\bigl[
	\psi_{\rm lin}(\pphi)
	+
	\psi_{\rm nl}\bigl(\bR(\pphi)\bigr)
	\bigr]\d\xx
	-
	\frac12
	\scalarproductO{\boldsymbol{h}_{\rm s}(\pphi)}{\pphi}.
	\]
	Accordingly, the total micromagnetic energy can be written as
	\begin{equation}\label{eq: energy 2}
		\begin{split}
			\EE(\mm,\ff)
			&=
			\frac{\lamex^2}{2}
			\normO{\nabla\mm}^2
			+
			\EE_{\rm low}(\mm)
			-
			\scalarproductO{\ff}{\mm}
			+
			\frac{\lDM}{2}
			\scalarproductO{\nabla\times\mm}{\mm}.
		\end{split}
	\end{equation}
	Since $\bR(\mm)=\mm$ whenever $|\mm|=1$, both the lower-order
	field and energy coincide with their original expressions \eqref{eq: energy}--\eqref{eq: heff} 
	for unit-length magnetizations. Thus, the extension does not modify 
	the continuous LLG problem or its micromagnetic energy.
	Moreover, the linear anisotropy contributions remain unchanged
	also at the discrete level.
	
	Notice that, since $\psi_{\rm lin}$ is quadratic and satisfies
	$\nabla\psi_{\rm lin}(\boldsymbol{0})=\boldsymbol{0}$,
	its gradient is linear. Moreover, since
	$\psi_{\rm nl}\in C^2(\R^3)$, its first and second derivatives
	are bounded on the compact closed unit ball. Hence, there
	exists $M_2>0$ such that
	\[
	\bigl|
	(\nabla\psi_{\rm nl})(\bR(\pphi_1))
	-
	(\nabla\psi_{\rm nl})(\bR(\pphi_2))
	\bigr|
	\le M_2|\bR(\pphi_1)-\bR(\pphi_2)|
	\le M_2|\pphi_1-\pphi_2|,
	\]
	where we have used the nonexpansiveness of $\bR$ for the last inequality. 
	Thus, the operator $\bpi$
	is globally Lipschitz continuous and the functional
	$\EE_{\rm low}$ is Lipschitz continuous on bounded subsets
	of $\bL^2(\Omega)$.
	Furthermore, for any $\mm\in H^1(0,T;\bL^2(\Omega))$
	satisfying $|\mm|=1$ a.e.\ in $\OT$, the usual chain rule,
	together with the linearity and self-adjointness of the
	stray-field operator, yields
	\begin{align*}
	\frac{\d}{\d t}\EE_{\rm low}(\mm(t))
	&=
	\int_\Omega
	\big[
		\nabla\psi_{\rm lin}(\mm(t))+
		\nabla\psi_{\rm nl}(\mm(t))
	\big]
	\cdot\partial_t\mm(t)\,\d\xx
	-
	\scalarproductO{\boldsymbol{h}_{\rm s}(\mm(t))}{\partial_t\mm(t)}
	\\
	&=
	-\scalarproductO{\bpi(\mm(t))}{\partial_t\mm(t)}
	\end{align*}
	for almost every $t\in(0,T)$.
	Consequently, assumption~\eqref{enum: punto A1} is satisfied.

Assumption~\eqref{enum: punto A2} is deliberately formulated for a general non-conservative torque. It covers, after a suitable extension away from the unit sphere if necessary, both first-order torque contributions depending on $\nabla\mm$ and zero-order contributions depending only on $\mm$. Concrete examples and their discrete counterparts are discussed above; see Section~\ref{section:torque examples}.

For later use, we collect the two variational contributions arising from the DMI energy in the symmetric bilinear form
\begin{equation}\label{eq: HDM}
	\HDM(\mm,\pphi)
	\coloneqq
	\frac{\lDM}{2}\scalarproductO{\nabla\times\mm}{\pphi}
	+
	\frac{\lDM}{2}\scalarproductO{\mm}{\nabla\times\pphi}
	\qquad
	\text{for all }
	\mm,\pphi\in\bH^1(\Omega).
\end{equation}
The corresponding natural boundary condition is the chiral boundary condition
\[
\lamex^2\partial_{\vec n}\mm
+
\frac{\lDM}{2}\mm\times\vec n
=0
\qquad\text{on }\Gamma.
\]

	Unlike the exchange, anisotropy, magnetostatic, Zeeman, and DMI
	contributions, the torque $\calQ$ is not assumed to arise from the
	micromagnetic energy~\eqref{eq: energy}. It is therefore treated as a
	general non-conservative forcing term.

Given a final time $T>0$, we seek $\mm\colon\Omega\times[0,T]\to\S^2$ satisfying
\begin{subequations}\label{eq: LLG system}
	\begin{empheq}[left=\empheqlbrace]{align}
		&\partial_t\mm
		=
		-\mm\times\heff
		+\alpha\mm\times\partial_t\mm
			-\mm\times\calQ(\mm)
		&&\hspace{-2em}\text{in }\Omega\times[0,T],
		\label{eq: LLG}
		\\
		&2\lamex^2\partial_{\vec n}\mm
		+ \lDM \mm\times\vec n
		=0
		&&\hspace{-2em}\text{on }\Gamma\times[0,T],
		\label{eq:subeq_b}
		\\
		&\mm(0)=\mm^0
		&&\hspace{-2em}\text{in }\Omega.
		\label{eq:subeq_c}
	\end{empheq}
\end{subequations}
Here, $\mm^0\colon\Omega\to\S^2$ is the given initial state, equation~\eqref{eq: LLG} is the Gilbert form of LLG, and $\alpha>0$ is the non-dimensional Gilbert damping parameter.

Taking the scalar product of~\eqref{eq: LLG} with $\mm$, we obtain the orthogonality property
\begin{equation}\label{eq: orthogonality property}
	\frac12\partial_t|\mm|^2
	=
	\mm\cdot\partial_t\mm
	\stackrel{\eqref{eq: LLG}}{=}0
	\quad\text{in }\OT.
\end{equation}
Indeed, all terms on the right-hand side of~\eqref{eq: LLG} are pointwise orthogonal to $\mm$. With the initial condition $|\mm^0|=1$, the magnetization therefore satisfies the pointwise modulus constraint $|\mm|=1$ in $\OT$.

Furthermore, the time derivative $\partial_t\mm$ is pointwise orthogonal to $\mm$. This built-in orthogonality suggests introducing the tangent space
\begin{equation}\label{eq: tangent space}
	\bT(\mm(t))
	\coloneqq
	\left\{
	\vvarphi\in\bH^1(\Omega)
	\,:\,
	\mm(t)\cdot\vvarphi=0
	\text{ a.e. in }\Omega
	\right\}
	\ni\partial_t\mm(t)
	\quad
	\text{for each }t\in[0,T].
\end{equation}
A formal application of the chain rule to the energy functional~\eqref{eq: energy}, using the chiral boundary condition~\eqref{eq:subeq_b}, yields
	\begin{equation}\label{eq: energy law}
		\frac{\d}{\d t}\EE(\mm,\ff)
		=
		-\alpha\int_\Omega|\partial_t\mm(t)|^2\,\d\xx
		-\int_\Omega\partial_t\ff(t)\cdot\mm(t)\,\d\xx
		+
		\int_\Omega
		\calQ(\mm(t))
		\cdot\partial_t\mm(t)\,\d\xx .
	\end{equation}
	In the absence of the non-conservative torque, i.e.,
	$\calQ=0$, and if additionally $\ff$ is constant in time, the
	micromagnetic energy thus is non-increasing. The last term in
	\eqref{eq: energy law} represents the work performed by the
	non-conservative torque.

\section{Numerical BDF2 integrator and unconditional weak convergence}
\label{section: weak conv no DMI}

\subsection{Discretization}\label{section:discretization}

Given a final time $T>0$ and $N\in\mathbb{N}$, we consider a uniform partition of the time interval $[0,T]$ into $N$ intervals of length $\tau\coloneqq T/N$ with time steps $t_j\coloneqq j\tau$ for all $j=0,\dots,N$. We define respectively the first-order and BDF2 discrete time derivatives for a sequence of functions $\{\mm^j\}_{j=0,\dots,N}$, associated with these time steps, as
\begin{align*}
	\d_t\mm^{j+1}
	\coloneqq
	\dfrac{\mm^{j+1}-\mm^j}{\tau},\qquad
	\Dtau\mm^{j+1}
	\coloneqq
	\frac{
		3\mm^{j+1}-4\mm^{j}+\mm^{j-1}
	}{
		2\tau
	},
\end{align*}
and the second-order discrete time derivative as
\begin{equation}\label{eq: second order discrete time derivative}
	\d_t^2\mm^{j+1}
	\coloneqq
	\dfrac{\d_t\mm^{j+1}-\d_t\mm^j}{\tau}
	=
	\frac{\mm^{j+1}-2\mm^j+\mm^{j-1}}{\tau^2}.
\end{equation}

For the spatial discretization, let $\TT_h$ be a $\gamma$-quasi-uniform family of regular triangular or tetrahedral meshes of the domain $\Omega$ with mesh size $h>0$, i.e., there exists $\gamma>0$ such that
\begin{align*}
	\gamma^{-1}h
	\le |K|^{1/d}
	\le \operatorname{diam}(K)
	\le h
	\quad
	\text{for all }K\in\TT_h
	\text{ and all }h>0,\quad d=2,3.
\end{align*}
Let
\[
\SSS^1(\TT_h)
\coloneqq
\left\{
\vv_h\in C(\overline{\Omega};\R^3)
\,:\,
\vv_h|_K\text{ is affine for all }K\in\TT_h
\right\}
\]
be the space of globally continuous and $\TT_h$-piecewise affine vector fields, and let $\NN_h$ denote the set of vertices of $\TT_h$. We will use the vector-valued nodal interpolation operator
\begin{align}\label{eq: nodal interpolation}
	\III_h
	\colon
	C(\overline{\Omega};\R^3)
	\to
	\SSS^1(\TT_h),
	\qquad
	\III_h\vv
	\coloneqq
	\sum_{\zz\in\NN_h}
	\vv(\zz)\zeta_{\zz},
\end{align}
where $\zeta_{\zz}\in C(\overline{\Omega})$ is the scalar hat function satisfying that $\zeta_{\zz}|_K$ is affine for all $K\in\TT_h$ and
\[
\zeta_{\zz}(\zz')
=
\delta_{\zz,\zz'}
\qquad
\text{for all }\zz,\zz'\in\NN_h.
\]
For all $\vv\in \bH^2(\Omega)$, the nodal interpolant $\III_h\vv\in\SSS^1(\TT_h)$ satisfies
\begin{equation}\label{eq: Lemma 1}
	\normO{\vv-\III_h\vv}
	+
	h\normO{\nabla(\vv-\III_h \vv)}
	\lesssim
	h^2\normO{D^2\vv};
\end{equation}
see, e.g., \cite[Theorem~4.4.4 and Theorem~4.4.20]{brennerscott}.

For later use, we define the nodal weights
\[
\beta_{\zz}
\coloneqq
\int_\Omega \zeta_{\zz}\d x
\qquad
\text{for all }\zz\in\NN_h.
\]
For every $1\le p<\infty$ and every
$\vv_h\in\SSS^1(\TT_h)$, the standard nodal norm equivalence yields
\begin{equation}\label{eq: nodal norm equivalence}
	\norm{\vv_h}_{\bL^p(\Omega)}^p
	\simeq
	\sum_{\zz\in\NN_h}
	\beta_{\zz}
	|\vv_h(\zz)|^p.
\end{equation}
The constants in~\eqref{eq: nodal norm equivalence} depend only on
$p$, the spatial dimension, and the shape regularity of the mesh family.

To mimic the constant-modulus constraint at a discrete level, we define the set of admissible discrete magnetizations, satisfying the modulus constraint at the vertices of the mesh, as
\[
\boldsymbol{\MM}_h
\coloneqq
\set{
	\pphi_h\in\SSS^1(\TT_h)
}{
	|\pphi_h(\zz)|=1
	\text{ for all }\zz\in\NN_h
}.
\]
To enforce the PDE-inherent orthogonality property $\mm\cdot\partial_t\mm=0$ at the discrete level, we impose it at every vertex $\zz\in\NN_h$. To this end, we define the discrete tangent space at $\pphi_h\in\SSS^1(\TT_h)$ as
\begin{equation*}
	\bT_h(\pphi_h)
	\coloneqq
	\left\{
	\bpsi_h\in\SSS^1(\TT_h)
	\,:\,
	\pphi_h(\zz)\cdot\bpsi_h(\zz)=0
	\text{ for all }\zz\in\NN_h
	\right\}.
\end{equation*}

\subsection{Numerical integrator}\label{section:integrator}

We make the following assumptions on the given data of LLG~\eqref{eq: LLG system} and their discretization:
\begin{enumerate}[label=(B\arabic*), ref=B\arabic*]
	\item\label{enum: punto B1} 
	$\mm^0 \in \bH^1(\Omega)$, and $\mm_h^0 \in \boldsymbol{\MM}_h$ satisfies $\norm{\mm^0-\mm_h^0}_{\bH^1(\Omega)}\to 0$ as $h\to 0$.
	\item\label{enum: punto B2}  
	$\ff\in C^1([0,T], \boldsymbol{C}(\overline{\Omega}))$ and the discretization $\ff_{h\tau}^+ \in L^2(0,T, \SSS^1(\TT_h))$ is given by
	\[
	\ff_{h\tau}^+(t,\zz)
	\coloneqq
	\ff_h^{j+1}(\zz)
	\coloneqq
	\ff(t_{j+1},\zz)\quad\text{for all}\quad t\in[t_j,t_{j+1})\text{ and }\zz \in \NN_h,
	\]
	i.e., $\ff_{h\tau}^+$ is piecewise constant in time and the nodal interpolant of $\ff(t_{j+1})$ in space so that, in particular,
	\[
	\normOT{
	\ff_{h\tau}^+-\ff}\to 0\quad \text{as}\quad (h,\tau)\to (0,0).
	\]
	\item\label{enum: punto B3}
	We define the discrete counterpart of $\bpi$ as $\bpi_h\colon\SSS^1(\TT_h)\to\bL^2(\Omega)$,
	and assume that there exists a constant $C_{\bpi}>0$, independent of $h$, such that
	\[
	\normO{\bpi_h(\pphi_h)}
	\le
	C_{\bpi}\normO{\pphi_h}
	\qquad
	\text{for all }\pphi_h\in\SSS^1(\TT_h).
	\]
	Moreover, we assume that $\bpi_h$ is consistent with $\bpi$ in the sense that $\normO{\pphi_h-\pphi}\to 0$ yields
	\[
	\normO{\bpi_h(\pphi_h)-\bpi(\pphi)}\to 0.
	\]
	\item\label{enum: punto B4}
		For every $j=0,\dots,N-1$, let $\calQ_h^{j+1} \colon \SSS^1(\TT_h) \to \bL^2(\Omega)$ denote a discrete counterpart of the frozen-time operator
			$\calQ(t_{j+1};\cdot\,)$. We assume that
		\begin{equation}\label{eq: discrete torque bound}
			\normO{\calQ_h^{j+1}(\pphi_h)}
			\le
			C_{\calQ}
			\left(
			1+\norm{\pphi_h}_{\bH^1(\Omega)}
			\right)
			\qquad
			\text{for all }
			\pphi_h\in\SSS^1(\TT_h),
		\end{equation}
		where $C_{\calQ}$ is independent of $h$, $\tau$, and $j$.
		
		For any piecewise constant function
		$\pphi_{h\tau}\colon[0,T]\to\SSS^1(\TT_h)$, we define
		\[
		\calQ_{h\tau}^+(\pphi_{h\tau})(t)
		\coloneqq
		\calQ_h^{j+1}(\pphi_{h\tau}(t))
		\qquad
		\text{for all }t\in[t_j,t_{j+1}).
		\]
		We further assume the following sequential consistency property:
		if $\pphi_{h\tau}\to\pphi$ in $\bL^2(\OT)$ and $\pphi_{h\tau}\rightharpoonup\pphi$ in $L^2(0,T;\bH^1(\Omega))$, then
		\begin{equation}\label{eq: weak torque consistency}
			\calQ_{h\tau}^+(\pphi_{h\tau})
			\rightharpoonup
			\calQ(\pphi)
			\qquad
			\text{in }\bL^2(\OT).
		\end{equation}
		\item\label{enum: punto B5}
		For the energy inequality only, we additionally assume the strong
		sequential consistency property that, under the convergences in
		\eqref{enum: punto B4},
		\begin{equation}\label{eq: strong torque consistency}
			\calQ_{h\tau}^+(\pphi_{h\tau})
			\to
			\calQ(\pphi)
			\qquad
			\text{in }\bL^2(\OT).
		\end{equation}

\end{enumerate}

Assumptions~\eqref{enum: punto B1}--\eqref{enum: punto B4} will be used for the stability and convergence analysis. To illustrate assumption~\eqref{enum: punto B3}, consider the lower-order field \eqref{eq:phi lin nl} introduced before, and define
\[
\bpi_h(\pphi_h)
=
-\nabla\psi_{\rm lin}(\pphi_h)
-
(\nabla\psi_{\rm nl})\bigl(\bR(\pphi_h)\bigr)
+
\boldsymbol{h}_{{\rm s},h}(\pphi_h),
\]
where $\nabla \psi$ denotes the gradient of the density $\psi$ with respect to its vector argument, while $\boldsymbol{h}_{{\rm s},h}$ is a uniformly bounded and strongly consistent approximation of the stray-field operator. 
This conforms to the structure of $\bpi(\pphi)$ in \eqref{eq:phi lin nl}.
The nonexpansiveness of $\bR$ and the boundedness of $D^2\psi_{\rm nl}$ on the closed unit ball ensure the boundedness and strong consistency of the
nonlinear anisotropy contribution.
The stray-field operator can be approximated by the usual hybrid FEM--BEM methods; see~\cite{bsffgppr2014}.
Consequently, the above choice of $\bpi_h$ satisfies~\eqref{enum: punto B3}.

Assumption~\eqref{enum: punto B5} is not required for the identification of the weak limit and will only be invoked when passing to the non-conservative work term in the
energy inequality.
Furthermore, assumptions~\eqref{enum: punto B1}--\eqref{enum: punto B2} yield the uniform bounds
\[
\max_{t\in[0,T]}\norm{\ff(t)}_{\bL^\infty(\Omega)}
\le C_f,
\qquad
\norm{\mm_h^0}_{\bH^1(\Omega)}
\le C_0\quad\text{for all }h>0.
\]

	Under the constraint $|\mm|=1$, equation~\eqref{eq: LLG} is equivalent to
	\begin{equation}\label{eq: alternative LLG}
		\alpha\partial_t\mm
		+
		\mm\times\partial_t\mm
		=
		\heff
		-
		(\mm\cdot\heff)\mm
		+
		\calQ(\mm)
		-
		\bigl(\mm\cdot\calQ(\mm)\bigr)\mm.
	\end{equation}
	Indeed, taking the cross product of~\eqref{eq: LLG} with $\mm$ and using $\mm\cdot\partial_t\mm=0$ together with the triple cross product identity gives
	\eqref{eq: alternative LLG}. Conversely, taking the cross product of
	\eqref{eq: alternative LLG} with $\mm$ recovers~\eqref{eq: LLG};
	see~\cite{goldenits2012konvergente,ruggeri2016coupling}.
To derive the discrete variational formulation for the numerical integration of~\eqref{eq: LLG system}, we note that the time derivative
$\partial_t\mm(t)\eqcolon\vv(t)\in\bT(\mm(t))$ satisfies, for all test functions $\pphi\in\bT(\mm(t))$,
	\begin{equation}\label{eq: variational form}
		\begin{split}
			\alpha\scalarproductO{\vv}{\pphi}
			+
			\scalarproductO{\mm\times\vv}{\pphi}
			={}&
			-\lamex^2\scalarproductO{\nabla\mm}{\nabla\pphi}
			-\HDM(\mm,\pphi)
			+
			\scalarproductO{
				\bpi(\mm)+\ff+\calQ(\mm)
			}{
				\pphi
			}.
		\end{split}
	\end{equation}
This follows from testing~\eqref{eq: alternative LLG} with $\pphi\in\bT(\mm(t))$, integrating by parts the exchange and DMI contributions, and using the chiral boundary condition~\eqref{eq:subeq_b} together with the definition~\eqref{eq: HDM}. In particular, the normal contribution
	$\bigl(\mm\cdot\calQ(\mm)\bigr)\mm$ vanishes against all
	$\pphi\in\bT(\mm(t))$.

For the numerical integration of~\eqref{eq: LLG system}, the proposed time-marching scheme exploits a finite element discretization of~\eqref{eq: variational form} for the time derivative $\vv$ in the tangent space of $\mm$ and uses second-order backward differences (BDF2) in time. The explicitly treated $\mm$-dependent contributions are evaluated at the second-order extrapolation
	\[
	\widehat{\mm}_h^{j+1}
	\coloneqq
	2\mm_h^j-\mm_h^{j-1}.
	\]
	The same extrapolation is employed in the DMI and non-conservative
	torque contributions. Since
	$\calQ_h^{j+1}(\widehat{\mm}_h^{j+1})$ is evaluated explicitly, a
	possibly nonlinear dependence of $\calQ_h^{j+1}$ on its argument does
	not affect the linearity of the discrete problem. Overall, the resulting scheme
	is an implicit-explicit (IMEX) BDF2-type integrator.
Since BDF2 requires an additional initial value $\mm_h^1\approx\mm(t_1)$ beyond $\mm_h^0\approx\mm^0$, Algorithm~\ref{alg: full discr}\eqref{eq: first step full discr item}--\eqref{item: def mm1} employs one step of the non-projected first-order tangent-plane integrator for LLG~\cite{ruggeri2016coupling}. 
\begin{algorithm}\label{alg: full discr}
	\textbf{Input:} Conforming mesh $\TT_h$ of $\Omega$, $\mm(0)\approx\mm_h^0\in\boldsymbol{\MM}_h$, $T>0$, $N\in\N$, $\tau\coloneqq T/N$, and $t_j\coloneqq j\tau$ for all $j=0,\dots,N$.
	\begin{enumerate}[label=\rm (\roman*), ref=\rm \roman*]
		\item\label{eq: first step full discr item} 
		For $j=0$, compute $\vv_h^0\in\bT_h(\mm_h^0)$ such that, for all $\pphi_h\in\bT_h(\mm_h^0)$,
		\begin{equation}\label{eq: first step full discr}
			\begin{aligned}
				\alpha&\scalarproductO{\vv_h^0}{\pphi_h}
				+
				\scalarproductO{\mm_h^0\times\vv_h^0}{\pphi_h}
				+
				\lamex^2\tau
				\scalarproductO{\nabla\vv_h^0}{\nabla\pphi_h}
				+\tau \HDM(\vv_h^0,\pphi_h)
				\\
				&=
				-\lamex^2
				\scalarproductO{\nabla\mm_h^0}{\nabla\pphi_h}
				-\HDM(\mm_h^0,\pphi_h)
				\\
				&\qquad
				+\scalarproductO{\bpi_h(\mm_h^0)}{\pphi_h}
				+\scalarproductO{\ff_h^1}{\pphi_h}
				+\scalarproductO{\calQ_h^1(\mm_h^0)}{\pphi_h}.
			\end{aligned}
		\end{equation}
		\item\label{item: def mm1}
		Set $\mm_h^1\coloneqq\mm_h^0+\tau\vv_h^0$.
		\item\label{item: next steps}  
		For $j=1,\dots,N-1$, repeat the following steps~\eqref{item: def mhat}--\eqref{item: def mmhj}:
		\begin{enumerate}[label=\rm(\alph*), ref=\rm \alph*]
			\item\label{item: def mhat}
			Set
			\(
			\widehat{\mm}_h^{j+1}
			=
			2\mm_h^j-\mm_h^{j-1}.
			\)
			\item\label{item: var form}
			Compute $\vv_h^j\in\bT_h(\widehat{\mm}_h^{j+1})$ such that, for all $\pphi_h\in\bT_h(\widehat{\mm}_h^{j+1})$,
			\begin{equation}\label{eq: full discr alg B}
				\begin{split}
					\alpha&\scalarproductO{\vv_h^j}{\pphi_h}
					+
					\scalarproductO{\widehat{\mm}_h^{j+1}\times\vv_h^j}{\pphi_h}
					+
					\frac{2}{3}\lamex^2\tau
					\scalarproductO{\nabla\vv_h^j}{\nabla\pphi_h}
					\\
					&=
					-\frac{1}{3}\lamex^2
					\scalarproductO{
						\nabla[4\mm_h^j-\mm_h^{j-1}]
					}{\nabla\pphi_h}
					-\HDM(\widehat{\mm}_h^{j+1},\pphi_h)
					\\
					&\qquad
					+\scalarproductO{\bpi_h(\widehat{\mm}_h^{j+1})}{\pphi_h}
					+\scalarproductO{\ff_h^{j+1}}{\pphi_h}
								+\scalarproductO{
							\calQ_h^{j+1}(\widehat{\mm}_h^{j+1})
						}{\pphi_h}.
				\end{split}
			\end{equation}
			\item\label{item: def mmhj}
			Set
			\(
			\mm_h^{j+1}
			\coloneqq
			\dfrac{4}{3}\mm_h^j
			-\dfrac{1}{3}\mm_h^{j-1}
			+\dfrac{2}{3}\tau\vv_h^j.
			\)
		\end{enumerate}
	\end{enumerate}
	\textbf{Output:} Sequences $\vv_h^j\approx\partial_t\mm(t_{j+1})$ and $\mm_h^{j+1}\approx\mm(t_{j+1})$ for all $j=0,\dots,N-1$.
\end{algorithm}

The continuous bilinear form on the left-hand side of the first-step variational formulation~\eqref{eq: first step full discr} is coercive for all $\tau>0$ satisfying
$2\lDM^2\tau\le\alpha\lamex^2$, while the corresponding bilinear form on the left-hand side of~\eqref{eq: full discr alg B} is continuous and coercive for any $\tau>0$.
Here, the existence and uniqueness of the solution $\vv_h^j$ is guaranteed by the Lax--Milgram lemma. Therefore, Algorithm~\ref{alg: full discr} is well-defined. Notice that $\calQ_h^{j+1}$ may be nonlinear in its argument; this causes no additional difficulty since it is evaluated only at already known discrete magnetizations.

\subsection{Convergence theorem}\label{section:theorem}

	Throughout this work, we consider weak solutions of LLG~\eqref{eq: LLG system}
	in the spirit of \cite[Definition~1.2]{Alouges_Soyeur}, adapted here to include
	the DMI and general non-conservative torque contributions.

\begin{definition}[Weak solution]\label{def: weak solution}
	Let $\mm^0\in\bH^1(\Omega)$ satisfy $|\mm^0|=1$ a.e.\ in $\Omega$.
	A function $\mm\colon\OT\to\R^3$ is a \textnormal{weak solution}
	of~\eqref{eq: LLG system} if the following properties {\rm(i)--(iii)} are
	satisfied:
	\begin{enumerate}[label={\rm(\roman*)}, ref={\rm\roman*}]
		\item\label{def: weak solution i}
		$\mm\in\bH^1(\OT)\cap L^\infty(0,T;\bH^1(\Omega))$ and
		$|\mm|=1$ a.e.\ in $\OT$;
		\item\label{def: weak solution ii}
		$\mm(0)=\mm^0$ in the sense of traces;
		\item\label{def: weak solution iii}
			for all $\vvarphi\in\bH^1(\OT)$, it holds that
		\begin{equation}\label{eq: weak formulation}
			\begin{split}
				&\int_0^T\scalarproductO{\partial_t\mm(t)}{\vvarphi(t)}\d t
				-\alpha\int_0^T
				\scalarproductO{\mm(t)\times\partial_t\mm(t)}{\vvarphi(t)}\d t
				\\
				&\quad
				=
				\lamex^2\int_0^T
				\scalarproductO{\mm(t)\times\nabla\mm(t)}
				{\nabla\vvarphi(t)}\d t
				-
				\int_0^T
				\HDM\big(\mm(t),\mm(t)\times\vvarphi(t)\big)\d t
				\\
				&\qquad
				-
				\int_0^T
				\scalarproductO{\mm(t)\times\bpi(\mm(t))}
				{\vvarphi(t)}\d t
				-
				\int_0^T
				\scalarproductO{\mm(t)\times\ff(t)}
				{\vvarphi(t)}\d t
				\\
				&\qquad
					-
					\int_0^T
					\scalarproductO{\mm(t)\times\calQ(\mm(t))}
					{\vvarphi(t)}\d t .
			\end{split}
		\end{equation}
	\end{enumerate}
\end{definition}

\begin{definition}[Energy inequality]\label{def: energy inequality}
	Let $\mm$ be a weak solution of~\eqref{eq: LLG system} in the sense of
	Definition~\ref{def: weak solution}. We say that $\mm$ satisfies the
	\emph{energy inequality} if, for a.e.\ $t'\in(0,T)$,
	\begin{equation}\label{eq: energy inequality}
		\begin{split}
			\EE(\mm(t'),\ff(t'))
			&+
			\alpha\int_0^{t'}\normO{\partial_t\mm(t)}^2\d t
			+
			\int_0^{t'}
			\scalarproductO{\partial_t\ff(t)}{\mm(t)}\d t
			\\
			&\qquad\qquad\qquad-\int_0^{t'}
			\scalarproductO{\calQ(\mm(t))}{\partial_t\mm(t)}\d t
			\le
			\EE(\mm^0,\ff(0)).
		\end{split}
	\end{equation}
\end{definition}

	The energy inequality in Definition~\ref{def: energy inequality} is the
	weak counterpart of the formal energy law~\eqref{eq: energy law}. In the
	absence of a non-conservative torque, i.e., if $\calQ=0$,
	\eqref{eq: energy inequality} reduces to the standard micromagnetic energy
	inequality.

Using the approximations $\{\mm_h^j\}_{0\le j\le N}$ and
$\{\vv_h^j\}_{0\le j\le N-1}$ obtained by Algorithm~\ref{alg: full discr},
we define
\[
\mm_{h\tau}\in\bH^1(\OT),
\qquad
\mm_{h\tau}^{\pm},
\mhatp,
\vv_{h\tau}^-
\in L^2(0,T;\bH^1(\Omega)) 
\]
as follows. For all $0\le j\le N-1$ and all $t\in[t_j,t_{j+1})$, let
\begin{align}
	\mm_{h\tau}(t)
	&\coloneqq
	\frac{t-t_j}{\tau}\mm_h^{j+1}
	+
	\frac{t_{j+1}-t}{\tau}\mm_h^j,
	\label{eq:mhtau}
	\\
	\mm_{h\tau}^+(t)
	&\coloneqq
	\mm_h^{j+1},
	\label{eq:mhtau-plus}
	\\
	\mm_{h\tau}^-(t)
	&\coloneqq
	\mm_h^j,
	\label{eq:mhtau-minus}
	\\
	\widehat{\mm}_{h\tau}^+(t)
	&\coloneqq
	\begin{cases}
		\mm_h^0,
		&\text{for }j=0,\\
		\widehat{\mm}_h^{j+1},
		&\text{for }1\le j\le N-1,
	\end{cases}
	\label{eq:mhtau-hat-plus}
	\\
	\vv_{h\tau}^-(t)
	&\coloneqq
	\vv_h^j.
	\label{eq:vhtau-minus}
\end{align}
	According to assumption~\eqref{enum: punto B4}, we also use the
	piecewise constant torque reconstruction
	\begin{equation}\label{eq: interpolation Q}
		\bigl[\calQ_{h\tau}^+(\mhatp)\bigr](t)
		\coloneqq
		\calQ_h^{j+1}(\mhatp(t))
		\qquad
		\text{for all }t\in[t_j,t_{j+1}).
	\end{equation}
	In view of~\eqref{eq:mhtau-hat-plus}, this is precisely the torque
	contribution employed by Algorithm~\ref{alg: full discr}, including the
	first time step.
The following theorem extends and improves \cite[Theorem~2.2]{part1} to the full model
considered here.

\begin{theorem}\label{thm: main result}
	Let $\{\TT_h\}_{h>0}$ be a $\gamma$-quasi-uniform family of meshes of
	$\Omega$, and suppose that assumptions
	\eqref{enum: punto A1}--\eqref{enum: punto A2} and
	\eqref{enum: punto B1}--\eqref{enum: punto B4} are satisfied.
	Then, for any sequence $(h,\tau)\to(0,0)$, without any coupling condition
	between $h$ and $\tau$, the discrete functions
	\[
	\mm_{h\tau}^{\star}
	\in
	\big\{
	\mm_{h\tau},
	\mm_{h\tau}^{\pm},
	\mhatp
	\big\}
	\]
	admit subsequences, not relabeled, and there exists a weak solution
	$\mm$ of~\eqref{eq: LLG system} in the sense of
	Definition~\ref{def: weak solution} such that
	\begin{alignat}{3}
		\mm_{h\tau}
		&\rightharpoonup \mm
		&&\;
		\text{ weakly in }\bH^1(\OT),
		\label{eq: weak conv mhtau H1}
		\\
		\mm_{h\tau}
		&\to \mm
		&&\;
		\text{ strongly in }\bL^2(\OT),
		\label{eq: strong conv mhtau L2}
		\\
			\mm_{h\tau}^{\star}
			&\to \mm
			&&\;
			\text{ strongly in }\bL^2(\OT),
			\label{eq: strong conv interpolants L2}
			\\
		\mm_{h\tau}^{\star}
		&\stackrel{*}{\rightharpoonup} \mm
		&&\;
		\text{ in }L^\infty(0,T;\bH^1(\Omega)),
		\label{eq: weak star conv interpolants}
		\\
		\vv_{h\tau}^-
		&\rightharpoonup \partial_t\mm
		&&\;
		\text{ weakly in }\bL^2(\OT),
		\label{eq: weak conv vhtau}
		\\
			\calQ_{h\tau}^+(\mhatp)
			&\rightharpoonup \calQ(\mm)
			&&\;
			\text{ weakly in }\bL^2(\OT),
			\label{eq: weak conv torque}
	\end{alignat}
	where all the limits hold for the same subsequence.
	If, in addition,
		assumption~\eqref{enum: punto B5} is satisfied, then the same subsequence
		\begin{equation}\label{eq: strong convergence torque}
			\calQ_{h\tau}^+(\mhatp)
			\to
			\calQ(\mm)
			\qquad
			\text{strongly in }\bL^2(\OT),
		\end{equation}
		and $\mm$ satisfies the energy inequality of
		Definition~\ref{def: energy inequality}.
\end{theorem}
\begin{corollary}\label{cor: main result}
	Under the assumptions of Theorem~\ref{thm: main result}, suppose that the
	weak solution of~\eqref{eq: LLG system} in the considered regularity class
	is unique. Then, all convergences
	\eqref{eq: weak conv mhtau H1}--\eqref{eq: weak conv torque} hold for the full
	sequences, rather than only for subsequences, as $(h,\tau)\to(0,0)$.
\end{corollary}

	\begin{remark}\label{rem: energy torque}
		The additional strong consistency assumption in the last part of
		Theorem~\ref{thm: main result} is needed only for the energy inequality.
		Indeed, the uniform estimates and assumption~\eqref{enum: punto B4}
		yield
		\[
		\calQ_{h\tau}^+(\mhatp)
		\rightharpoonup
		\calQ(\mm),
		\qquad
		\vv_{h\tau}^-
		\rightharpoonup
		\partial_t\mm
		\qquad
		\text{in }\bL^2(\OT).
		\]
		These convergences are sufficient to identify the torque contribution
		in the weak formulation because the torque is paired there with a
		strongly convergent discrete test function. They do not, however,
		determine the limit of the product
		\[
		\int_0^T
		\scalarproductO{
			\calQ_{h\tau}^+(\mhatp)
		}{
			\vv_{h\tau}^-
		}\d t
		\]
		appearing in the discrete energy balance. Assumption~\eqref{enum: punto B5}
		provides the strong convergence of the first factor needed to pass to
		this term.
	\end{remark}

\begin{remark}
	The assumptions imposed above are slightly stronger than what is strictly
	required to identify the weak formulation, but they provide natural and
	readily verifiable sufficient conditions for the analysis.
	In particular, for the identification of the weak limit alone,
	weak convergence of the initial data in $\bH^1(\Omega)$
	and weak convergence of $\ff_{h\tau}^+$ in $\bL^2(\OT)$ would already be
	sufficient for the corresponding terms in the compactness argument.
	Likewise, for the lower-order operator it would be enough to assume a
	sequential consistency property guaranteeing
	\[
	\bpi_h(\mhatp)\rightharpoonup\bpi(\mm)
	\qquad\text{in }\bL^2(\OT),
	\]
	instead of~\eqref{enum: punto B3}.
		Similarly, assumption~\eqref{enum: punto B4} is stronger than the
		minimal condition needed for the non-conservative torque. For the
		weak-convergence argument it would be sufficient to know that
		$\calQ_{h\tau}^+(\mhatp)$ is uniformly bounded in $\bL^2(\OT)$ and
		that
		\[
		\calQ_{h\tau}^+(\mhatp)
		\rightharpoonup
		\calQ(\mm)
		\qquad\text{in }\bL^2(\OT).
		\]
		The more general formulation in~\eqref{enum: punto B4} provides a
		convenient and readily verifiable sequential consistency condition,
		while assumption~\eqref{enum: punto B5} is imposed only when the
		energy inequality is considered.
\end{remark}
\section{Proofs}\label{sec: proofs}

\subsection{Discrete energy identity}\label{section:energy identity}

	In this section, we derive the discrete energy identity and the uniform
	estimates required for the compactness argument. Compared with the
	corresponding results in \cite{part1}, the lower-order field, DMI, and
	non-conservative torque contributions give additional terms. The DMI
	terms require a separate argument exploiting the symmetry of $\HDM$,
	while the torque contribution is handled by assumption~\eqref{enum: punto B4}.

For ease of presentation, we define for  $1\le n \le N$ the so called $G$-norm associated with the BDF$2$ method, see~\cite{Dahlquist,nevanlinna_odeh,baiocchi_crouzeix,hairer}.
\begin{equation}\label{eq: Ghn}
	\begin{split}
	\mathcal{G}_h^n 
	& \coloneqq 
	\frac14\normO{\nabla\mm_h^{n-1}}^2 
	- \scalarproductO{\nabla\mm_h^n}{\nabla\mm_h^{n-1}} 
	+ \frac54\normO{\nabla\mm_h^{n}}^2 \\
	&\,= \frac14\big(\normO{\nabla\mm_h^n}^2+\normO{2\nabla\mm_h^n-\nabla\mm_h^{n-1}}^2\big).
	\end{split}
\end{equation}
The following lemma extends \cite[Lemma~3.5]{part1}.
\begin{lemma}[Discrete energy identity]\label{lem: energy inequality 2}
	Define
	\(
		\eta_0
		\coloneqq
		\normO{\nabla\mm_h^1}^2-\normO{\nabla\mm_h^0}^2,
	\)
	and, for uniformity of the notation, $\widehat{\mm}_h^1 \coloneqq \mm_h^0$.
	Then, for all $1\le n\le N$, Algorithm~\ref{alg: full discr} guarantees the discrete energy
	identity
	\begin{equation}\label{eq: energy identity}
		\begin{split}
			&
			\lamex^2 \mathcal{G}_h^n
			+
			\alpha\tau\sum_{j=0}^{n-1}\normO{\vv_h^j}^2
			+
			\frac{\lamex^2}{4}\tau^4
			\sum_{j=1}^{n-1}
			\normO{\nabla\d_t^2\mm_h^{j+1}}^2
			=
			\frac{\lamex^2}{2}\normO{\nabla\mm_h^0}^2
			+
			\frac{\lamex^2}{4}\eta_0
			-
			\tau^2 \HDM(\vv_h^0,\vv_h^0)
			\\
			&\qquad+
			\tau\sum_{j=0}^{n-1}
			\scalarproductO{
				\bpi_h(\widehat{\mm}_h^{j+1})
				+
				\ff_h^{j+1}
					+
					\calQ_h^{j+1}(\widehat{\mm}_h^{j+1})
			}{
				\vv_h^j
			}
			-
			\tau\sum_{j=0}^{n-1}
			\HDM(\widehat{\mm}_h^{j+1},\vv_h^j).
		\end{split}
	\end{equation}
	Moreover, there exist $\newconst{c:discrete_energy_1},\newconst{c:discrete_energy_2}\geq 0$ depending only on $\alpha$, $\lDM$, $\lamex$, $C_{\calQ}$ such that the following estimates hold for all $\tau>0$ satisfying $32\lDM^2\tau\le\alpha\lamex^2$: For $n=1$,	
		\begin{equation}\label{eq: energy inequality 1}
			\begin{split}
				&\frac{\lamex^2}{2}\normO{\nabla\mm_h^1}^2
				+
				\frac{\alpha\tau}{2}\normO{\vv_h^0}^2
				+
				\frac{\lamex^2\tau^2}{4}
				\normO{\nabla\vv_h^0}^2
				\\
				&\quad
				\le
				\frac{\lamex^2}{2}\normO{\nabla\mm_h^0}^2
				+
				\const{c:discrete_energy_1}\normO{\mm_h^0}^2
				+
				\const{c:discrete_energy_2}\tau
				\big(
				1
				+
				\normO{\bpi_h(\mm_h^0)}^2
				+
				\normO{\ff_h^1}^2
				+
				\norm{\mm_h^0}_{\bH^1(\Omega)}^2
				\big).
			\end{split}
		\end{equation}
	There also exist constants
	$\newconst{c:final-energy-1},\newconst{c:final-energy-2}>0$
	such that, for all $2\le n\le N$,
		\begin{equation}\label{eq: energy inequality 2}
			\begin{split}
				&\frac{\lamex^2\gamma^-}{2}
				\left(
				\normO{\nabla\mm_h^n}^2
				+
				\normO{\nabla\mm_h^{n-1}}^2
				\right)
				+
				\frac{\alpha\tau}{2}
				\sum_{j=1}^{n-1}
				\normO{\vv_h^j}^2
				+
				\frac{\lamex^2}{4}\tau^4
				\sum_{j=1}^{n-1}
				\normO{\nabla\d_t^2\mm_h^{j+1}}^2
				\\
				&\quad\le
				\const{c:final-energy-1}
				\left(
				1
				+
				\norm{\mm_h^0}_{\bH^1(\Omega)}^2
				+
				\norm{\mm_h^1}_{\bH^1(\Omega)}^2
				+
				\normO{\mm_h^{n-1}}^2
				+
				\normO{\mm_h^{n-2}}^2
				\right)
				\\
				&\qquad+
				\const{c:final-energy-2}\tau
				\sum_{j=1}^{n-1}
				\left(
				\norm{\mm_h^j}_{\bH^1(\Omega)}^2
				+
				\normO{\bpi_h(\widehat{\mm}_h^{j+1})}^2
				+
				\normO{\ff_h^{j+1}}^2
				\right),
			\end{split}
		\end{equation}
	where $\gamma^\pm=(3\pm2\sqrt{2})/{4}$
	are the eigenvalues of
	\[
	G
	\coloneqq
	\frac14
	\begin{pmatrix}
		1 & -2\\
		-2 & 5
	\end{pmatrix}.
	\]
	The constants $\const{c:final-energy-1}$ and
	$\const{c:final-energy-2}$ depend only on
	$\alpha$, $\lamex$, $\lDM$, $T$, $C_{\bpi}$,
	$C_{\calQ}$, $C_f$, and $|\Omega|$,
	and are independent of $h$, $\tau$, and $n$.
\end{lemma}

\begin{proof}
	The proof follows the same BDF2 identities as
	\cite[Lemma~3.5]{part1}. We only detail the additional contributions.
	
	\textbf{Step 1 (Proof of~(\ref{eq: energy identity}) for $\boldsymbol{n=1}$ and
		(\ref{eq: energy inequality 1})).}
	Testing~\eqref{eq: first step full discr} with $\vv_h^0$, multiplying by
	$\tau$, and using $\mm_h^1=\mm_h^0+\tau\vv_h^0$
	gives
	\begin{equation}\label{eq: energy proof 0}
		\begin{split}
			\alpha\tau\normO{\vv_h^0}^2
			+
			\frac{\lamex^2}{2}\normO{\nabla\mm_h^1}^2
			+
			\frac{\lamex^2\tau^2}{2}\normO{\nabla\vv_h^0}^2
			={}&
			\frac{\lamex^2}{2}\normO{\nabla\mm_h^0}^2
			-
			\tau\HDM(\mm_h^0,\vv_h^0)
			-
			\tau^2 \HDM(\vv_h^0,\vv_h^0)
			\\
			&+
			\tau\scalarproductO{
				\bpi_h(\mm_h^0)
				+
				\ff_h^1
					+
					\calQ_h^1(\mm_h^0)
			}{
				\vv_h^0
			}.
		\end{split}
	\end{equation}
	Since $\d _t \mm_h^1 = \vv_h^0$ and $\mathcal{G}_h^1 = \dfrac{1}{2}\normO{\nabla\mm_h^1}^2 + \dfrac{\tau^2}{2}\normO{\nabla\vv_h^0}^2+ \dfrac14 \eta_0$, the exchange terms on the left-hand side of \eqref{eq: energy proof 0} equal $\lamex^2\mathcal G_h^1-\frac{\lamex^2}{4}\eta_0$; hence this is~\eqref{eq: energy identity} for $n=1$.
		By~\eqref{eq: discrete torque bound},
		\[
		\normO{\calQ_h^1(\mm_h^0)}
		\le
		C_{\calQ}
		\left(
		1+\norm{\mm_h^0}_{\bH^1(\Omega)}
		\right).
		\]
		Hence, Young's inequality yields
		\begin{align*}
			&\tau
			\left|
			\scalarproductO{
				\bpi_h(\mm_h^0)
				+
				\ff_h^1
				+
				\calQ_h^1(\mm_h^0)
			}{
				\vv_h^0
			}
			\right|
			\\
			&\qquad\le
			\frac{\alpha\tau}{4}\normO{\vv_h^0}^2
			+
			\frac{2(1+C_{\calQ}^2)}{\alpha}\tau
			\big(
			1
			+
			\normO{\bpi_h(\mm_h^0)}^2
			+
			\normO{\ff_h^1}^2
			+
			\norm{\mm_h^0}_{\bH^1(\Omega)}^2
			\big).
		\end{align*}
	Moreover, by the definition of $\HDM$ and the estimate
	$\normO{\nabla\times\pphi}\le\sqrt{2}\normO{\nabla\pphi}$,
	we have
	\[
	|\HDM(\mm_h^0,\vv_h^0)|
	\le
	\frac{|\lDM|}{\sqrt{2}}
	\left(
	\normO{\nabla\mm_h^0}\normO{\vv_h^0}
	+
	\normO{\mm_h^0}\normO{\nabla\vv_h^0}
	\right)
	\]
	and
	\[
	|\HDM(\vv_h^0,\vv_h^0)|
	\le
	\sqrt{2}|\lDM|
	\normO{\vv_h^0}\normO{\nabla\vv_h^0}.
	\]
	Consequently, Young's inequality gives
	\begin{align*}
		&\tau|\HDM(\mm_h^0,\vv_h^0)|
		+
		\tau^2|\HDM(\vv_h^0,\vv_h^0)|
		\\
		&\quad\le
		\frac{|\lDM|}{\sqrt{2}} \tau
		\left(
		\normO{\nabla\mm_h^0}\normO{\vv_h^0}
		+
		\normO{\mm_h^0}\normO{\nabla\vv_h^0}
		\right)
		+
		\sqrt{2} |\lDM| \tau^2
		\normO{\vv_h^0}\normO{\nabla\vv_h^0}
		\\
		&\quad\le
		\frac{\alpha\tau}{8}\normO{\vv_h^0}^2
		+
		\frac{\lDM^2}{\alpha}\tau
		\normO{\nabla\mm_h^0}^2
		+
		\frac{\lamex^2\tau^2}{8}
		\normO{\nabla\vv_h^0}^2
		+
		\frac{\lDM^2}{\lamex^2}
		\normO{\mm_h^0}^2
		\\
		&\qquad+
		\frac{\alpha\tau}{8}\normO{\vv_h^0}^2
		+
		\frac{4\lDM^2}{\alpha}\tau^3
		\normO{\nabla\vv_h^0}^2
		\\
		&\quad=
		\frac{\alpha\tau}{4}\normO{\vv_h^0}^2
		+
		\Big(
		\frac{\lamex^2\tau^2}{8}
		+
		\frac{4\lDM^2}{\alpha}\tau^3
		\Big)
		\normO{\nabla\vv_h^0}^2
		+
		\frac{\lDM^2}{\alpha}\tau
		\normO{\nabla\mm_h^0}^2
		+
		\frac{\lDM^2}{\lamex^2}
		\normO{\mm_h^0}^2.
	\end{align*}
	Since
	\(
	\dfrac{4\lDM^2}{\alpha}\tau
	\le
	\dfrac{\lamex^2}{8}
	\)
	implied by the assumption $32\lDM^2\tau\le\alpha\lamex^2$,
	we have
	\begin{align}
		&\tau|\HDM(\mm_h^0,\vv_h^0)|
		+
		\tau^2|\HDM(\vv_h^0,\vv_h^0)|
		\nonumber\\
		&\quad\le
		\frac{\alpha\tau}{4}\normO{\vv_h^0}^2
		+
		\frac{\lamex^2\tau^2}{4}
		\normO{\nabla\vv_h^0}^2
		+
		\frac{\lDM^2}{\alpha}\tau
		\normO{\nabla\mm_h^0}^2
		+
		\frac{\lDM^2}{\lamex^2}
		\normO{\mm_h^0}^2.
		\label{eq: first step DMI bound}
	\end{align}
	Inserting these estimates into~\eqref{eq: energy proof 0} and absorbing
	$\dfrac{\alpha\tau}{2}\normO{\vv_h^0}^2$ and $\dfrac{\lamex^2\tau^2}{4}\normO{\nabla\vv_h^0}^2$ 
	into the left-hand side
	gives~\eqref{eq: energy inequality 1}, with the choices
	\[
	\const{c:discrete_energy_1}
	=
	\frac{\lDM^2}{\lamex^2},
	\qquad
	\const{c:discrete_energy_2}
	=
	\frac{2(1+C_{\calQ}^2)+\lDM^2}{\alpha}.
	\]
	
	\textbf{Step 2 (Proof of (\ref{eq: energy identity}) for
		$\boldsymbol{n\ge2}$).}
	The BDF2 update gives
	\[
	\vv_h^j
	=
	\frac{1}{\tau}
	\Big(
	\frac32\mm_h^{j+1}
	-
	2\mm_h^j
	+
	\frac12\mm_h^{j-1}
	\Big).
	\]
	Testing~\eqref{eq: full discr alg B} with
	$\vv_h^j\in\bT_h(\widehat{\mm}_h^{j+1})$ and multiplying by $\tau$
	yields
	\begin{equation}\label{eq: energy proof 1}
		\begin{split}
			\alpha\tau\normO{\vv_h^j}^2
			+
			\lamex^2\tau
			\scalarproductO{\nabla\mm_h^{j+1}}{\nabla\vv_h^j}
			={}&
			\tau
			\scalarproductO{
				\bpi_h(\widehat{\mm}_h^{j+1})
				+
				\ff_h^{j+1}
					+
					\calQ_h^{j+1}(\widehat{\mm}_h^{j+1})
			}{
				\vv_h^j
			}
			\\
			&-
			\tau\HDM(\widehat{\mm}_h^{j+1},\vv_h^j).
		\end{split}
	\end{equation}
	As in \cite[Lemma~3.5]{part1}, we have, for all $j=1,\dots,N-1$,
	\begin{equation}\label{eq: BDF2 exchange identity}
		\begin{split}
			\tau
			\scalarproductO{\nabla\mm_h^{j+1}}{\nabla\vv_h^j}
			= \mathcal{G}_h^{j+1}-\mathcal{G}_h^j+\frac14\tau^4\normO{\nabla\d_t^2\mm_h^{j+1}}^2.
		\end{split}
	\end{equation}
	Combining this with~\eqref{eq: energy proof 1} and summing it over $j=1,\dots,n-1$ gives
	\begin{equation}\label{eq: energy proof G}
		\begin{split}
			&\alpha\tau\sum_{j=1}^{n-1}\normO{\vv_h^j}^2
			+
			\frac{\lamex^2}{4}\tau^4
			\sum_{j=1}^{n-1}
			\normO{\nabla\d_t^2\mm_h^{j+1}}^2
			+
			\lamex^2
			\mathcal{G}_h^n
			\\
			&=
			\lamex^2
			\mathcal{G}_h^1
			+
			\tau\sum_{j=1}^{n-1}
			\scalarproductO{
				\bpi_h(\widehat{\mm}_h^{j+1})
				+
				\ff_h^{j+1}
					+
					\calQ_h^{j+1}(\widehat{\mm}_h^{j+1})
			}{
				\vv_h^j
			}
			-
			\tau\sum_{j=1}^{n-1}
			\HDM(\widehat{\mm}_h^{j+1},\vv_h^j).
		\end{split}
	\end{equation}
	Recalling the definition of $\mathcal{G}_h^1$, and
	adding~\eqref{eq: energy proof 0} to include also the initial time step proves~\eqref{eq: energy identity}.

	\textbf{Step 3 (Proof of (\ref{eq: energy inequality 2})).}
	We aim to estimate the last two terms on the right-hand side of~\eqref{eq: energy proof G}. First we estimate the DMI contribution. 
	By the definition of $\HDM$, there holds
	\begin{align*}
    \tau\sum_{j=1}^{n-1}
    \HDM(\widehat{\mm}_h^{j+1},\vv_h^j)
    &=
    \frac{\lDM}{2}\tau\sum_{j=1}^{n-1}
    \scalarproductO{
        \nabla\times\widehat{\mm}_h^{j+1}
    }{
        \vv_h^j
    }
    +
    \frac{\lDM}{2}\sum_{j=1}^{n-1}
    \scalarproductO{
        \widehat{\mm}_h^{j+1}
    }{
        \nabla\times(\tau\vv_h^j)
    }
    \eqqcolon T_1+T_2.
\end{align*}
For the first term, using
$\normO{\nabla\times\pphi}\le\sqrt{2}\normO{\nabla\pphi}$,
Young's inequality $\frac{|\lDM|}{\sqrt{2}} ab \leq \frac{\alpha}{8} a^2 + \frac{\lDM^2}{\alpha} b^2$, and the definition of
$\widehat{\mm}_h^{j+1}=2\mm_h^j-\mm_h^{j-1}$, we obtain
\begin{align*}
    |T_1|
    \le
    \frac{|\lDM|}{\sqrt{2}}\tau
    \sum_{j=1}^{n-1}
    \normO{\nabla\widehat{\mm}_h^{j+1}}
    \normO{\vv_h^j}
    &\le
    \frac{\alpha\tau}{8}
    \sum_{j=1}^{n-1}\normO{\vv_h^j}^2
    +
    \frac{\lDM^2\tau}{\alpha}
    \sum_{j=1}^{n-1}
    \normO{\nabla\widehat{\mm}_h^{j+1}}^2
    \\
    &\le
    \frac{\alpha\tau}{8}
    \sum_{j=1}^{n-1}\normO{\vv_h^j}^2
    +
    \frac{9\lDM^2\tau}{\alpha}
    \sum_{j=0}^{n-1}\normO{\nabla\mm_h^j}^2.
\end{align*}
To estimate $T_2$, we apply summation by parts in time.
To this end, set the auxiliary extrapolation
\begin{equation}\label{eq: definition mtilde}
    \widetilde{\mm}_h^{j+1}
    \coloneqq
    \frac32\mm_h^j-\frac12\mm_h^{j-1},
    \qquad \text{for}\quad j=1,\dots,N,
\end{equation}
and note that the BDF2 update gives
\[
    \tau\vv_h^j
    =
    \widetilde{\mm}_h^{j+2}
    -
    \widetilde{\mm}_h^{j+1},
    \qquad \text{for}\quad j=1,\dots,N-1.
\]
Consequently, for $n\ge2$ we have
\begin{align*}
    T_2
    &=
    \frac{\lDM}{2}\sum_{j=1}^{n-1}
    \scalarproductO{
        \widehat{\mm}_h^{j+1}
    }{
        \nabla\times
        (
            \widetilde{\mm}_h^{j+2}
            -
            \widetilde{\mm}_h^{j+1}
        )
    }
    \\
    &=
    \frac{\lDM}{2}
    \scalarproductO{
        \widehat{\mm}_h^n
    }{
        \nabla\times\widetilde{\mm}_h^{n+1}
    }
    -
    \frac{\lDM}{2}
    \scalarproductO{
        \widehat{\mm}_h^2
    }{
        \nabla\times\widetilde{\mm}_h^2
    }
    -
    \frac{\lDM}{2}\sum_{j=2}^{n-1}
    \scalarproductO{
        \widehat{\mm}_h^{j+1}-\widehat{\mm}_h^j
    }{
        \nabla\times\widetilde{\mm}_h^{j+1}
    }.
\end{align*}
Since
\(
    \widehat{\mm}_h^{j+1}-\widehat{\mm}_h^j
    =
    \tau\bigl(
        2\d_t\mm_h^j-\d_t\mm_h^{j-1}
    \bigr),
\)
the norm equivalence
\eqref{eq: norm equivalence 2} in Appendix~\ref{appendix} implies
\[
\frac1\tau\sum_{j=2}^{n-1}
\normO{\widehat{\mm}_h^{j+1}-\widehat{\mm}_h^j}^2
\le
9\tau\sum_{j=1}^{n-1}\normO{\d_t\mm_h^j}^2
\le
\frac{81}{7}\tau\sum_{j=0}^{n-1}\normO{\vv_h^j}^2.
\]
Therefore, by definition~\eqref{eq: definition mtilde} and 

\[
\tau\sum_{j=2}^{n-1}
\|\nabla\times\widetilde{\boldsymbol m}_h^{j+1}\|_\Omega^2
\le
8\tau\sum_{j=0}^{n-1}
\|\nabla\boldsymbol m_h^j\|_\Omega^2,
\]
defining $\newconst{c:product}\coloneq |\lDM|9\sqrt{{2}/{7}}$,
the
Cauchy--Schwarz and Young's inequality prove that
\begin{align*}
    \Big|
        \frac{\lDM}{2}\sum_{j=2}^{n-1}
        \scalarproductO{
            \widehat{\mm}_h^{j+1}-\widehat{\mm}_h^j
        }{
            \nabla\times\widetilde{\mm}_h^{j+1}
        }
    \Big|
    &\le
    \frac{|\lDM|}{2}\frac{9\sqrt{7}}{7}
    \Big(
        \tau\sum_{j=0}^{n-1}\normO{\vv_h^j}^2
    \Big)^{1/2}
	\sqrt{8}
    \Big(
        \tau\sum_{j=0}^{n-1}\normO{\nabla\mm_h^j}^2
    \Big)^{1/2}
    \\
    &\le
    \frac{\alpha\tau}{8}
    \sum_{j=0}^{n-1}\normO{\vv_h^j}^2
    +
    \frac{2 \const{c:product}^2}{\alpha}\tau\sum_{j=0}^{n-1}\normO{\nabla\mm_h^j}^2.
\end{align*}

For the endpoint terms, the definition of the extrapolated
quantities and Young's inequality give
\begin{align*}
    &\Big|
        \frac{\lDM}{2}
        \scalarproductO{
            \widehat{\mm}_h^n
        }{
            \nabla\times\widetilde{\mm}_h^{n+1}
        }
    \Big|
    \le
    \frac{\lamex^2\gamma^-}{2}
    \big(
        \normO{\nabla\mm_h^n}^2
        +
        \normO{\nabla\mm_h^{n-1}}^2
    \big)
    +
    \frac{25\lDM^2}{8\lamex^2\gamma^-}
	\big(
        \normO{\mm_h^{n-1}}^2
        +
        \normO{\mm_h^{n-2}}^2
    \big)
\end{align*}
and
\[
    \Big|
        \frac{\lDM}{2}
        \scalarproductO{
            \widehat{\mm}_h^2
        }{
            \nabla\times\widetilde{\mm}_h^2
        }
    \Big|
    \le
    \frac{5}{4}|\lDM|\big(
        \norm{\mm_h^0}_{\bH^1(\Omega)}^2
        +
        \norm{\mm_h^1}_{\bH^1(\Omega)}^2
    \big).
\]
Set
$\newconst{c:dm-bound}:=
\max\left\{
    \dfrac{25\lDM^2}{8\lamex^2\gamma^-},
    \dfrac54|\lDM|
\right\}$.
Therefore, it holds that
\begin{align*}
|T_2|
&\le
\frac{\alpha\tau}{8}
\sum_{j=0}^{n-1}\normO{\vv_h^j}^2
+
\frac{\lamex^2\gamma^-}{2}
\left(
\normO{\nabla\mm_h^n}^2
+
\normO{\nabla\mm_h^{n-1}}^2
\right)
\\
&+
\frac{2\const{c:product}^2}{\alpha}\tau\sum_{j=0}^{n-1}\normO{\nabla\mm_h^j}^2
+
\const{c:dm-bound}
\left(
\norm{\mm_h^0}_{\bH^1(\Omega)}^2
+
\norm{\mm_h^1}_{\bH^1(\Omega)}^2
+
\normO{\mm_h^{n-1}}^2
+
\normO{\mm_h^{n-2}}^2
\right).
\end{align*}
Combining the bounds for $T_1$ and $T_2$ implies
\begin{align*}
&\Big|
        \tau\sum_{j=1}^{n-1}
        \HDM(\widehat{\mm}_h^{j+1},\vv_h^j)
\Big|
\le
\frac{\alpha\tau}{4}
\sum_{j=1}^{n-1}\normO{\vv_h^j}^2
+
\frac{\alpha\tau}{8}
\normO{\vv_h^0}^2
+
\frac{\lamex^2\gamma^-}{2}
\left(
\normO{\nabla\mm_h^n}^2
+
\normO{\nabla\mm_h^{n-1}}^2
\right)
\\
&+
\Big(\frac{9\lDM^2}{\alpha}+\frac{2\const{c:product}^2}{\alpha}\Big)\tau\sum_{j=0}^{n-1}\normO{\nabla\mm_h^j}^2
+
\const{c:dm-bound}\left(
\norm{\mm_h^0}_{\bH^1(\Omega)}^2
+
\norm{\mm_h^1}_{\bH^1(\Omega)}^2
+
\normO{\mm_h^{n-1}}^2
+
\normO{\mm_h^{n-2}}^2
\right).
\end{align*}
By~\eqref{eq: energy inequality 1},
assumptions~\eqref{enum: punto B2}--\eqref{enum: punto B3},
and $\tau\le T$, we obtain
\begin{align*}
\frac{\alpha\tau}{8}\normO{\vv_h^0}^2
\le&
\frac{\lambda_{\rm ex}^2}{8}
\normO{\nabla\mm_h^0}^2
+\frac{\const{c:discrete_energy_1}}{4}\normO{\mm_h^0}^2
+\frac{\const{c:discrete_energy_2}\tau}{4}
\Big(
1+\normO{\boldsymbol\pi_h(\mm_h^0)}^2
+\normO{\ff_h^1}^2
+\norm{\mm_h^0}_{\bH^1(\Omega)}^2
\Big)\\
\le&
\Big[
\frac{\lambda_{\rm ex}^2}{8}
+\frac{\const{c:discrete_energy_1}}{4}
+\frac{T\const{c:discrete_energy_2}}{4}(1+C_{\boldsymbol\pi}^2)
\Big]
\|\mm_h^0\|_{\bH^1(\Omega)}^2+\frac{T\const{c:discrete_energy_2}}{4}(1+|\Omega|C_f^2).
\end{align*}
Defining
\[
\newconst{c:initial-work}
\coloneqq
\max\left\{
	\frac{\lamex^2}{8}+\frac{\const{c:discrete_energy_1}}{4}
	+\frac{T\const{c:discrete_energy_2}}{4}(1+C_{\bpi}^2),
	\frac{T\const{c:discrete_energy_2}}{4}(1+|\Omega|C_f^2)
\right\},
\]
this gives
\begin{align*}
	\frac{\alpha\tau}{8}\|\vv_h^0\|_\Omega^2
	\le&
	\const{c:initial-work}
	\big(
		1+\norm{\mm_h^0}_{\bH^1(\Omega)}^2
	\big).
\end{align*}

We thus obtain, after defining $\newconst{c:dm-total}\coloneqq\const{c:dm-bound}+\const{c:initial-work}$
\begin{align}
    \Big|
        \tau\sum_{j=1}^{n-1}
        \HDM(\widehat{\mm}_h^{j+1},\vv_h^j)
    \Big|
    &\le
    \frac{\alpha\tau}{4}
    \sum_{j=1}^{n-1}\normO{\vv_h^j}^2
    +
    \frac{\lamex^2\gamma^-}{2}
    \big(
        \normO{\nabla\mm_h^n}^2
        +
        \normO{\nabla\mm_h^{n-1}}^2
    \big)\nonumber
    \\
    &+
    \Big(\frac{9\lDM^2}{\alpha}+\frac{2\const{c:product}^2}{\alpha}\Big)\tau\sum_{j=0}^{n-1}\normO{\nabla\mm_h^j}^2\label{eq: bound DMI contribution}
    \\
	&+
    \const{c:dm-total}
	\big(
        1
        +
        \norm{\mm_h^0}_{\bH^1(\Omega)}^2
        +
        \norm{\mm_h^1}_{\bH^1(\Omega)}^2
        +
        \normO{\mm_h^{n-1}}^2
        +
        \normO{\mm_h^{n-2}}^2
    \big).\nonumber
\end{align}
To estimate the second term in the right-hand side of \eqref{eq: energy proof G},
notice that, by~\eqref{eq: discrete torque bound} and
		$\widehat{\mm}_h^{j+1}=2\mm_h^j-\mm_h^{j-1}$,
		\begin{align}
			\tau\sum_{j=1}^{n-1}
			\normO{
				\calQ_h^{j+1}(\widehat{\mm}_h^{j+1})
			}^2
			&\le
			2C^2_{\calQ}\tau\sum_{j=1}^{n-1}
			\big(
			1+
			\norm{\widehat{\mm}_h^{j+1}}_{\bH^1(\Omega)}^2
			\big)
			\le
			2C^2_{\calQ}T
			+
			18 C^2_{\calQ}\tau\sum_{j=0}^{n-1}
			\norm{\mm_h^j}_{\bH^1(\Omega)}^2.
			\label{eq: bound Q energy}
		\end{align}
		This implies, by Cauchy-Schwarz and Young's inequalities, that
		\begin{align}
			&\tau\sum_{j=1}^{n-1}
			\left|
			\scalarproductO{
				\bpi_h(\widehat{\mm}_h^{j+1})
				+
				\ff_h^{j+1}
				+
				\calQ_h^{j+1}(\widehat{\mm}_h^{j+1})
			}{
				\vv_h^j
			}
			\right|
			\nonumber\\
			&\qquad\le
			\frac{\alpha\tau}{4}
			\sum_{j=1}^{n-1}\normO{\vv_h^j}^2
			+
			\frac{3}{\alpha}
			\tau\sum_{j=1}^{n-1}
			\left(
			\normO{\bpi_h(\widehat{\mm}_h^{j+1})}^2
			+
			\normO{\ff_h^{j+1}}^2
			+
			\normO{
				\calQ_h^{j+1}(\widehat{\mm}_h^{j+1})
			}^2
			\right)\label{eq: forcing energy estimate}\\
			&\qquad\le
			\frac{\alpha\tau}{4}
			\sum_{j=1}^{n-1}\normO{\vv_h^j}^2
			+
			\frac{3}{\alpha}
			\tau\sum_{j=1}^{n-1}
			\left(
				\normO{\bpi_h(\widehat{\mm}_h^{j+1})}^2
				+
				\normO{\ff_h^{j+1}}^2
			\right)
			+
			\frac{6C_{\calQ}^2T}{\alpha}
			+
			\frac{54C_{\calQ}^2}{\alpha}\tau\sum_{j=0}^{n-1}
			\norm{\mm_h^j}_{\bH^1(\Omega)}^2.
			\nonumber
		\end{align}
	By~\eqref{eq: Ghn}, $\mathcal{G}_h^n$ is the quadratic form associated
    with $G$ evaluated at $(\nabla\mm_h^{n-1},\nabla\mm_h^n)$. Hence, the eigenvalue
    bounds for $G$ give, for all $1\le n\le N$,
    \[
    \gamma^-\left(\normO{\nabla\mm_h^{n-1}}^2+\normO{\nabla\mm_h^n}^2\right)
    \le
    \mathcal{G}_h^n
    \le
    \gamma^+\left(\normO{\nabla\mm_h^{n-1}}^2+\normO{\nabla\mm_h^n}^2\right),
    \]
    which we apply to $\mathcal{G}_h^n$ on the left-hand side and to $\mathcal{G}_h^1$
    on the right-hand side of~\eqref{eq: energy proof G}.
	
	Inserting~\eqref{eq: bound DMI contribution}
	and~\eqref{eq: forcing energy estimate}
	into~\eqref{eq: energy proof G}, and 
	absorbing the velocity and terminal gradient terms into
	the left-hand side, we obtain
	\begin{align*}
		&\frac{\lamex^2\gamma^-}{2}
		\left(
			\normO{\nabla\mm_h^n}^2
			+
			\normO{\nabla\mm_h^{n-1}}^2
		\right)
		+
		\frac{\alpha\tau}{2}
		\sum_{j=1}^{n-1}\normO{\vv_h^j}^2
		+
		\frac{\lamex^2\tau^4}{4}
		\sum_{j=1}^{n-1}
		\normO{\nabla\d_t^2\mm_h^{j+1}}^2
		\\
		&\quad\le
	\Big(
    	\const{c:dm-total}
    	+\lamex^2\gamma^+
    	+\frac{6C_{\calQ}^2T}{\alpha}
		\Big)
		\left(
			1
			+
			\norm{\mm_h^0}_{\bH^1(\Omega)}^2
			+
			\norm{\mm_h^1}_{\bH^1(\Omega)}^2
			+
			\normO{\mm_h^{n-1}}^2
			+
			\normO{\mm_h^{n-2}}^2
		\right)
		\\
		&\qquad+
		\Big(
			\frac{54 C_{\calQ}^2}{\alpha} + \frac{9\lDM^2}{\alpha} + \frac{2\const{c:product}^2}{\alpha}
		\Big)\tau\sum_{j=0}^{n-1}
		\norm{\mm_h^j}_{\bH^1(\Omega)}^2
		+
		\frac{3}{\alpha}\tau\sum_{j=1}^{n-1}
		\left(
			\normO{\bpi_h(\widehat{\mm}_h^{j+1})}^2
			+
			\normO{\ff_h^{j+1}}^2
		\right).
	\end{align*}
	Finally, since $\tau\le T$, we have
\[
    \tau\sum_{j=0}^{n-1}
    \norm{\mm_h^j}_{\bH^1(\Omega)}^2
    \le
    T\norm{\mm_h^0}_{\bH^1(\Omega)}^2
    +
    \tau\sum_{j=1}^{n-1}
    \norm{\mm_h^j}_{\bH^1(\Omega)}^2.
\]
The first term can therefore be included in the initial-data
contribution. Consequently, \eqref{eq: energy inequality 2}
holds with the choices
\[
\begin{aligned}
    \const{c:final-energy-1}
    &\coloneqq
    \const{c:dm-total}
    +\lamex^2\gamma^+
    +\frac{T}{\alpha}
    \Big(
        60C_{\calQ}^2
        +9\lDM^2
        +2\const{c:product}^2
    \Big),
    \\
    \const{c:final-energy-2}
    &\coloneqq
    \frac{
        54C_{\calQ}^2
        +9\lDM^2
        +2\const{c:product}^2
        +3
    }{\alpha}.
\end{aligned}
\]
These constants depend only on
$\alpha$, $\lamex$, $\lDM$, $T$, $C_{\bpi}$,
$C_{\calQ}$, $C_f$, and $|\Omega|$,
and are independent of $h$, $\tau$, and $n$.
This proves~\eqref{eq: energy inequality 2}
and concludes the proof of the lemma.
\end{proof}

The previous lemma and assumptions
\eqref{enum: punto A1}--\eqref{enum: punto A2} and
\eqref{enum: punto B1}--\eqref{enum: punto B4}
allow us to state and prove the following proposition, which extends
\cite[Propositions~3.7--3.8]{part1} to the present setting.

\begin{proposition}[Boundedness of discrete approximations]
	\label{prop: bounds}
	For every $0\le n\le N$, and for all $\tau>0$ 
	satisfying $32\lDM^2\tau\le\alpha\lamex^2$, 
	Algorithm~\ref{alg: full discr} guarantees
	\begin{equation}\label{eq: uniform bound}
		\begin{split}
			\norm{\mm_h^n}_{\bH^1(\Omega)}^2
			&+
			\tau\sum_{j=0}^{n-1}\normO{\vv_h^j}^2
			+
			\tau^2\normO{\nabla\vv_h^0}^2
			+
			\tau^4\sum_{j=1}^{n-1}
			\normO{\nabla\d_t^2\mm_h^{j+1}}^2
			\le C.
		\end{split}
	\end{equation}
	Moreover, the violation of the pointwise modulus constraint satisfies
	\begin{equation}\label{eq: modulus defect final}
		\norm{|\mm_h^n|^2-1}_{L^1(\Omega)}
		\le
		C(\tau+h^2).
	\end{equation}
	The constant $C>0$ depends only on
	$\gamma^\pm$, $\lamex$, $\lDM$, $\alpha$, $T$, $C_0$, $C_f$,
	$C_{\bpi}$, $C_{\calQ}$, and the $\gamma$-quasi-uniformity of
	$\TT_h$, but is independent of $h$, $\tau$, and $n$.
\end{proposition}

\begin{proof}
	Assumptions~\eqref{enum: punto B1}--\eqref{enum: punto B4} and
	\eqref{eq: energy inequality 1} imply
	\begin{equation}\label{eq: first step uniform bound}
		\norm{\mm_h^1}_{\bH^1(\Omega)}^2
		+
		\tau\normO{\vv_h^0}^2
		+
		\tau^2\normO{\nabla\vv_h^0}^2
		\lesssim1.
	\end{equation}
	For $j\ge1$, assumption~\eqref{enum: punto B3} and
	$\widehat{\mm}_h^{j+1}=2\mm_h^j-\mm_h^{j-1}$ give
	\[
	\normO{\bpi_h(\widehat{\mm}_h^{j+1})}^2
	\lesssim
	\normO{\mm_h^j}^2
	+
	\normO{\mm_h^{j-1}}^2.
	\]
	Hence, using also assumption~\eqref{enum: punto B2},
	estimate~\eqref{eq: energy inequality 2} yields, for every $n\ge2$,
	\begin{equation}\label{eq: bound m H1}
		\begin{split}
			&\normO{\nabla\mm_h^n}^2
			+
			\tau\sum_{j=0}^{n-1}\normO{\vv_h^j}^2
			+
			\tau^4\sum_{j=1}^{n-1}
			\normO{\nabla\d_t^2\mm_h^{j+1}}^2
			\\
			&\qquad\lesssim
			1
			+
			\normO{\mm_h^{n-1}}^2
			+
			\normO{\mm_h^{n-2}}^2
			+
			\tau\sum_{j=0}^{n-1}
			\norm{\mm_h^j}_{\bH^1(\Omega)}^2.
		\end{split}
	\end{equation}
	
	Now, we recall that the nodal argument in the proof of
	\cite[Proposition~3.8]{part1} depends only on the BDF2 update and the
	nodal tangent-space orthogonality and therefore applies unchanged up
	to the pointwise identity \cite[equation~(34)]{part1}. In particular, for every
	$n\ge1$ and $\zz\in\NN_h$,
	\begin{equation}\label{eq: nodal modulus defect bound}
		0
		\le
		|\mm_h^n(\zz)|^2-1
		\lesssim
		\tau^2|\vv_h^0(\zz)|^2
		+
		\tau^4
		\sum_{j=1}^{n-1}
		|\d_t^2\mm_h^{j+1}(\zz)|^2.
	\end{equation}
	Here, the case $n=1$ follows directly from
	$\mm_h^1=\mm_h^0+\tau\vv_h^0$ and the nodal orthogonality
	$\mm_h^0(\zz)\cdot\vv_h^0(\zz)=0$.
	
	Let $I_h$ denote the scalar nodal interpolation operator. Since the
	nodal values in~\eqref{eq: nodal modulus defect bound} are
	non-negative, the discrete norm equivalence and
	\eqref{eq: inverse estimate} yield
	\begin{align}
		\norm{
			I_h\bigl(|\mm_h^n|^2-1\bigr)
		}_{L^1(\Omega)}
		&\lesssim
		\tau^2\normO{\vv_h^0}^2
		+
		\tau^4
		\sum_{j=1}^{n-1}
		\normO{\d_t^2\mm_h^{j+1}}^2
		\eqreff{eq: inverse estimate}{\lesssim}
		\tau^2
		\sum_{j=0}^{n-1}
		\normO{\vv_h^j}^2.
		\label{eq: interpolated modulus defect bound}
	\end{align}
	
	For the present stability argument, we first use the nodal estimate
	directly to control the $L^2$-norm of $\mm_h^n$. Indeed, the discrete
	norm equivalence gives
	\begin{align}
		\normO{\mm_h^n}^2
		&\lesssim
		h^d\sum_{\zz\in\NN_h}|\mm_h^n(\zz)|^2
		\lesssim
		1
		+
		h^d\sum_{\zz\in\NN_h}
		\bigl(|\mm_h^n(\zz)|^2-1\bigr)
		\lesssim
		1
		+
		\tau^2 \sum_{j=0}^{n-1}
		\normO{\vv_h^j}^2.
		\label{eq: L2 magnetization bound preliminary}
	\end{align}
	Now, for $n\ge1$, let
	\[
	E_n
	\coloneqq
	\normO{\nabla\mm_h^n}^2
	+
	\tau\sum_{j=0}^{n-1}\normO{\vv_h^j}^2
	+
	\tau^4\sum_{j=1}^{n-1}
	\normO{\nabla\d_t^2\mm_h^{j+1}}^2,
	\]
	and set $E_0\coloneqq\normO{\nabla\mm_h^0}^2$.
	Then~\eqref{eq: L2 magnetization bound preliminary} implies, for every
	$1\le j\le n$,
	\begin{equation}\label{eq: L2 bound by En}
		\normO{\mm_h^j}^2
		\lesssim
		1+\tau E_j.
	\end{equation}
	In particular, we have $\norm{\mm_h^j}_{\bH^1(\Omega)}^2 \lesssim 1+E_j$.
	Inserting these estimates into~\eqref{eq: bound m H1} gives
	\[
	E_n
	\lesssim
	1
	+
	\tau E_{n-1}
	+
	\tau E_{n-2}
	+
	\tau\sum_{j=0}^{n-1}E_j
	\lesssim
	1
	+
	\tau\sum_{j=0}^{n-1}E_j.
	\]
	Therefore, the discrete Gronwall inequality
	(Lemma~\ref{lem: gronwall}), together with
	\eqref{eq: first step uniform bound}, yields
	\[
	E_n\lesssim1
	\qquad
	\text{for all }0\le n\le N.
	\]
	Combining this with~\eqref{eq: L2 bound by En} proves
	\eqref{eq: uniform bound}.
	
	It remains to translate the nodal modulus estimate into an estimate
	for the actual finite element function. By
	\eqref{eq: interpolated modulus defect bound} and
	\eqref{eq: uniform bound},
	\begin{equation}\label{eq: modulus defect Ih}
	\norm{
		I_h\bigl(|\mm_h^n|^2-1\bigr)
	}_{L^1(\Omega)}
	\lesssim\tau.
	\end{equation}
	On each $K\in\TT_h$, the function $|\mm_h^n|^2-1$ is a quadratic
	polynomial. Hence, by the local nodal interpolation estimate,
	\begin{align*}
		\norm{
			(\operatorname{Id}-I_h)
			\bigl(|\mm_h^n|^2-1\bigr)
		}_{L^1(K)}
		&\lesssim
		h_K^2
		\norm{
			D^2|\mm_h^n|^2
		}_{L^1(K)}
		\lesssim
		h_K^2
		\norm{\nabla\mm_h^n}_{\bL^2(K)}^2.
	\end{align*}
	Summing over $K\in\TT_h$, then using~\eqref{eq: uniform bound} and~\eqref{eq: modulus defect Ih}, we obtain \eqref{eq: modulus defect final}. This concludes the proof. 
\end{proof}

\begin{remark}[Quadratic constraint violation] \label{rem: quadratic modulus defect}
	The estimate~\eqref{eq: modulus defect final} only uses the stability established in \eqref{eq: uniform bound}. If the discrete approximations additionally satisfy the regularity estimate
	\begin{equation}\label{eq: add reg for 2nd order}
		\normO{\vv_h^0}^2
		+
		\tau^2
		\sum_{j=1}^{N-1}
		\normO{\d_t^2\mm_h^{j+1}}^2
		\le C_{\rm reg},
	\end{equation}
	with $C_{\rm reg}>0$ independent of $h$ and $\tau$, then the temporal order of the modulus defect improves. Indeed, retaining the first estimate in~\eqref{eq: interpolated modulus defect bound} yields
	\[
	\norm{I_h\bigl(|\mm_h^n|^2-1\bigr)}_{L^1(\Omega)}
	\lesssim
	\tau^2
	\qquad
	\text{for all }0\le n\le N.
	\]
	Combining this with the local interpolation estimate used at the end of the proof of Proposition~\ref{prop: bounds} and~\eqref{eq: uniform bound}, we obtain
	\begin{equation}
	\norm{|\mm_h^n|^2-1}_{L^1(\Omega)}
	\lesssim
	\tau^2+h^2
	\qquad
	\text{for all }0\le n\le N.
	\label{eq: unit defect quadratic}
	\end{equation}
	This is the analogue of the result discussed in~\cite[Remark~3.9]{part1}, and is consistent with the quadratic temporal decay of the constraint violation observed numerically.
\end{remark}

The following Proposition~\ref{prop: bounds interpolants} extends
\cite[Proposition~3.10]{part1} and allows us to prove that also the
time interpolations defined in
\eqref{eq:mhtau}--\eqref{eq:vhtau-minus} are uniformly bounded.

\begin{proposition}[Uniform boundedness of interpolants]
	\label{prop: bounds interpolants}
	Under the assumptions of Theorem~\ref{thm: main result}, it holds that
	\begin{equation}\label{eq: uniform bound S-T norms}
		\begin{split}
			\norm{\mm_{h\tau}}_{\bH^1(\OT)}
			&+
			\norm{\mm_{h\tau}^{\star}}_{L^\infty(0,T;\bH^1(\Omega))}
			+
			\norm{\vv_{h\tau}^-}_{\bL^2(\OT)}
			\\
			&+
			\norm{\bpi_h(\mhatp)}_{\bL^2(\OT)}
			+
			\norm{\calQ_{h\tau}^+(\mhatp)}_{\bL^2(\OT)}
			\le C,
		\end{split}
	\end{equation}
	where
	\[
	\mm_{h\tau}^\star
	\in
	\big\{
	\mm_{h\tau},
	\mm_{h\tau}^{\pm},
	\mhatp
	\big\},
	\]
	and $C>0$ depends only on
	$\gamma^\pm$, $\lamex$, $\lDM$, $\alpha$, $T$, $C_0$, $C_f$,
	$C_{\bpi}$, $C_{\calQ}$, and the $\gamma$-quasi-uniformity of
	$\TT_h$, but is independent of $h$ and $\tau$.
\end{proposition}

\begin{proof}
	\cite[Proposition~3.10]{part1} allows us to state that, for
	\[
	\mm_{h\tau}^\star
	\in
	\big\{
	\mm_{h\tau},
	\mm_{h\tau}^{\pm},
	\mhatp
	\big\},
	\]
	it holds that
	\[
	\norm{\mm_{h\tau}}_{\bH^1(\OT)}
	+
	\norm{\mm_{h\tau}^{\star}}_{L^\infty(0,T;\bH^1(\Omega))}
	+
	\norm{\vv_{h\tau}^-}_{\bL^2(\OT)}
	\lesssim1.
	\]
	The proof is a direct consequence of
	Proposition~\ref{prop: bounds} and the definitions of the space-time
	interpolants~\eqref{eq:mhtau}--\eqref{eq:vhtau-minus}.
	We remark that it does not depend on the explicit expression of the
	effective field or of the non-conservative torque and therefore applies
	verbatim in the present setting; see
	\cite[Proposition~3.10]{part1} for the details.
	
	The uniform boundedness of the lower-order field follows from
	assumption~\eqref{enum: punto B3} and the bound for $\mhatp$, namely,
	\[
	\norm{\bpi_h(\mhatp)}_{\bL^2(\OT)}
	\lesssim
	\norm{\mhatp}_{\bL^2(\OT)}
	\lesssim1.
	\]
	Finally, by~\eqref{eq: discrete torque bound},
	\begin{align*}
		\norm{\calQ_{h\tau}^+(\mhatp)}_{\bL^2(\OT)}^2
		&=
		\tau\sum_{j=0}^{N-1}
		\normO{
			\calQ_h^{j+1}
			\bigl(
			\widehat{\mm}_{h\tau}^+|_{[t_j,t_{j+1})}
			\bigr)
		}^2
		\\
		&\eqreff*{eq: discrete torque bound}{\lesssim}
		\tau\sum_{j=0}^{N-1}
		\left(
		1
		+
		\norm{
			\widehat{\mm}_{h\tau}^+|_{[t_j,t_{j+1})}
		}_{\bH^1(\Omega)}^2
		\right)
		\lesssim1.
	\end{align*}
	This concludes the proof.
\end{proof}

\subsection{Extraction of weakly convergent subsequences}\label{section:extraction}

Uniform boundedness~\eqref{eq: uniform bound} and \eqref{eq: uniform bound S-T norms} allow us to state the following proposition.

\begin{proposition}[Existence of weakly convergent subsequences]
	\label{prop: weak convergence extraction}
	Under the assumptions of Theorem~\ref{thm: main result}, there exists
	$\mm\in\bH^1(\OT)\cap L^\infty(0,T;\bH^1(\Omega))$ with
	$|\mm|=1$ a.e.\ in $\OT$ such that, up to the extraction of the same
	subsequence, for
	\[
	\mm_{h\tau}^\star
	\in
	\big\{
	\mm_{h\tau},
	\mm_{h\tau}^\pm,
	\mhatp
	\big\},
	\]
	the following convergences hold simultaneously as $(h,\tau)\to(0,0)$:
	\begin{subequations}\label{eq: conv}
		\begin{alignat}{3}
			\mm_{h\tau}
			&\rightharpoonup\mm
			&\quad&
			\text{in }\bH^1(\OT),
			\label{eq: conv H1}
			\\
			\mm_{h\tau}
			&\to\mm
			&&
			\text{in }\bH^s(\OT)
			\quad\text{for all }0<s<1,
			\label{eq: conv Hs}
			\\
			\mm_{h\tau}^\star
			&\to\mm
			&&
			\text{in }L^2(0,T;\bH^s(\Omega))
			\quad\text{for all }0<s<1,
			\label{eq: conv L2 0,T,Hs}
			\\
			\mm_{h\tau}^\star
			&\rightharpoonup\mm
			&&
			\text{in }L^2(0,T;\bH^1(\Omega)),
			\label{eq: conv L2 0,T,H1}
			\\
			\mm_{h\tau}^\star
			&\to\mm
			&&
			\text{in }\bL^2(\OT),
			\label{eq: conv L2}
			\\
			\mm_{h\tau}^\star
			&\to\mm
			&&
			\text{pointwise a.e.\ in }\OT,
			\label{eq: conv pointwise}
			\\
			\mm_{h\tau}^\star
			&\overset{*}{\rightharpoonup}\mm
			&&
			\text{in }L^\infty(0,T;\bH^1(\Omega)),
			\label{eq: conv L infty}
			\\
			\vv_{h\tau}^-
			&\rightharpoonup\partial_t\mm
			&&
			\text{in }\bL^2(\OT).
			\label{eq: conv L2 derivatives}
			\\
			\bpi_h(\mhatp)
			&\to\bpi(\mm)
			&&
			\text{in }\bL^2(\OT),
			\label{eq: conv pi}
			\\
			\calQ_{h\tau}^+(\mhatp)
			&\rightharpoonup\calQ(\mm)
			&&
			\text{in }\bL^2(\OT).
			\label{eq: conv Q}
		\end{alignat}
	\end{subequations}
\end{proposition}

\begin{proof}
	The proof is split into four steps. The construction of a subsequence
	satisfying all the stated limits follows by successive extraction.
	
	\textbf{Step~1 (Proof of \eqref{eq: conv H1}--\eqref{eq: conv L infty}).}
	Let
	\(
	\mm_{h\tau}^\star
	\in
	\big\{
	\mm_{h\tau},
	\mm_{h\tau}^\pm,
	\mhatp
	\big\}.
	\)
	Thanks to~\eqref{eq: uniform bound S-T norms}, the Eberlein--Šmulian
	theorem, the Banach--Alaoglu theorem, successive extraction of
	subsequences, and the compact embedding $\bH^1(\OT)\Subset\bH^s(\OT)$ for $0<s<1$, we obtain~\eqref{eq: conv H1} and~\eqref{eq: conv Hs} for
	$\mm_{h\tau}$ and, in particular, $\mm_{h\tau}\to\mm$ in $\bL^2(\OT)$.
	Moreover, for each $\mm_{h\tau}^\star$, the uniform $L^\infty(0,T;\bH^1(\Omega))$-bound from \eqref{eq: uniform bound S-T norms} yields, after further extraction, weak convergence in $L^2(0,T;\bH^1(\Omega))$ and weak-star convergence in $L^\infty(0,T;\bH^1(\Omega))$, initially with possibly different limits.
	
	It remains to identify these limits. From the definitions of the time
	reconstructions,
	\begin{align*}
		\norm{\mm_{h\tau}-\mm_{h\tau}^\pm}_{\bL^2(\OT)}^2
		=
		\frac{\tau^3}{3}
		\sum_{j=0}^{N-1}
		\normO{\d_t\mm_h^{j+1}}^2
		\eqreff{eq: norm equivalence 2}{\lesssim}
		\tau^3 \sum_{j=0}^{N-1} \normO{\vv_h^j}^2
		\eqreff{eq: uniform bound}{\lesssim}
		\tau^2,
	\end{align*}
	Furthermore, by the definition of $\mhatp$, its restriction to $[t_0,t_1)$
	coincides with $\mm_{h\tau}^-$, while for
	$t\in[t_j,t_{j+1})$ with $j\ge1$, we have $\mhatp(t)-\mm_{h\tau}^-(t)= \mm_h^j-\mm_h^{j-1} = \tau\d_t\mm_h^j$. Hence similarly,
	\begin{align*}
		\norm{\mhatp-\mm_{h\tau}^-}_{\bL^2(\OT)}^2
		&=
		\tau^3
		\sum_{j=1}^{N-1}
		\normO{\d_t\mm_h^j}^2
		\eqreff{eq: norm equivalence 2}{\lesssim}
		\tau^3 \sum_{j=0}^{N-1} \normO{\vv_h^j}^2
		\eqreff{eq: uniform bound}{\lesssim}
		\tau^2.
	\end{align*}
	Consequently, we obtain $\mm_{h\tau}^\star \to \mm$ in $\bL^2(\OT)$.
	Since the embedding $L^2(0,T;\bH^1(\Omega)) \hookrightarrow \bL^2(\OT)$ is continuous, the uniqueness of the weak $\bL^2(\OT)$-limit identifies all previously extracted weak and weak-star limits with $\mm$. This proves~\eqref{eq: conv L2 0,T,H1}, \eqref{eq: conv L2}, and~\eqref{eq: conv L infty}.
	
	Finally, for every $0<s<1$, the interpolation inequality gives
	\begin{align*}
		\norm{\mm_{h\tau}^\star-\mm}_{L^2(0,T;\bH^s(\Omega))}
		&\lesssim
		\norm{\mm_{h\tau}^\star-\mm}_{\bL^2(\OT)}^{1-s}
		\norm{\mm_{h\tau}^\star-\mm}_{L^2(0,T;\bH^1(\Omega))}^{s}
		\\
		&\lesssim
		\norm{\mm_{h\tau}^\star-\mm}_{\bL^2(\OT)}^{1-s}
		\longrightarrow0,
	\end{align*}
	which proves~\eqref{eq: conv L2 0,T,Hs}. Upon a further extraction,~\eqref{eq: conv pointwise} follows from the strong $\bL^2(\OT)$-convergence.

	\textbf{Step~2 (Proof of~\eqref{eq: conv L2 derivatives}).}
	By~\eqref{eq: uniform bound}, we may extract a further subsequence such that
	\[
	\vv_{h\tau}^-\rightharpoonup\vv
	\qquad\text{in }\bL^2(\OT).
	\]
	Since $\mm_{h\tau}\rightharpoonup\mm$ in $\bH^1(\OT)$, we infer that
	$\partial_t\mm_{h\tau} \rightharpoonup \partial_t\mm$ in $\bL^2(\OT)$.
	It remains to identify $\vv=\partial_t\mm$.
	
	For $j\ge1$, the BDF2 update gives
	\begin{equation}\label{eq: v minus dtm identity}
		\vv_h^j-\d_t\mm_h^{j+1}
		=
		\frac12
		\left(
		\d_t\mm_h^{j+1}
		-
		\d_t\mm_h^j
		\right).
	\end{equation}
	Moreover, $\vv_h^0=\d_t\mm_h^1$.
	Let $\pphi\in\boldsymbol{C}_c^\infty(\OT)$. Using
	\eqref{eq: v minus dtm identity}, we obtain
	\begin{align*}
		&\int_0^T
		\scalarproductO{
			\vv_{h\tau}^-(t)
			-
			\partial_t\mm_{h\tau}(t)
		}{
			\pphi(t)
		}\d t
		\\
		&\qquad=
		\frac12
		\int_{t_1}^T
		\scalarproductO{
			\partial_t\mm_{h\tau}(t)
		}{
			\pphi(t)
		}\d t
		-
		\frac12
		\int_{t_1}^T
		\scalarproductO{
			\partial_t\mm_{h\tau}(t-\tau)
		}{
			\pphi(t)
		}\d t
		\\
		&\qquad=
		\frac12
		\int_{t_1}^T
		\scalarproductO{
			\partial_t\mm_{h\tau}(t)
		}{
			\pphi(t)
		}\d t
		-
		\frac12
		\int_0^{T-\tau}
		\scalarproductO{
			\partial_t\mm_{h\tau}(t)
		}{
			\pphi(t+\tau)
		}\d t.
	\end{align*}
	Note that
	$\partial_t\mm_{h\tau}\rightharpoonup\partial_t\mm$ in
	$\bL^2(\OT)$ and
	$\pphi(\cdot+\tau)\to\pphi$ strongly in $\bL^2(\OT)$, while the
	endpoint contributions vanish because $\pphi$ is compactly supported
	in $\OT$. It follows that
	\[
	\int_0^T
	\scalarproductO{
		\vv_{h\tau}^-(t)
		-
		\partial_t\mm_{h\tau}(t)
	}{
		\pphi(t)
	}\d t
	\longrightarrow0.
	\]
	We conclude that
	\[
	\int_0^T
	\scalarproductO{
		\vv-\partial_t\mm
	}{
		\pphi
	}\d t
	=
	0
	\qquad
	\text{for all }
	\pphi\in\boldsymbol{C}_c^\infty(\OT),
	\]
	which proves $\vv=\partial_t\mm$ a.e.\ in $\OT$, and hence
	\eqref{eq: conv L2 derivatives}.
	
	\textbf{Step~3 (Proof of $|\mm|=1$ a.e.\ in $\OT$).}
	By Hölder's inequality,~\eqref{eq: uniform bound}, and
	\eqref{eq: conv L2},
	\begin{align*}
		\norm{
			|\mm|^2-|\mm_{h\tau}^-|^2
		}_{L^1(\OT)}
		&\le
		\normOT{\mm+\mm_{h\tau}^-}
		\normOT{\mm-\mm_{h\tau}^-}
		\eqreff{eq: uniform bound S-T norms}{\lesssim}
		\normOT{\mm-\mm_{h\tau}^-}
		\longrightarrow 0.
	\end{align*}
	Moreover, Proposition~\ref{prop: bounds} gives
	
		\begin{align*}
			\norm{
				|\mm_{h\tau}^-|^2-1
			}_{L^1(\OT)}
			&=
			\tau\sum_{j=0}^{N-1}
			\norm{
				|\mm_h^j|^2-1
			}_{L^1(\Omega)}
			\lesssim
			\tau+h^2
			\longrightarrow0.
		\end{align*}
	
	Therefore, by the triangle inequality, we infer that
	\[
	\norm{
		|\mm|^2-1
	}_{L^1(\OT)}
	\longrightarrow 0,
	\]
	thus implying $|\mm|=1$ a.e.\ in $\OT$.
	
	\textbf{Step~4 (Proof of
		\eqref{eq: conv pi}--\eqref{eq: conv Q}).}
	By~\eqref{eq: conv L2}, after extracting a further subsequence if
	necessary, we have $\mhatp(t)\to\mm(t)$ in $\bL^2(\Omega)$
	for a.e.\ $t\in(0,T)$. Hence, assumption~\eqref{enum: punto B3}
	implies
	\[
	\bpi_h(\mhatp(t))
	\to
	\bpi(\mm(t))
	\qquad\text{in }\bL^2(\Omega)
	\]
	for a.e.\ $t\in(0,T)$. Moreover, by the uniform boundedness of
	$\bpi_h$ and the boundedness of $\bpi$,
	\begin{align*}
		\normO{
			\bpi_h(\mhatp(t))-\bpi(\mm(t))
		}^2
		&\le
		2\normO{\bpi_h(\mhatp(t))}^2
		+
		2\normO{\bpi(\mm(t))}^2
		\\
		&\lesssim
		\normO{\mhatp(t)}^2
		+
		\normO{\mm(t)}^2
		\lesssim 1
	\end{align*}
	for a.e.\ $t\in(0,T)$, where the last bound follows from
	\eqref{eq: uniform bound S-T norms} and
	\eqref{eq: conv L infty}. Since $T<\infty$, the right-hand side is
	integrable over $(0,T)$. The dominated convergence theorem then
	yields
	\[
	\norm{
		\bpi_h(\mhatp)-\bpi(\mm)
	}_{\bL^2(\OT)}^2
	=
	\int_0^T
	\normO{
		\bpi_h(\mhatp(t))-\bpi(\mm(t))
	}^2
	\d t
	\longrightarrow0,
	\]
	which proves~\eqref{eq: conv pi}.
	
	Finally, $\mhatp\to\mm$ strongly in $\bL^2(\OT)$ and
	$\mhatp\rightharpoonup\mm$ weakly in
	$L^2(0,T;\bH^1(\Omega))$. Hence, the sequential consistency
	assumption~\eqref{enum: punto B4} directly gives~\eqref{eq: conv Q}. 
\end{proof}

\subsection{Identification of the limit with a weak solution of LLG}\label{section:identification}

In this section, we conclude the proof of Theorem~\ref{thm: main result}
by showing that the limit function $\mm$ from
Proposition~\ref{prop: weak convergence extraction} is a weak solution
of LLG~\eqref{eq: LLG system} in the sense of
Definition~\ref{def: weak solution}.

\begin{proof}[{\bfseries Proof of Theorem~\ref{thm: main result} (Limit $\mm$ satisfies Definition~{\ref{def: weak solution}\bf(\ref{def: weak solution i})--(\ref{def: weak solution ii})})}]
	Proposition~\ref{prop: weak convergence extraction} proves
	$\mm \in \bH^1(\OT) \cap L^{\infty}(0, T ; \bH^1(\Omega))$
	with $|\mm|=1$ a.e.\ in $\OT$.
	Thus, Definition~\ref{def: weak solution}\eqref{def: weak solution i}
	is verified.
	Since $\mm_{h \tau} \rightharpoonup \mm$ in $\bH^1(\OT)$, it follows
	$\mm_{h \tau}(0) \rightharpoonup \mm(0)$ in $\bH^{1 / 2}(\Omega)$.
	By assumption~\eqref{enum: punto B1},
	it holds that
	$\mm_{h \tau}(0)=\mm_h^0 \to \mm^0$ strongly in
	$\bH^1(\Omega)$. It follows that $\mm(0)=\mm^0$ in the sense of traces,
	which verifies
	Definition~\ref{def: weak solution}\eqref{def: weak solution ii}.
\end{proof}

\def\zzeta{\boldsymbol{\zeta}}
\begin{proof}[{\bfseries Proof of Theorem~\ref{thm: main result} (Limit $\mm$ satisfies Definition~{\ref{def: weak solution}\bf(\ref{def: weak solution iii})})}]
	The proof is split into four steps, where the core argument is summarized
	in Step~1.
	
	\textbf{Step~1 (Overview: discrete variational formulation).}
	Let $\vvarphi \in \boldsymbol{C}^{\infty}(\overline{\OT})$ be an
	arbitrary test function.
	Define $\zzeta_h\colon[0,T]\rightarrow\SSS^1(\TT_h)$ by
	\begin{equation}\label{eq: zeta_h}
		\zzeta_h(t)
		\coloneqq
		\III_h[\widehat{\mm}_{h\tau}^{+}(t)\times\vvarphi(t)]
		\quad\text{for all }t\in[0,T].
	\end{equation}
	For all $j=0,\dots,N-1$, all $t\in[t_j,t_{j+1})$, and all
	$\zz\in\NN_h$, we have
	(after defining $\widehat{\mm}_h^1\coloneqq\mm_h^0$ for a uniform
	notation) that
	$\zzeta_h(t)(\zz)\cdot\widehat{\mm}_h^{j+1}(\zz)=0$,
	which implies
	$\zzeta_h(t)\in\bT_h(\widehat{\mm}_h^{j+1})$.
	
	Thus, by using $\mm_h^1=\mm_h^0+\tau\vv_h^0$ for $j=0$ and the BDF2
	update for $j\ge1$, Algorithm~\ref{alg: full discr} gives, for all
	$t\in[t_j,t_{j+1})$,
	
		\begin{align}\label{eq1:dpr}
			\alpha\scalarproductO{\vv_h^j}{\zzeta_h(t)}
			&+
			\scalarproductO{
				\widehat{\mm}_h^{j+1}\times\vv_h^j
			}{
				\zzeta_h(t)
			}
			+
			\lamex^2
			\scalarproductO{
				\nabla\mm_h^{j+1}
			}{
				\nabla\zzeta_h(t)
			}
			+\mathbf{1}_{\{j=0\}}\tau
			\HDM(\vv_h^0,\zzeta_h(t))
			\nonumber\\
			&=
			-\HDM(
			\widehat{\mm}_h^{j+1},
			\zzeta_h(t)
			)
			+
			\scalarproductO{
				\bpi_h(\widehat{\mm}_h^{j+1})
				+
				\ff_h^{j+1}
				+
				\calQ_h^{j+1}(\widehat{\mm}_h^{j+1})
			}{
				\zzeta_h(t)
			}.
		\end{align}
	
	Here, $\mathbf{1}_{\{j=0\}}$ equals $1$ for $j=0$
	and $0$ otherwise, so that the additional term accounts
	for the implicit DMI contribution in the first time step.

	We integrate~\eqref{eq1:dpr} over $[t_j,t_{j+1})$ and sum these
	integrals over $j=0,\dots,N-1$.
	Recalling the definitions~\eqref{eq:mhtau-plus}--\eqref{eq:vhtau-minus}
	of the interpolants $\mm_{h\tau}^+$,
	$\widehat{\mm}_{h\tau}^+$, and $\vv_{h\tau}^-$, together with
	\eqref{eq: interpolation Q}, we obtain
	
		\begin{equation}\label{eq: discrete weak formulation limit}
			\begin{aligned}
				\alpha\int_0^T
				&\scalarproductO{\vv_{h\tau}^-(t)}{\zzeta_h(t)}\d t
				+
				\int_0^T
				\scalarproductO{
					\widehat{\mm}_{h\tau}^+(t)\times\vv_{h\tau}^-(t)
				}{
					\zzeta_h(t)
				}\d t
				+
				\lamex^2\int_0^T
				\scalarproductO{
					\nabla\mm_{h\tau}^+(t)
				}{
					\nabla\zzeta_h(t)
				}\d t
				\\
				&\quad=
				-\int_0^T
				\HDM(
				\widehat{\mm}_{h\tau}^+(t),
				\zzeta_h(t)
				)\d t
				-\tau\int_0^{t_1}
				\HDM(\vv_h^0,\zzeta_h(t))\d t
				+
				\int_0^T
				\scalarproductO{
					\bpi_h(\widehat{\mm}_{h\tau}^+(t))
				}{
					\zzeta_h(t)
				}\d t
				\\
				&\qquad
				+
				\int_0^T
				\scalarproductO{
					\ff_{h\tau}^+(t)
				}{
					\zzeta_h(t)
				}\d t
				+
				\int_0^T
				\scalarproductO{
					\calQ_{h\tau}^+(\widehat{\mm}_{h\tau}^+)(t)
				}{
					\zzeta_h(t)
				}\d t.
			\end{aligned}
		\end{equation}
	
	The subsequent steps consider the convergence of these terms separately. We will prove that the limit equation reads
	
		\begin{equation}\label{eq3:dpr}
			\begin{aligned}
				-\alpha\int_0^T
				&\scalarproductO{
					\mm(t)\times\partial_t\mm(t)
				}{
					\vvarphi(t)
				}\d t
				+
				\int_0^T
				\scalarproductO{
					\partial_t\mm(t)
				}{
					\vvarphi(t)
				}\d t
				\\
				&=
				\lamex^2
				\int_0^T
				\scalarproductO{
					\mm(t)\times\nabla\mm(t)
				}{
					\nabla\vvarphi(t)
				}\d t
				-
				\int_0^T
				\HDM(
				\mm(t),
				\mm(t)\times\vvarphi(t)
				)\d t
				\\
				&\quad
				-
				\int_0^T
				\scalarproductO{
					\mm(t)\times\bpi(\mm(t))
				}{
					\vvarphi(t)
				}\d t
				-
				\int_0^T
				\scalarproductO{
					\mm(t)\times\ff(t)
				}{
					\vvarphi(t)
				}\d t
				\\
				&\quad
				-
				\int_0^T
				\scalarproductO{
					\mm(t)\times\calQ(\mm(t))
				}{
					\vvarphi(t)
				}\d t,
			\end{aligned}
		\end{equation}
	
	thus verifying
	Definition~\ref{def: weak solution}\eqref{def: weak solution iii}.
	
	\textbf{Step~2 (Auxiliary convergence result).}
	We record some relevant convergence result for our analysis later.
	First, note that for all $\pphi_h\in\SSS^1(\TT_h)$ and all
	$\vvarphi\in\boldsymbol{C}^{\infty}(\overline{\Omega})$, it holds that
	\begin{equation}\label{eq: Lemma 2}
		\sum_{K\in\TT_h}
		\norm{
			D^2(\pphi_h\times\vvarphi)
		}_{\bL^2(K)}^2
		\lesssim
		\norm{\pphi_h}_{\bH^1(\Omega)}^2
		\norm{\vvarphi}_{\bW^{2,\infty}(\Omega)}^2.
	\end{equation}
	Therefore, we have
	\begin{align*}
		\norm[\big]{
			(\III_h-\mathbf{Id})
			[
			\widehat{\mm}_{h\tau}^+(t)\times\vvarphi(t)
			]
		}_{\bH^1(\Omega)}^2
		&=
		\sum_{K\in\TT_h}
		\norm[\big]{
			(\III_h-\mathbf{Id})
			[
			\widehat{\mm}_{h\tau}^+(t)\times\vvarphi(t)
			]
		}_{\bH^1(K)}^2
		\\
		&\eqreff*{eq: Lemma 1}{\lesssim}
		h^2
		\sum_{K\in\TT_h}
		\norm{
			D^2(
			\widehat{\mm}_{h\tau}^+(t)\times\vvarphi(t)
			)
		}_{\bL^2(K)}^2
		\\
		&\eqreff*{eq: Lemma 2}{\lesssim}
		h^2
		\norm{
			\widehat{\mm}_{h\tau}^+(t)
		}_{\bH^1(\Omega)}^2
		\norm{
			\vvarphi(t)
		}_{\bW^{2,\infty}(\Omega)}^2.
	\end{align*}
	In particular, it follows that
	\begin{subequations}\label{eq: bound (I-Ih)}
		\begin{equation}\label{eq: bound (I-Ih) primo}
			\norm{
				(\III_h-\mathbf{Id})
				[
				\widehat{\mm}_{h\tau}^+\times\vvarphi
				]
			}_{L^2(0,T;\bH^1(\Omega))}
			\lesssim
			h
			\norm{
				\widehat{\mm}_{h\tau}^+
			}_{L^2(0,T;\bH^1(\Omega))}
			\norm{\vvarphi}_{\bW^{2,\infty}(\OT)}
			\eqreff*{eq: uniform bound S-T norms}{\lesssim}
			h
			\to0.
		\end{equation}
		Furthermore, note that $\widehat{\mm}_{h\tau}^+\to\mm$ in
		$L^2(0,T;\bH^s(\Omega))$ yields $\widehat{\mm}_{h\tau}^+\times\vvarphi
		\to \mm\times\vvarphi$ in $L^2(0,T;\bH^s(\Omega))$ for $0<s<1$, with weak convergence for $s=1$. In particular, for $\zzeta_h \coloneq \III_h[
		\widehat{\mm}_{h\tau}^+\times\vvarphi]$, we obtain
		\begin{align}
			\zzeta_h
			&\xrightarrow{(h,\tau)\to(0,0)}
			\mm\times\vvarphi
			\quad
			\text{in }L^2(0,T;\bH^s(\Omega)),
			\label{eq: conv zeta}
			\\
			\zzeta_h
			&\xrightharpoonup{(h,\tau)\to(0,0)}
			\mm\times\vvarphi
			\quad
			\text{in }L^2(0,T;\bH^1(\Omega)).
			\label{eq: conv zeta H1}
		\end{align}
	\end{subequations}
	
	\textbf{Step~3 (Convergence of first, second, third, and external-field
		integrals in~\eqref{eq: discrete weak formulation limit}--\eqref{eq3:dpr}).}
	As before, the details of this step are contained in the proof of
	\cite[Theorem~2.2]{part1}. By the weak-strong convergence principle, together with the convergence properties
	\eqref{eq: bound (I-Ih)} of $\zzeta_h$, the convergences from
	Proposition~\ref{prop: weak convergence extraction}, we obtain
	\begin{alignat}{2}
		\int_0^T
		\scalarproductO{
			\vv_{h\tau}^-(t)
		}{
			\zzeta_h(t)
		}\d t
		&\longrightarrow
		-
		\int_0^T
		\scalarproductO{
			\mm(t)\times\partial_t\mm(t)
		}{
			\vvarphi(t)
		}\d t,
		&&\label{eq: conv first term}
		\\
		\int_0^T
		\scalarproductO{
			\widehat{\mm}_{h\tau}^{+}(t)\times\vv_{h\tau}^-(t)
		}{
			\zzeta_h(t)
		}\d t
		&\longrightarrow
		\int_0^T
		\scalarproductO{
			\partial_t\mm(t)
		}{
			\vvarphi(t)
		}\d t,
		&&\label{eq: conv second term}
		\\
		\int_0^T
		\scalarproductO{
			\nabla\mm_{h\tau}^+(t)
		}{
			\nabla\zzeta_h(t)
		}\d t
		&\longrightarrow
		-
		\int_0^T
		\scalarproductO{
			\mm(t)\times\nabla\mm(t)
		}{
			\nabla\vvarphi(t)
		}\d t,
		&&\label{eq: conv third term}
		\\
		\int_0^T
		\scalarproductO{
			\ff_{h\tau}^+(t)
		}{
			\zzeta_h(t)
		}\d t
		&\longrightarrow
		-
		\int_0^T
		\scalarproductO{
			\mm(t)\times\ff(t)
		}{
			\vvarphi(t)
		}\d t.
		&&\label{eq: conv f term}
	\end{alignat}

		\textbf{Step~4 (Convergence of the DMI and remaining lower-order
			integrals in~(\ref{eq: discrete weak formulation limit})--(\ref{eq3:dpr})).}
		First, using~\eqref{eq: conv L2 0,T,H1},
		\eqref{eq: conv L2}, \eqref{eq: conv zeta}, and
		\eqref{eq: conv zeta H1}, we have
		\begin{align}
			\int_0^T
			\HDM(
			\widehat{\mm}_{h\tau}^+(t),
			\zzeta_h(t)
			)\d t
			\xrightarrow{(h,\tau)\to(0,0)}
			\int_0^T
			\HDM(
			\mm(t),
			\mm(t)\times\vvarphi(t)
			)\d t.
			\label{eq: conv DMI term}
		\end{align}
		Indeed, by the definition~\eqref{eq: HDM}, the first contribution
		contains
		$\nabla\times\widehat{\mm}_{h\tau}^+
		\rightharpoonup\nabla\times\mm$
		weakly in $\bL^2(\OT)$ paired with
		$\zzeta_h\to\mm\times\vvarphi$ strongly in $\bL^2(\OT)$,
		while the second contribution contains
		$\widehat{\mm}_{h\tau}^+\to\mm$ strongly in $\bL^2(\OT)$ paired with
		$\nabla\times\zzeta_h \rightharpoonup \nabla\times(\mm\times\vvarphi)$ in $\bL^2(\OT)$.

		On the first time interval, the definition of $\zzeta_h$,
		the interpolation estimates~\eqref{eq: Lemma 1} and~\eqref{eq: Lemma 2},
		and assumption~\eqref{enum: punto B1} give
		\[
			\sup_{0\le t<t_1}
			\norm{\zzeta_h(t)}_{\bH^1(\Omega)}\lesssim1.
		\]
		Therefore, by~\eqref{eq: first step uniform bound},
		\begin{equation}\label{eq: conv first step DMI term}
			\tau\left|
				\int_0^{t_1}
				\HDM(\vv_h^0,\zzeta_h(t))\d t
			\right|
			\lesssim
			\tau^2\left(
				\normO{\vv_h^0}+\normO{\nabla\vv_h^0}
			\right)
			\lesssim\tau^{3/2}+\tau\longrightarrow0.
		\end{equation}
		
		Next, thanks to the strong convergence~\eqref{eq: conv pi} and the
		strong convergence~\eqref{eq: conv zeta} of $\zzeta_h$ in
		$\bL^2(\OT)$, we have
		\begin{align}
			\int_0^T
			\scalarproductO{
				\bpi_h(\widehat{\mm}_{h\tau}^+(t))
			}{
				\zzeta_h(t)
			}\d t
			&\xrightarrow{(h,\tau)\to(0,0)}
			\int_0^T
			\scalarproductO{
				\bpi(\mm(t))
			}{
				\mm(t)\times\vvarphi(t)
			}\d t
			\nonumber\\
			&=
			-
			\int_0^T
			\scalarproductO{
				\mm(t)\times\bpi(\mm(t))
			}{
				\vvarphi(t)
			}\d t.
			\label{eq: conv pi term}
		\end{align}
		Finally, by the weak convergence~\eqref{eq: conv Q} and again the
		strong convergence~\eqref{eq: conv zeta},
		\begin{align}
			\int_0^T
			\scalarproductO{
				\calQ_{h\tau}^+(\widehat{\mm}_{h\tau}^+)(t)
			}{
				\zzeta_h(t)
			}\d t
			&\xrightarrow{(h,\tau)\to(0,0)}
			\int_0^T
			\scalarproductO{
				\calQ(\mm(t))
			}{
				\mm(t)\times\vvarphi(t)
			}\d t
			\nonumber\\
			&=
			-
			\int_0^T
			\scalarproductO{
				\mm(t)\times\calQ(\mm(t))
			}{
				\vvarphi(t)
			}\d t.
			\label{eq: conv Q term}
		\end{align}

	As already noted in Step~1, the convergences
	\eqref{eq: conv first term}--\eqref{eq: conv Q term} prove that the
	discrete variational formulation~\eqref{eq: discrete weak formulation limit}
	tends to its continuous counterpart~\eqref{eq3:dpr}.
	This proves~\eqref{eq: weak formulation}
	for smooth test functions. A density argument therefore allows to
	extend the identity to every
	$\vvarphi\in\bH^1(\OT)$ and concludes the proof.
\end{proof}

\begin{proof}[{\bfseries Proof of Theorem~\ref{thm: main result}
		(Energy inequality)}]
	
		Assume, in addition, that assumption~\eqref{enum: punto B5} holds.
	
	We prove that $\mm$ satisfies the energy inequality in
	Definition~\ref{def: energy inequality}.
	
	For $j=0,\dots,N$, let $\ff^j\coloneqq\ff(t_j)$. In analogy with the
	time interpolants of the magnetization, we define the interpolants
	\[
	\ff_\tau(t)
	\coloneqq
	\frac{t-t_j}{\tau}\ff^{j+1}
	+
	\frac{t_{j+1}-t}{\tau}\ff^j,
	\qquad
	\ff_\tau^+(t)
	\coloneqq
	\ff^{j+1}
	\]
	for all $t\in[t_j,t_{j+1})$. Since
	$\ff\in C^1([0,T];\bL^2(\Omega))$, it holds that, as $\tau\to 0$,
	\begin{equation}\label{eq: convergence f time interpolants}
		\ff_\tau,\ff_\tau^+
		\to\ff
		\quad\text{and}\quad
		\partial_t\ff_\tau
		\to\partial_t\ff
		\qquad
		\text{in }\bL^2(\OT).
	\end{equation}
	
	\textbf{Step~1 (Discrete energy inequality).}
	For $n=1,\dots,N$, set
	\begin{align*}
		T_{\bpi}^n
		&\coloneqq
		\tau\sum_{j=0}^{n-1}
		\scalarproductO{
			\bpi_h(\widehat{\mm}_h^{j+1})
		}{
			\vv_h^j
		}
		+
		\EE_{\rm low}(\mm_h^n)
		-
		\EE_{\rm low}(\mm_h^0),
		\\
		T_{\ff}^n
		&\coloneqq
		\tau \sum_{j=0}^{n-1}
		\scalarproductO{
			\ff_{h}^{j+1}
		}{
			\vv_{h}^j
		}
		-
		\scalarproductO{\ff^n}{\mm_h^n}
		+
		\scalarproductO{\ff^0}{\mm_h^0},
		\\
		T_{\rm DM}^n
		&\coloneqq
		\frac12
		\left[
		\HDM(\mm_h^n,\mm_h^n)
		-
		\HDM(\mm_h^0,\mm_h^0)
		\right]
		-
		\tau\sum_{j=0}^{n-1}
		\HDM(\widehat{\mm}_h^{j+1},\vv_h^j)
		-
		\tau^2\HDM(\vv_h^0,\vv_h^0)
		,
	\end{align*}
	and define the following modified discrete energy, where $\lamex^2\mathcal{G}_h^n$ replaces the exchange energy contribution:
	\begin{equation}\label{eq: energy discrete modified energy}
		\EE_h^n
		\coloneqq
		\lamex^2\mathcal{G}_h^n
		+
		\EE_{\rm low}(\mm_h^n)
		-
		\scalarproductO{\ff^n}{\mm_h^n}
		+
		\frac12\HDM(\mm_h^n,\mm_h^n).
	\end{equation}
		
	Then~\eqref{eq: energy identity}, the definition of the energy,
	and omission of the non-negative second-difference term on the left-hand side imply,
	for every $1\le n\le N$,
\begin{equation}\label{eq: discrete energy inequality n}
	\begin{split}
		\EE_h^n
		-
		\EE(\mm_h^0,\ff^0)
		+
		\alpha\tau\sum_{j=0}^{n-1}\normO{\vv_h^j}^2
		\le{}&
		T_{\bpi}^n
		+
		T_{\ff}^n
		+
		T_{\rm DM}^n
		+
		\tau
		\sum_{j=0}^{n-1}
		\scalarproductO{
			\calQ_{h}^{j+1}(\widehat{\mm}_h^{j+1})
		}{
			\vv_{h}^j
		}
		+
		\frac{\lamex^2}{4}\eta_0 .
	\end{split}
\end{equation} Let $I\subseteq(0,T)$ be measurable.
For $t'\in[0,T)$, let $n=n(t')$ be the unique index with $t'\in[t_{n-1},t_{n})$.
Since $t'\le t_{n(t')}$ and the discrete dissipation is non-negative, we replace
$$
\alpha\tau\sum_{j=0}^{n-1}\normO{\vv_h^j}^2=\alpha\int_0^{t_{n(t')}}
\normO{\vv_{h\tau}^-}^2\d t
$$
by $\alpha\int_0^{t'}\normO{\vv_{h\tau}^-}^2\d t.$
This gives
\begin{equation}\label{eq: energy in progress}
	\begin{split}
		\EE_h^{n(t')}
		-
		\EE(\mm_{h\tau}^-(0),\ff^0)
		+
		\alpha
		\int_0^{t'}
		\normO{\vv_{h\tau}^-}^2\d t
		\le{}&
		T_{\bpi}^{n(t')}
		+
		T_{\ff}^{n(t')}
		+
		T_{\rm DM}^{n(t')}
		\\
		&+
		\int_0^{t_{n(t')}}
		\scalarproductO{
			\calQ_{h\tau}^+(\mhatp)
		}{
			\vv_{h\tau}^-
		}\d t
		+
		\frac{\lamex^2}{4}\eta_0 .
	\end{split}
\end{equation}

 \medskip
\textbf{Step~2 (Convergence of the energy and work terms).}
We first pass to the limit in the exchange contribution. For
$t'\in[t_{n-1},t_n)$ we have $\mm_{h\tau}^+(t')=\mm_h^{n}$ and
$\mm_{h\tau}^-(t')=\mm_h^{n-1}$, so that~\eqref{eq: Ghn} reads
\begin{equation}\label{eq: G as interpolants}
	\mathcal{G}_h^{n(t')}
	=
	\frac14
	\left[
	\normO{\nabla\mm_{h\tau}^+(t')}^2
	+
	\normO[\big]{
		2\nabla\mm_{h\tau}^+(t')-\nabla\mm_{h\tau}^-(t')
	}^2
	\right].
\end{equation}
By \eqref{eq: conv L2 0,T,H1}, both $\mm_{h\tau}^+$ and
$2\mm_{h\tau}^+-\mm_{h\tau}^-$ converge weakly to $\mm$ in
$L^2(0,T;\bH^1(\Omega))$, hence
$\nabla\mm_{h\tau}^+\rightharpoonup\nabla\mm$ and
$2\nabla\mm_{h\tau}^+-\nabla\mm_{h\tau}^-\rightharpoonup\nabla\mm$ in
$\bL^2(\OT)$. Applying weak lower semicontinuity to the two terms of
\eqref{eq: G as interpolants} and adding the resulting inequalities, we obtain
\begin{equation}\label{eq: energy G liminf}
	\frac12\int_I\normO{\nabla\mm(t')}^2\d t'
	\le
	\liminf_{(h,\tau)\to(0,0)}
	\int_I\mathcal{G}_h^{n(t')}\d t' .
\end{equation}
 
The remaining contributions to
\eqref{eq: energy discrete modified energy} converge.
Indeed, the continuity of $\EE_{\rm low}$ on bounded subsets
of $\bL^2(\Omega)$, together with the uniform bounds
\eqref{eq: uniform bound S-T norms} and the strong convergence
\eqref{eq: conv L2}, gives
\[
\int_I
\EE_{\rm low}(\mm_{h\tau}^+(t'))\d t'
\longrightarrow
\int_I
\EE_{\rm low}(\mm(t'))\d t';
\]
the weak--strong convergence principle applied to the two contributions in the
definition \eqref{eq: HDM} of $\HDM$, using
\eqref{eq: conv L2 0,T,H1} and \eqref{eq: conv L2}, gives
\[
\int_I
\HDM(\mm_{h\tau}^+(t'),\mm_{h\tau}^+(t'))\d t'
\longrightarrow
\int_I
\HDM(\mm(t'),\mm(t'))\d t' ;
\]
while \eqref{eq: convergence f time interpolants} and
\eqref{eq: conv L2} yield the convergence of the Zeeman contribution. Combining these three convergences with
\eqref{eq: energy G liminf} and recalling \eqref{eq: energy 2}, we conclude
\begin{equation}\label{eq: energy lsc limit}
	\int_I\EE(\mm(t'),\ff(t'))\d t'
	\le
	\liminf_{(h,\tau)\to(0,0)}
	\int_I\EE_h^{n(t')}\d t' .
\end{equation}
The strong convergence $\mm_h^0\to\mm^0$ in $\bH^1(\Omega)$ from
assumption~\eqref{enum: punto B1} likewise gives
\begin{equation}\label{eq: initial energy convergence}
	\EE(\mm_h^0,\ff^0)
	\longrightarrow
	\EE(\mm^0,\ff(0)) .
\end{equation}
 
	By~\eqref{eq: conv L2 derivatives}, weak lower semicontinuity gives,
	for every fixed $t'\in(0,T)$,
	\[
	\int_0^{t'}
	\normO{\partial_t\mm(t)}^2\d t
	\le
	\liminf_{(h,\tau)\to(0,0)}
	\int_0^{t'}
	\normO{\vv_{h\tau}^-(t)}^2\d t.
	\]
	Since the integrands are non-negative, Fatou's lemma yields
	\begin{align}
		\int_I
		\int_0^{t'}
		\normO{\partial_t\mm(t)}^2\d t\,\d t'
		&\le
		\int_I
		\liminf_{(h,\tau)\to(0,0)}
		\int_0^{t'}
		\normO{\vv_{h\tau}^-(t)}^2\d t\,\d t'
		\nonumber\\
		&\le
		\liminf_{(h,\tau)\to(0,0)}
		\int_I
		\int_0^{t'}
		\normO{\vv_{h\tau}^-(t)}^2\d t\,\d t'.
		\label{eq: dissipation integrated lsc}
	\end{align}
	
	We next consider $T_{\bpi}^n$. By~\eqref{eq: conv pi} and
	\eqref{eq: conv L2 derivatives}, the weak--strong convergence principle
	gives, for every fixed $t'\in(0,T)$,
	\[
	\int_0^{t'}
	\scalarproductO{
		\bpi_h(\mhatp(t))
	}{
		\vv_{h\tau}^-(t)
	}\d t
	\longrightarrow
	\int_0^{t'}
	\scalarproductO{
		\bpi(\mm(t))
	}{
		\partial_t\mm(t)
	}\d t.
	\]
	Notice that the replacement of $t_{n(t')}$ by $t'\le t_{n(t')}$ does not affect the limit after
	integration over $I$. Indeed, for any
	$\aaa_{h\tau},\bb_{h\tau}$ uniformly bounded in $\bL^2(\OT)$,
	Fubini's theorem and
	$0\le t_{n(t')}-t'\le\tau$ give
	\begin{equation}\label{eq: endpoint strip estimate}
		\begin{split}
			\int_I
			\left|
			\int_{t'}^{t_{n(t')}}
			\scalarproductO{
				\aaa_{h\tau}(t)
			}{
				\bb_{h\tau}(t)
			}\d t
			\right|\d t'
			&\le
			\tau
			\int_0^T
			\normO{\aaa_{h\tau}(t)}
			\normO{\bb_{h\tau}(t)}
			\d t
			\\
			&\le
			\tau
			\normOT{\aaa_{h\tau}}
			\normOT{\bb_{h\tau}}
			\lesssim\tau.
		\end{split}
	\end{equation}
	Moreover, for any family $\aaa_{h\tau}$ uniformly bounded in
	$\bL^2(\OT)$,
	\[
		\left|
			\int_0^{t'}
			\scalarproductO{\aaa_{h\tau}(t)}{\vv_{h\tau}^-(t)}
			\d t
		\right|
		\le
		\normOT{\aaa_{h\tau}}
		\normOT{\vv_{h\tau}^-}
		\lesssim1.
	\]
	Thus, convergence of the work primitives for fixed $t'$
	passes to the integral over $I$ by dominated convergence.

	Furthermore, the strong convergence $\mm_{h\tau}^+\to\mm$ in $\bL^2(\OT)$ and the continuity of $\EE_{\rm low}$ on bounded subsets of $\bL^2(\Omega)$ yield
	\[
	\EE_{\rm low}(\mm_{h\tau}^+)
	\to
	\EE_{\rm low}(\mm)
	\qquad
	\text{in }L^1(0,T).
	\]
	Consequently,
	\begin{align}
		\int_I T_{\bpi}^{n(t')}\d t'
		\longrightarrow
		\int_I
		\Bigg[
		&\int_0^{t'}
		\scalarproductO{
			\bpi(\mm(t))
		}{
			\partial_t\mm(t)
		}\d t
		+
		\EE_{\rm low}(\mm(t'))
		-
		\EE_{\rm low}(\mm^0)
		\Bigg]\d t'.
		\label{eq: limit T pi}
	\end{align}
	By the chain-rule property in assumption~\eqref{enum: punto A1}, it holds that
	\[
	\EE_{\rm low}(\mm(t'))
	-
	\EE_{\rm low}(\mm^0)
	=
	-
	\int_0^{t'}
	\scalarproductO{
		\bpi(\mm(t))
	}{
		\partial_t\mm(t)
	}\d t
	\]
	for all $t'\in[0,T]$.
	Hence, the integrand on the right-hand side of~\eqref{eq: limit T pi} vanishes, and so does the limit.

	We proceed analogously with the external field. By
	assumption~\eqref{enum: punto B2},
	$\ff_{h\tau}^+\to\ff$ strongly in $\bL^2(\OT)$, and hence
	weak--strong convergence,~\eqref{eq: endpoint strip estimate}, and
	\eqref{eq: conv L2} yield
	\begin{align}
		\int_I T_{\ff}^{n(t')}\d t'
		&\longrightarrow
		\int_I
		\Bigg[
		\int_0^{t'}
		\scalarproductO{
			\ff(t)
		}{
			\partial_t\mm(t)
		}\d t
		-
		\scalarproductO{
			\ff(t')
		}{
			\mm(t')
		}
		+
		\scalarproductO{
			\ff(0)
		}{
			\mm^0
		}
		\Bigg]\d t'
		\nonumber\\
		&\qquad=
		-\int_I
		\int_0^{t'}
		\scalarproductO{
			\partial_t\ff(t)
		}{
			\mm(t)
		}\d t\,\d t',
		\label{eq: limit T f}
	\end{align}
	where the last step follows since $\ff\in C^1([0,T];\bL^2(\Omega))$ and
	$\mm\in H^1(0,T;\bL^2(\Omega))$, and thus integration by parts in time gives
	\[
	\scalarproductO{\ff(t')}{\mm(t')}
	-
	\scalarproductO{\ff(0)}{\mm^0}
	=
	\int_0^{t'}
	\scalarproductO{
		\partial_t\ff(t)
	}{
		\mm(t)
	}\d t
	+
	\int_0^{t'}
	\scalarproductO{
		\ff(t)
	}{
		\partial_t\mm(t)
	}\d t.
	\]

	Finally, note that assumption~\eqref{enum: punto B5} gives $\calQ_{h\tau}^+(\mhatp)\to
	\calQ(\mm)$ in $\bL^2(\OT)$.
	Together with~\eqref{eq: conv L2 derivatives}, the weak--strong
	convergence principle and~\eqref{eq: endpoint strip estimate} give
	\begin{align}
		&\int_I
		\int_0^{t_{n(t')}}
		\scalarproductO{
			\calQ_{h\tau}^+(\mhatp)
		}{
			\vv_{h\tau}^-
		}\d t\,\d t'
		\longrightarrow
		\int_I
		\int_0^{t'}
		\scalarproductO{
			\calQ(\mm(t))
		}{
			\partial_t\mm(t)
		}\d t\,\d t'.
		\label{eq: limit torque work}
	\end{align}
	
	\textbf{Step~3 (DMI and BDF2 initial-step corrections).}
	For brevity, set $\mathcal{D}_h^j \coloneqq \HDM(\mm_h^j,\mm_h^j)$.
	For $j\ge1$, symmetry and bilinearity of $\HDM$, together with the
	BDF2 update, give
	\begin{align}
		\tau\HDM(
		\widehat{\mm}_h^{j+1},
		\vv_h^j
		)
		={}&
		\left(
		\frac34\mathcal{D}_h^{j+1}
		-
		\frac14\mathcal{D}_h^j
		\right)
		-
		\left(
		\frac34\mathcal{D}_h^j
		-
		\frac14\mathcal{D}_h^{j-1}
		\right)
		-
		\frac34
		\HDM\left(
		\tau^2 \d^2 _t \mm_h^{j+1},
		\tau^2 \d^2 _t \mm_h^{j+1}
		\right).
		\label{eq: DMI BDF2 energy identity}
	\end{align}
	At the first time step,
	\begin{equation}\label{eq: DMI first step}
	\tau\HDM(\mm_h^0,\vv_h^0)
	=
	\frac12
	(\mathcal{D}_h^1-\mathcal{D}_h^0)
	-
	\frac{\tau^2}{2}
	\HDM(\vv_h^0,\vv_h^0).
	\end{equation}
	Summing~\eqref{eq: DMI BDF2 energy identity} over
	$j=1,\dots,n-1$ for $n\ge2$, and using~\eqref{eq: DMI first step}
	in the definition of $T_{\rm DM}^n$ yields
	\begin{equation}\label{eq: DMI energy residual}
		\begin{split}
			T_{\rm DM}^n
			={}&
			\frac14
			(\mathcal{D}_h^{n-1}-\mathcal{D}_h^n)
			+
			\frac14
			(\mathcal{D}_h^1-\mathcal{D}_h^0)
			-
			\frac{\tau^2}{2}
			\HDM(\vv_h^0,\vv_h^0)
			+
			\frac34
			\sum_{j=1}^{n-1}
			\HDM\big(
			\tau^2 \d_t^2 \mm_h^{j+1},
			\tau^2 \d_t^2 \mm_h^{j+1}
			\big).
		\end{split}
	\end{equation}
	For $n=1$, the terms involving $\mathcal{D}_h^j$ in \eqref{eq: DMI energy residual} cancel and the sum is empty, so that
	\[
	T_{\rm DM}^1
	=
	-\frac{\tau^2}{2}
	\HDM(\vv_h^0,\vv_h^0).
	\]
	
	Using the inequality $|\HDM(\pphi,\pphi)| \leq \sqrt{2} |\lDM|
	\normO{\pphi} \normO{\nabla\pphi}$,
	the estimate~\eqref{eq: inverse estimate}, and
	Proposition~\ref{prop: bounds}, we obtain
	\begin{align}\label{eq: DMI sum hat hat}
		&\sum_{j=1}^{n-1}
		\left|
		\HDM\left(
		\tau^2\d_t^2\mm_h^{j+1},
		\tau^2\d_t^2\mm_h^{j+1}
		\right)
		\right|
		\nonumber\\
		&\qquad\mkern+9mu\le
		\sqrt{2}|\lDM|
		\tau^{1/2}
		\bigg(
		\tau^3
		\sum_{j=1}^{n-1}
		\normO{\d_t^2\mm_h^{j+1}}^2
		\bigg)^{1/2}
		\bigg(
		\tau^4
		\sum_{j=1}^{n-1}
		\normO{\nabla\d_t^2\mm_h^{j+1}}^2
		\bigg)^{1/2}
		\nonumber\\
		&\qquad
		\stackrel{\eqref{eq: inverse estimate}}{\lesssim}
		\tau^{1/2}
		\bigg(
		\tau
		\sum_{j=0}^{n-1}
		\normO{\vv_h^j}^2
		\bigg)^{1/2}
		\bigg(
		\tau^4
		\sum_{j=1}^{n-1}
		\normO{\nabla\d_t^2\mm_h^{j+1}}^2
		\bigg)^{1/2}
		\stackrel{\eqref{eq: uniform bound}}{\lesssim}
		\tau^{1/2}
		\longrightarrow0.
	\end{align}
	Similarly, the first-step bound gives
	\begin{align}\label{eq: DMI first step bound}
		\tau^2
		|\HDM(\vv_h^0,\vv_h^0)|
		&\lesssim
		\tau^2
		\normO{\vv_h^0}
		\normO{\nabla\vv_h^0}
		\lesssim
		\tau^{1/2}
		\longrightarrow0.
	\end{align}
	Furthermore, since $\mm_{h\tau}^\pm\to\mm$ strongly in $\bL^2(\OT)$ and
	$\mm_{h\tau}^\pm\rightharpoonup\mm$ weakly in
	$L^2(0,T;\bH^1(\Omega))$, the weak--strong convergence principle
	and the definition~\eqref{eq: HDM} give
	\[
	\int_I
	\HDM(\mm_{h\tau}^\pm,\mm_{h\tau}^\pm)\d t'
	\longrightarrow
	\int_I
	\HDM(\mm,\mm)\d t'.
	\]
	For $t'\in[t_{n-1},t_n)$, we have
	$\mm_{h\tau}^-(t')=\mm_h^{n-1}$ and
	$\mm_{h\tau}^+(t')=\mm_h^n$. Hence,
	\begin{align}\label{eq: terminal DMI difference}
		\int_I
		\left(
		\mathcal{D}_h^{n(t')-1}
		-
		\mathcal{D}_h^{n(t')}
		\right)\d t'
		=
		\int_I
		\left[
		\HDM(\mm_{h\tau}^-,\mm_{h\tau}^-)
		-
		\HDM(\mm_{h\tau}^+,\mm_{h\tau}^+)
		\right]\d t'
		\longrightarrow 0.
	\end{align}
	
	It remains to control the initial BDF2 exchange correction.
Combining~\eqref{eq: energy proof 0} with
\eqref{eq: DMI first step} and rearranging gives
\begin{equation}\label{eq: eta0 identity}
    \begin{split}
        &\Big[
            \frac{\lamex^2}{4}\eta_0
            +\frac14(\mathcal{D}_h^1-\mathcal{D}_h^0)
            -\frac{\tau^2}{2}\HDM(\vv_h^0,\vv_h^0)
        \Big]
        +\frac{\alpha\tau}{2}\normO{\vv_h^0}^2
        +\frac{\lamex^2\tau^2}{4}
            \normO{\nabla\vv_h^0}^2
        \\
        &\quad=
        \frac{\tau}{2}
        \scalarproductO{
            \bpi_h(\mm_h^0)+\ff_h^1+\calQ_h^1(\mm_h^0)
        }{\vv_h^0}
        -\frac{3\tau^2}{4}\HDM(\vv_h^0,\vv_h^0).
    \end{split}
\end{equation}
	The growth assumption~\eqref{eq: discrete torque bound} and
	assumptions~\eqref{enum: punto B1}--\eqref{enum: punto B3} yield
	\begin{align*}
		\tau
		\left|
		\scalarproductO{
			\bpi_h(\mm_h^0)
			+
			\ff_h^1
			+
			\calQ_h^1(\mm_h^0)
		}{
			\vv_h^0
		}
		\right|
		&\lesssim
		\tau^{1/2} \normO{
			\bpi_h(\mm_h^0)+\ff_h^1+\calQ_h^1(\mm_h^0)}
		\left(
		\tau\normO{\vv_h^0}^2
		\right)^{1/2}
		\\
		&\mkern-7mu\eqreff{enum: punto B3}{\lesssim}
		\tau^{1/2}
		\left(\norm{\mm_h^0}_{\bH^1(\Omega)}
		+
		\normO{\ff_h^1}
		+
		1\right)
		\left(
		\tau\normO{\vv_h^0}^2
		\right)^{1/2}
		\\
		&\mkern-6mu\eqreff{eq: first step uniform bound}{\lesssim}
		\tau^{1/2}
		\longrightarrow0.
	\end{align*}
	Dropping the two non-negative dissipation terms on the
	left-hand side of~\eqref{eq: eta0 identity},
	we infer from the above estimate and the
	DMI estimate \eqref{eq: DMI first step bound} that
	\begin{equation}\label{eq: limsup initial correction}
		\limsup_{(h,\tau)\to(0,0)}
		\Big[\frac{\lamex^2}{4}\eta_0
		+
		\frac14(\mathcal{D}_h^1-\mathcal{D}_h^0)
		-
		\frac{\tau^2}{2}
		\HDM(\vv_h^0,\vv_h^0)\Big]
		\le0.
	\end{equation}
	Moreover,
	\eqref{eq: DMI energy residual} yields
	\begin{align*}
		T_{\rm DM}^{n(t')}
		+
		\frac{\lamex^2}{4}\eta_0
		={}&
		\Big[
		\frac{\lamex^2}{4}\eta_0
		+
		\frac14(\mathcal{D}_h^1-\mathcal{D}_h^0)
		-
		\frac{\tau^2}{2}
		\HDM(\vv_h^0,\vv_h^0)
		\Big]
		\\
		&+
		\frac14
		\big(
		\mathcal{D}_h^{n(t')-1}
		-
		\mathcal{D}_h^{n(t')}
		\big)
		+
		\frac34
		\sum_{j=1}^{n(t')-1}
		\HDM\big(
		\tau^2\d_t^2\mm_h^{j+1},
		\tau^2\d_t^2\mm_h^{j+1}
		\big).
	\end{align*}
	and therefore,
	\eqref{eq: DMI sum hat hat},\eqref{eq: terminal DMI difference},
	and~\eqref{eq: limsup initial correction}
	imply
	\begin{equation}\label{eq: correction terms vanish}
	\limsup_{(h,\tau)\to(0,0)}\int_I
	\Big[T_{\rm DM}^{n(t')}+\frac{\lamex^2}{4}\eta_0\Big]\d t'
	\le0.
	\end{equation}

	\textbf{Step~4 (Continuous energy inequality).}
	Integrating~\eqref{eq: energy in progress} over $I$, and using~\eqref{eq: initial energy convergence},
	\eqref{eq: limit T pi}--\eqref{eq: limit torque work}, and
	\eqref{eq: correction terms vanish} in the last inequality, we obtain
	\begin{align*}
	&\int_I\Big[\EE(\mm(t'),\ff(t'))
	+\alpha\int_0^{t'}\normO{\partial_t\mm(t)}^2\d t\Big]\d t'
	\\
	&\quad\stackrel{\eqref{eq: energy lsc limit}
	,\eqref{eq: dissipation integrated lsc}}{\leq}\liminf_{(h,\tau)\to(0,0)}
	\int_I\Big[\EE_h^{n(t')}
	+\alpha\int_0^{t'}\normO{\vv_{h\tau}^-(t)}^2\d t\Big]\d t'
	\\
	&\mkern+18mu\quad\stackrel{\eqref{eq: energy in progress}}{\le}\limsup_{(h,\tau)\to(0,0)}\int_I\Big[
	\EE(\mm_{h\tau}^-(0),\ff^0)+T_{\bpi}^{n(t')}+T_{\ff}^{n(t')}+\int_0^{t_{n(t')}}
	\scalarproductO{\calQ_{h\tau}^+(\mhatp)}{\vv_{h\tau}^-}\d t
	\\
	&\hspace{35mm}
	+ T_{\rm DM}^{n(t')}+\frac{\lamex^2}{4}\eta_0 \Big]\d t'
	\\
	&\quad\mkern+27mu\le\int_I\Big[\EE(\mm^0,\ff(0))
	-\int_0^{t'}\scalarproductO{\partial_t\ff(t)}{\mm(t)}\d t
	+\int_0^{t'}\scalarproductO{\calQ(\mm(t))}{\partial_t\mm(t)}\d t
	\Big]\d t'.
	\end{align*}
	Since $I\subseteq(0,T)$ is arbitrary, we conclude that, for a.e.
	$t'\in(0,T)$,
	\[
	\begin{aligned}
	\EE(\mm(t'),\ff(t'))
	&+\alpha\int_0^{t'}\normO{\partial_t\mm(t)}^2\d t
	+\int_0^{t'}\scalarproductO{\partial_t\ff(t)}{\mm(t)}\d t
	\\
	&\qquad\le\EE(\mm^0,\ff(0))
	+\int_0^{t'}\scalarproductO{\calQ(\mm(t))}{\partial_t\mm(t)}\d t.
	\end{aligned}
	\]
	This is exactly the energy inequality in
	Definition~\ref{def: energy inequality}.
	\end{proof}

\begin{proof}[{\bfseries Proof of Corollary~\ref{cor: main result}
		(convergence of the full sequence)}]
	The proof of Theorem~\ref{thm: main result} shows that every sequence
	$(h,\tau)\to(0,0)$ admits a subsequence converging to a weak solution of
	LLG~\eqref{eq: LLG system}. If this weak solution is unique, then every
	such subsequential limit coincides with $\mm$. Hence every subsequence
	admits a further subsequence converging to $\mm$, which is equivalent
	to convergence of the full sequence.
\end{proof}
\begin{remark}[Unconditional energy inequality for the reduced model]\label{rem: part1 unconditional}
	For $\lDM=0$, $\bpi=\boldsymbol{0}$, and $\calQ=\boldsymbol{0}$, Algorithm~\ref{alg: full discr} coincides with the integrator of~\cite{part1}, and the proof of the energy inequality given above applies verbatim.
	In this case, the DMI and torque contributions vanish, and the only initial-step correction is $\eta_0$. Since~\eqref{eq: energy proof 0} gives
	\[
	\frac{\lamex^2}{2}\eta_0
	\le
	\tau\scalarproductO{\ff_h^1}{\vv_h^0}
	\le
	\big(\tau\normO{\ff_h^1}^2\big)^{1/2}\big(\tau\normO{\vv_h^0}^2\big)^{1/2}
	\]
	and $\tau\normO{\ff_h^1}^2=\int_0^{t_1}\normO{\ff_{h\tau}^+(t)}^2\d t\to0$ by the strong convergence $\ff_{h\tau}^+\to\ff$ in $\bL^2(\OT)$ assumed in~\cite{part1}, it follows that $\limsup_{(h,\tau)\to(0,0)}\eta_0\le0$.
	Consequently, the weak solution constructed in~\cite{part1} satisfies the energy inequality without the CFL-type condition $\tau=o(h^2)$ required there.
\end{remark}
\section{Numerical experiments}\label{section:numerics}

\subsection{Experimental convergence rates and stability of an isolated skyrmion}
We consider the experiment on the stability of isolated skyrmions
in~\cite[Section~4.2]{hpprss2019}, in
a reduced two-dimensional model on a square.
Compared with that experiment, we replace the disk by a square,
consider bulk DMI covered by our analysis instead of interfacial DMI,
and use a relaxed thin-film approximation of the three-dimensional energy.

\subsubsection{Reduced two-dimensional model and initial data}\label{section:numerics:model}
The original experiment concerns a 3D domain consisting in a
cylindrical nanodisk of diameter \(80\,\mathrm{nm}\) and thickness
\(d=0.4\,\mathrm{nm}\), i.e.,
\(\Omega_{\varepsilon}=\omega\times(0,\varepsilon)\),
where \(\omega\subset\mathbb R^2\)
denotes the disk of radius \(40\,\mathrm{nm}\) centered at the origin.
In the setting of \cite{hpprss2019}
(with bulk DMI substituting the interfacial one),
the magnetization \(\mm:\Omega_{\varepsilon}\to \S^2\)
is associated with the total energy
\[
\begin{aligned}
\mathcal E_{\rm 3D}(\mm,\boldsymbol H_{\rm ext})
={}&
\int_{\Omega_{\varepsilon}}
\Big[
A|\nabla\mm|^2
+D\,(\nabla\times\mm)\cdot\mm
+K(1-(\mm\cdot\ee_3)^2)
-\mu_0M_s\boldsymbol H_{\rm ext}\cdot\mm
\Big]\,\d\xx
\\
&\quad
+\frac{\mu_0M_s^2}{2}
\int_{\mathbb R^3}
|\boldsymbol h_{\mathbf{s}}(\mm)|^2\,\d\xx,
\end{aligned}
\]
where \(A>0\) is the exchange stiffness,
\(D\geq0\) the Dzyaloshinskii--Moriya interaction constant,
\(K>0\) the uniaxial anisotropy constant,
\(M_s>0\) the saturation magnetization, and
\(\mu_0=4\pi\times10^{-7}\,\mathrm{H\,m^{-1}}\)
the vacuum permeability.
Here, \(\boldsymbol H_{\rm ext}=\boldsymbol H_{\rm ext}(\xx,t)\)
is a prescribed physical external field, measured in
\(\mathrm{A\,m^{-1}}\).

Moreover, \(\boldsymbol h_{\mathbf{s}}(\mm)\) denotes the normalized
stray field generated by the extension of \(\mm\) by zero outside
\(\Omega_{\varepsilon}\); the corresponding physical field is
\(\boldsymbol H_s=M_s\boldsymbol h_{\mathbf{s}}\).
The corresponding effective field is obtained from the negative
first variation of the energy and reads as
\[
\begin{aligned}
\boldsymbol H_{\mathrm{eff}}^{\rm 3D}
(\mm,\boldsymbol H_{\rm ext})
&=
-\frac{1}{\mu_0M_s}
\frac{\delta\mathcal E_{\rm 3D}
}{\delta\mm}
=
\frac{2A}{\mu_0M_s}\Delta\mm
-\frac{2D}{\mu_0M_s}\boldsymbol\nabla\times\mm
+\frac{2K}{\mu_0M_s}m_3\boldsymbol e_3
+\boldsymbol H_s(\mm)
+\boldsymbol H_{\rm ext}.
\end{aligned}
\]
For \(\boldsymbol H_{\rm ext}=\boldsymbol0\), this coincides with
the effective field employed in the experiment
in~\cite{hpprss2019}, up to the change of the DMI contribution.
We choose the same material parameters, which mimic those of cobalt,
i.e.,
\[
A=1.5\times10^{-11}\,\mathrm{J\,m^{-1}},
\qquad
K=8\times10^5\,\mathrm{J\,m^{-3}},
\qquad
M_s=5.8\times10^5\,\mathrm{A\,m^{-1}},
\]
with Gilbert damping parameter \(\alpha=0.3\),
and DMI strengths
\(D\in\{0,1,\ldots,8\}\times10^{-3}\,\mathrm{J\,m^{-2}}\).

Since the thickness of the nanodisk is much smaller than its diameter,
we can consider a 2D thin-film approximation of the 3D problem,
which allows us to avoid a boundary element
treatment of the full-space magnetostatic problem.
For this reduction, we assume that the prescribed external field
is independent of the thickness coordinate and retain the notation
\(\boldsymbol H_{\rm ext}\) for the resulting field on the cross-section.

For the limiting magnetization, we use the 2D curl operator
\[
\operatorname{curl}\mm\coloneqq
\begin{pmatrix}
\partial_2m_3\\
-\partial_1m_3\\
\partial_1m_2-\partial_2m_1
\end{pmatrix}.
\]
We recall that
\cite{davoli2020} allows us to replace the stray-field energy by its
thin-film approximation as we pass to the thin-film approximation,
i.e.,
\[
\frac{\mu_0M_s^2}{2}
\int_{\R^3}|\boldsymbol h_{\rm s}(\mm)|^2\,\d\xx
\approx
\varepsilon\,\frac{\mu_0M_s^2}{2}
\int_{\omega}(\mm\cdot\ee_3)^2\,\d\xx.
\]
Following the thin-film limit approximation
of~\cite{davoli2020}, the dimension reduction
produces an anisotropy contribution induced by the bulk DMI.
Accordingly,
\(\mathcal E_{\rm 3D}(\mm,\boldsymbol H_{\rm ext})
\approx
\varepsilon\,\mathcal E_{\rm 2D}(\mm,\boldsymbol H_{\rm ext})\),
where the energy per unit thickness is
\[
\begin{aligned}
\mathcal E_{\rm 2D}(\mm,\boldsymbol H_{\rm ext})
\coloneqq
\int_{\omega}
\Bigl[
&A|\nabla\mm|^2
+D\,\mm\cdot\operatorname{curl}\mm
+K\bigl(1-(\mm\cdot\ee_3)^2\bigr)
\\
&+\frac{\mu_0M_s^2}{2}(\mm\cdot\ee_3)^2
+\frac{D^2}{4A}(\mm\cdot\ee_3)^2-\frac{D^2}{4A}
-\mu_0M_s\boldsymbol H_{\rm ext}\cdot\mm
\Bigr]\,\d\xx.
\end{aligned}
\]
For comparison, consider the two-dimensional energy
in~\cite[equation~(53)]{dFIPS2025}.
When both energies are expressed on the same cross-section domain
and in physical units, with the same prescribed external field,
the present model contains the additional contribution
\[
\frac{D^2}{4A}\int_\omega(m_3^2-1)\,\d\xx.
\]
This term accounts for the relaxation of the bulk DMI
in the thickness direction.

We note that at the two-dimensional level the rotation
\((m_1,m_2,m_3)\mapsto(m_2,-m_1,m_3)\) transforms
\(\mm\cdot\operatorname{curl}\mm\) into the interfacial DMI density
of~\cite{hpprss2019} and leaves all other local terms invariant,
provided that the external field is rotated simultaneously.
In particular, no change of the external field is needed when
\(\boldsymbol H_{\rm ext}=\boldsymbol0\), as chosen below.
Hence, for a fixed 2D energy with the same remaining
coefficients and with this corresponding choice of external field,
this rotation preserves the out-of-plane component \(m_3\)
and the energy and rotates the in-plane magnetization by \(90^\circ\),
thereby transforming N\'eel-type into Bloch-type textures;
cf.~\cite[Remark~19]{dFIPS2025}.

Notice that the stray-field contribution and the relaxation
of the bulk DMI in the thickness direction both modify the effective
anisotropy. Defining
\[
K_{\mathrm{eff}}
:=
K-\frac{\mu_0M_s^2}{2}-\frac{D^2}{4A},
\]
we have
\[
\begin{aligned}
&K(1-(\mm\cdot\ee_3)^2)
+\frac{\mu_0M_s^2}{2}(\mm\cdot\ee_3)^2
+\frac{D^2}{4A}(\mm\cdot\ee_3)^2-\frac{D^2}{4A}
=
K_{\mathrm{eff}}(1-m_3^2)+\frac{\mu_0M_s^2}{2}.
\end{aligned}
\]
Since the last term is constant, it does not contribute to the
effective field, which consequently reads as
\[
\boldsymbol H_{\mathrm{eff}}^{\rm 2D}
(\mm,\boldsymbol H_{\rm ext})
=
\frac{2A}{\mu_0M_s}\Delta\mm
-\frac{2D}{\mu_0M_s}\operatorname{curl}\mm
+\frac{2K_{\mathrm{eff}}}{\mu_0M_s}m_3\boldsymbol e_3
+\boldsymbol H_{\rm ext}.
\]

We now nondimensionalize the reduced physical model directly
on the unit square.
For the numerical experiments, the physical cross-section is
\(\omega_L=(0,L)^2\), with \(L=80\,\mathrm{nm}\).
The physical energy \(\mathcal E_{\rm 2D}\) derived above is
understood on this square.

We introduce the physical exchange length
\[
\ell_{\rm ex}:=\sqrt{\frac{2A}{\mu_0M_s^2}}
\approx8.42\,\mathrm{nm},
\]
and set
\[
\boldsymbol\xi=\frac{\xx}{L},
\qquad
t'=\gamma_0M_s t,
\qquad
\Omega=(0,1)^2,
\]
where
\(\gamma_0=2.21\times10^5\,\mathrm{m\,A^{-1}s^{-1}}\)
is the rescaled gyromagnetic ratio used in the physical
Gilbert equation; see~\cite[Section~2.1]{hpprss2019}.

In the notation of Section~\ref{section:model}, the
dimensionless exchange and DMI parameters are
\[
\lamex^2
:=\frac{2A}{\mu_0M_s^2L^2}
=\frac{\ell_{\rm ex}^2}{L^2},
\qquad
\lDM
:=\frac{2D}{\mu_0M_s^2L}.
\]
Thus, \(\ell_{\rm ex}\) is a physical length, whereas
\(\lamex=\ell_{\rm ex}/L\) is dimensionless.

Dividing the physical energy per unit thickness by
\(\mu_0M_s^2L^2\)
gives the dimensionless energy
\begin{equation*}\label{eq: reduced energy 2d}
\begin{aligned}
&E_{\rm 2D}(\mm,\ff)
\coloneqq
\frac{\mathcal E_{\rm 2D}}
{\mu_0M_s^2L^2}
\\
&\qquad=
\int_\Omega
\Big[
\frac{\lamex^2}{2}|\nabla\mm|^2
+\frac{\lDM}{2}\mm\cdot\operatorname{curl}\mm
+\frac{K}{\mu_0M_s^2}(1-m_3^2)
+\frac12m_3^2
+\frac{\lDM^2}{8\lamex^2}(m_3^2-1)
-\ff\cdot\mm
\Big]\,\d\boldsymbol\xi.
\end{aligned}
\end{equation*}
Here, the physical and dimensionless fields are evaluated at
corresponding times.
The terms \(m_3^2/2\) and
\(\lDM^2(m_3^2-1)/(8\lamex^2)\) arise, respectively,
from the local thin-film approximation of the stray field
and the relaxation of the bulk DMI in the thickness direction.
The term \(-\ff\cdot\mm\) is the normalized Zeeman energy density.

Setting
\[
c_{\rm ani}
\coloneqq
\frac{2K_{\rm eff}}{\mu_0M_s^2}
=
\frac{2K}{\mu_0M_s^2}
-1-\frac{D^2}{2A\mu_0M_s^2}
=
\frac{2K}{\mu_0M_s^2}
-1-\frac{\lDM^2}{4\lamex^2},
\]
the corresponding normalized effective field is
\[
\heff(\mm,\ff)
=
-\frac{\delta E_{\rm 2D}(\mm,\ff)}{\delta\mm}
=
\lamex^2\Delta\mm
-\lDM\operatorname{curl}\mm
+c_{\rm ani}m_3\ee_3
+\ff.
\]
The natural boundary condition is
\[
2\lamex^2\partial_{\boldsymbol n}\mm
+\lDM\,\mm\times\boldsymbol n=0
\qquad\text{on }\partial\Omega,
\]
in agreement with~\eqref{eq:subeq_b}.

With the chosen time scaling, the evolution equation
has precisely the Gilbert form in~\eqref{eq: LLG system}.
In the present experiments, we set
\(
\boldsymbol H_{\rm ext}=\boldsymbol0,
\)
and
\(
\calQ=\boldsymbol0.
\)
A physical time step \(\Delta t\) corresponds to the
dimensionless step
\(
\tau=\gamma_0M_s\Delta t
\)
used by the numerical scheme.

For the temporal convergence experiments below, we fix
\(D=3\times10^{-3}\,\mathrm{J\,m^{-2}}\).
For this choice and the selected material and sample size,
\[
\lamex^2\approx0.011089,
\qquad
\lDM\approx0.177417,
\qquad
c_{\rm ani}\approx2.075229.
\]
As in \cite[Section~4.2]{hpprss2019}, we investigate the evolution of the following two initial magnetization configurations:
\begin{enumerate}[label=(\roman*)]
\item a uniform magnetization pointing in the positive \(z\)-direction, $\mm^0=\ee_3$
\item a skyrmion-like initial state, prescribed by a reversed disk of radius
$R=15\,\mathrm{nm}$ centred in the centre of the square, defined at each mesh vertex
$\boldsymbol z\in\mathcal N_h$ by
\[
\mm_h^0(\boldsymbol z)
=
\begin{cases}
(0,0,-1),
&
\displaystyle
\big|\boldsymbol z-\big(1/2,1/2\big)\big|
\leq R/L,
\\[1ex]
(0,0,1),
&
\displaystyle
\big|\boldsymbol z-\big(1/2,1/2\big)\big|
> R/L
\end{cases}
\qquad \frac{R}{L} = \frac{3}{16}.
\]
The initial finite element magnetization is the continuous
piecewise affine interpolant of these nodal values. Notice that 
the initial condition has an abrupt reversal rather than a smooth
rotation across a wall of fixed physical width.
\end{enumerate}
\begin{figure}[ht]
  \centering
  \begin{subfigure}[b]{0.44\textwidth}
    \centering
    \includegraphics[width=\linewidth]{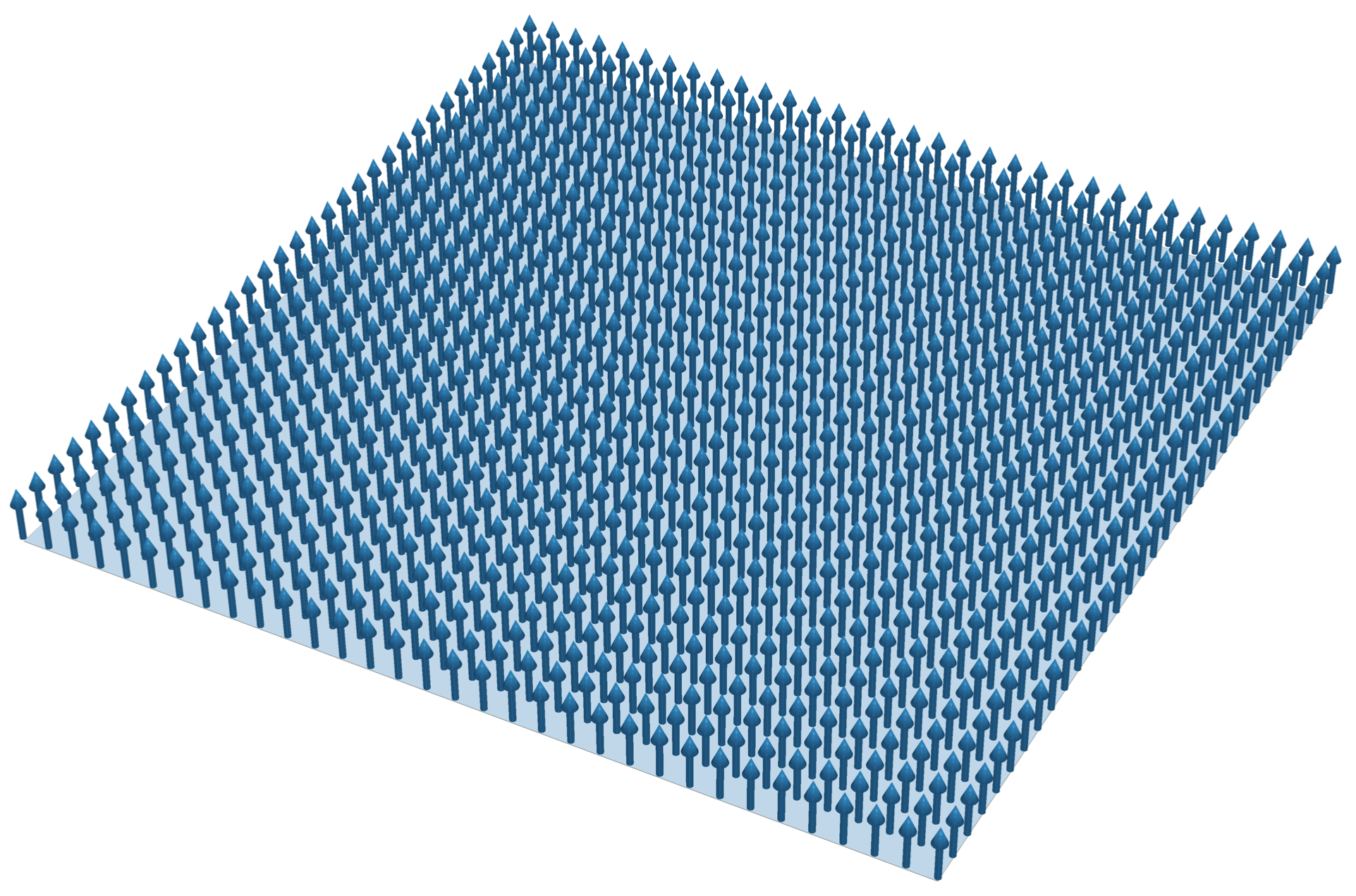}
    \caption{Uniform state $\mm^0=\ee_3$.}
    \label{fig:m0:uniform}
  \end{subfigure}
  \hfill
  \begin{subfigure}[b]{0.44\textwidth}
    \centering
    \includegraphics[width=\linewidth]{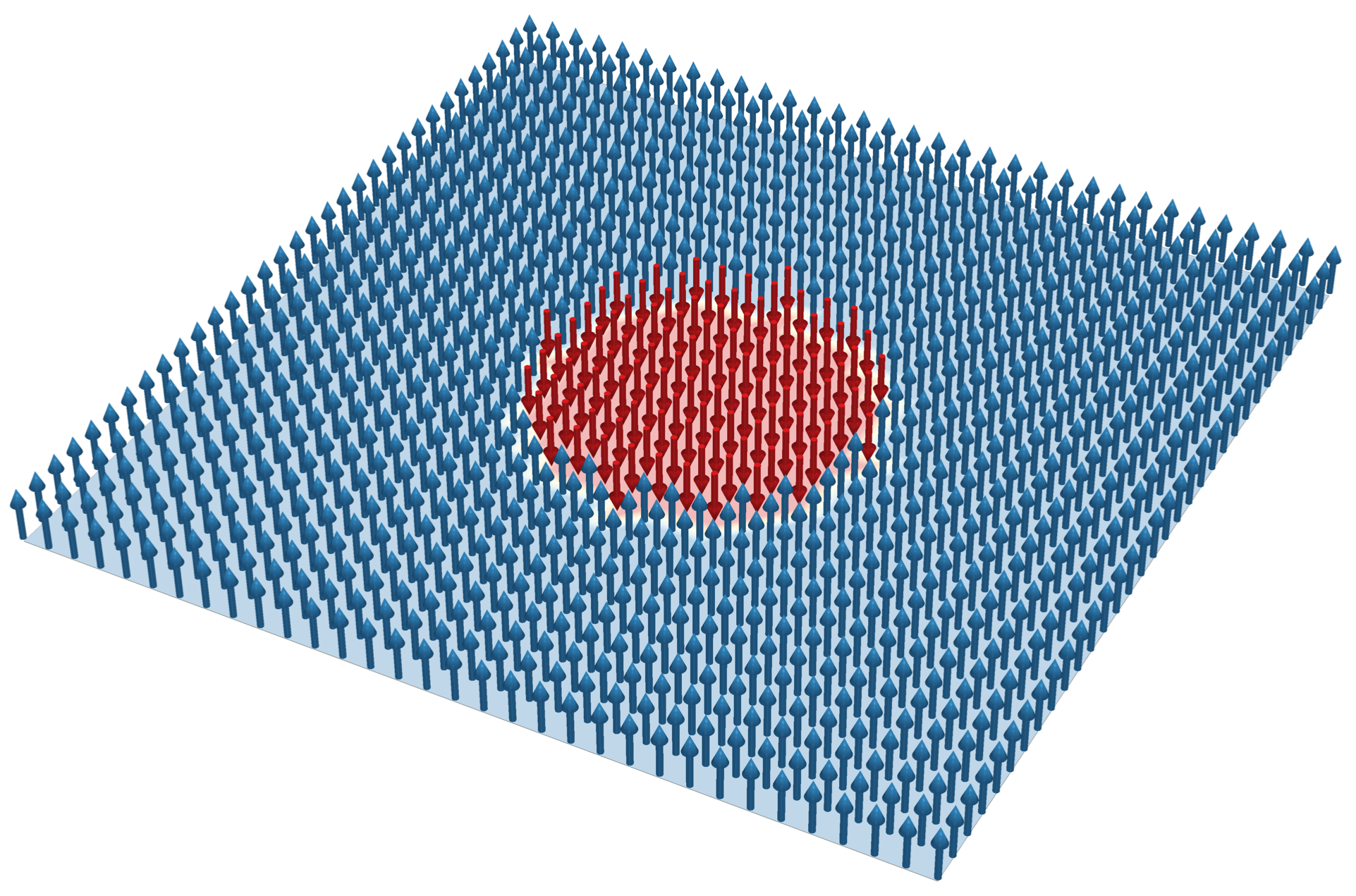}
    \caption{Skyrmion-like state.}
    \label{fig:m0:skyrmion}
  \end{subfigure}
  \caption{Initial magnetization configurations.}
 \end{figure}

\subsubsection{Temporal convergence}\label{section:numerics:temporal}
To investigate the temporal convergence of the BDF2 scheme,
we keep the spatial mesh fixed and successively bisect the time step.
We fix a triangulation of the unit square consisting of
\(1024\) elements and non-dimensional meshsize $h=1/16\approx 0.0625$ 
(or physical size $Lh=5\,\mathrm{nm}$).
We let the initial state evolve up to the 
final physical time \(T_{\rm phys}=20\,\mathrm{ps}\), 
corresponding to the dimensionless final time \(T'=\gamma_0M_sT_{\rm phys}\approx2.5636\).

For each initial state, we compute numerical solutions with
physical time steps
\[
\Delta t_\ell=2^{-\ell}\Delta t_0,
\qquad
\Delta t_0=0.2\,\mathrm{ps},
\qquad
\ell=0,\ldots,6,
\]
or, equivalently, with dimensionless time steps
\[
\tau_\ell=2^{-\ell}\tau_0,
\qquad
\tau_0=\gamma_0M_s\Delta t_0=0.025636.
\]
Thus, the corresponding numbers of time steps are
\(N_\ell=100\cdot2^\ell\).
According to Algorithm~\ref{alg: full discr}, for every computation, 
one first-order tangent-plane step
initializes the BDF2 scheme and 
no nodal renormalization is performed during the time evolution.

Since no exact solution is available, we compute a reference
solution for each initial state on the same spatial mesh,
refining three times further, i.e., using 
\(\tau_{\rm ref}=2^{-9}\tau_0\), and the reference computation
uses \(51,200\) time steps.

For either initial state, let
\(\mm_{h,\tau_\ell}(T')\) denote the numerical magnetization
at the final time and let
\(\mm_{\rm ref}(T')\) denote the corresponding reference solution.
We measure the final-time errors
\[
\text{err}_{\bL^2,\ell}
:=
\normO{
\mm_{h,\tau_\ell}(T')-\mm_{\rm ref}(T')
},
\qquad
\text{err}_{\bH^1,\ell}
:=
\norm{
\mm_{h,\tau_\ell}(T')-\mm_{\rm ref}(T')
}_{\bH^1(\Omega)}.
\]
In Figure~\ref{fig:temporal-convergence-bulk}, we plot the final-time errors against the physical time-step size
on logarithmic scales, 
to assess second-order temporal convergence.
For each initial state, we also report the errors and the
corresponding experimental orders of convergence in Table~\ref{tab:temporal-convergence-bulk}.
\begin{table}[tbp]
\centering
\small
\setlength{\tabcolsep}{4pt}
\renewcommand{\arraystretch}{1.2}

\begin{tabular}{@{}rr|
  r@{\hspace{12pt}}c|
  r@{\hspace{12pt}}c||
  r@{\hspace{12pt}}c|
  r@{\hspace{12pt}}c@{}}
\toprule
\multicolumn{2}{c|}{}
&
\multicolumn{4}{c||}{Constant initial data $\mm^0=(0,0,1)$}
&
\multicolumn{4}{c}{Skyrmion-like initial data}
\\
\cmidrule(lr){3-6}
\cmidrule(lr){7-10}
$\ell$ & $1/\tau_\ell$
& $\text{err}_{\bL^2,\ell}$ & EOC$_\ell$
& $\text{err}_{\bH^1,\ell}$ & EOC$_\ell$
& $\text{err}_{\bL^2,\ell}$ & EOC$_\ell$
& $\text{err}_{\bH^1,\ell}$ & EOC$_\ell$
\\
\midrule
0 & 39.01
& $3.254\times10^{-3}$ & --
& $1.069\times10^{-1}$ & --
& $1.555\times10^{-1}$ & --
& $2.548$             & --
\\
1 & 78.02
& $1.210\times10^{-3}$ & 1.427
& $3.726\times10^{-2}$ & 1.521
& $6.681\times10^{-2}$ & 1.219
& $1.115$             & 1.193
\\
2 & 156.03
& $4.458\times10^{-4}$ & 1.441
& $1.262\times10^{-2}$ & 1.562
& $2.211\times10^{-2}$ & 1.595
& $3.706\times10^{-1}$ & 1.589
\\
3 & 312.06
& $1.429\times10^{-4}$ & 1.641
& $3.960\times10^{-3}$ & 1.672
& $6.300\times10^{-3}$ & 1.811
& $1.057\times10^{-1}$ & 1.810
\\
4 & 624.12
& $3.857\times10^{-5}$ & 1.889
& $1.112\times10^{-3}$ & 1.833
& $1.675\times10^{-3}$ & 1.911
& $2.812\times10^{-2}$ & 1.910
\\
5 & 1248.24
& $9.617\times10^{-6}$ & 2.004
& $2.923\times10^{-4}$ & 1.927
& $4.292\times10^{-4}$ & 1.964
& $7.207\times10^{-3}$ & 1.964
\\
6 & 2496.49
& $2.344\times10^{-6}$ & 2.037
& $7.411\times10^{-5}$ & 1.980
& $1.073\times10^{-4}$ & 2.000
& $1.802\times10^{-3}$ & 2.000
\\
\bottomrule
\end{tabular}

\caption{%
Temporal convergence of the BDF2 scheme for constant initial data $\mm^0 = (0,0,1)$
and skyrmion-like initial data.
The experimental orders of convergence are computed as
$\mathrm{EOC}_\ell=\log_2(\text{err}_{\ell-1}/\text{err}_\ell)$.%
}
\label{tab:temporal-convergence-bulk}
\end{table}

\begin{figure}[tbp]
\centering
\begin{minipage}[t]{0.49\linewidth}
  \centering
  \includegraphics[width=\linewidth]{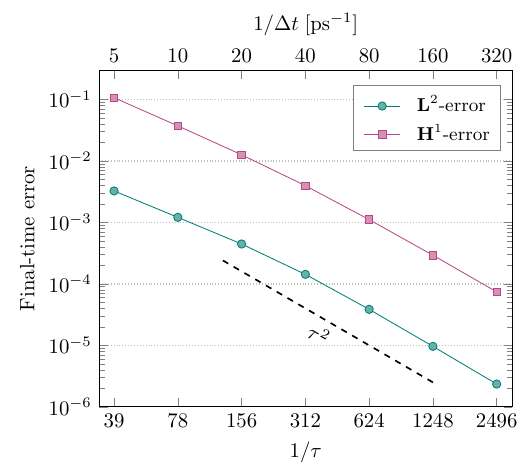}
\end{minipage}
\hfill
\begin{minipage}[t]{0.49\linewidth}
  \centering
  \includegraphics[width=\linewidth]{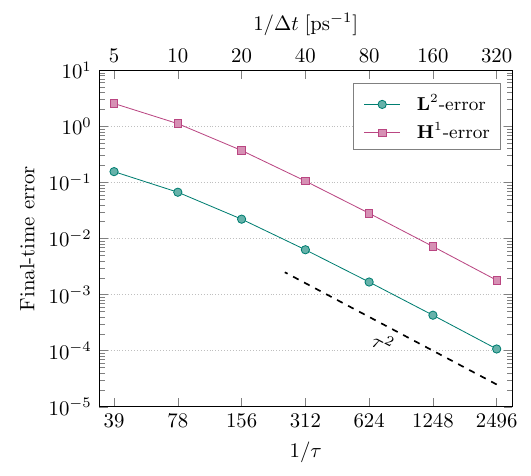}
\end{minipage}
\caption{%
Temporal convergence of the BDF2 scheme for constant initial data
$\mm^0=(0,0,1)$ (left) and skyrmion-like initial data (right). Both dimensional and 
nondimensional time steps are indicated on the horizontal axis.
}
\label{fig:temporal-convergence-bulk}
\end{figure}

\subsubsection{Spatial convergence}\label{section:numerics:spatial}
To investigate the spatial convergence of the BDF2 scheme,
we keep the time step fixed and successively refine the spatial mesh.
For the constant initial state $\mm^0=(0,0,1)$, 
starting from a triangulation of the unit square consisting of
\(64\) elements and non-dimensional meshsize \(h_0=1/4\)
(or physical size \(Lh_0=20\,\mathrm{nm}\)),
we perform six uniform refinements by newest-vertex bisection.
The resulting triangulations $\mathcal T_\ell$ satisfy
\[
h_\ell=2^{-\ell/2}h_0,
\qquad
\#\mathcal T_\ell=64\cdot2^\ell,
\qquad
\ell=0,\ldots,6.
\]
Thus, the finest tested mesh consists of \(4,096\) elements
and has meshsize \(h_6=1/32\), corresponding to
the physical meshsize of \(Lh_6=2.5\,\mathrm{nm}\).

As before, we let the initial state evolve up to the
final physical time \(T_{\rm phys}=20\,\mathrm{ps}\),
corresponding to the dimensionless final time
\(T'=\gamma_0M_sT_{\rm phys}\approx2.5636\).
On every mesh, we use the same fixed physical time step
\(
\Delta t=0.00625\,\mathrm{ps},
\)
or, equivalently, the dimensionless time step
\(
\tau=\gamma_0M_s\Delta t=8.01125\cdot 10^{-4},
\)
corresponding to \(N=3,200\) time steps.

We compute the reference
solution 
on a finer spatial mesh refined 
four times further than the finest tested mesh,
obtaining a reference triangulation with \(65,536\) elements
and \(33,025\) vertices, with meshsize
\(
h_{\rm ref}=2^{-2}h_6=1/128
\)
(equivalently, \(Lh_{\rm ref}=0.625\,\mathrm{nm}\)).

To study the spatial convergence of the skyrmion-like initial state, 
we use the same sequence of meshes and parameters as for the constant initial state. 
The only difference is that we use a smaller time step
to catch the rapid change of the gradient in that case, i.e., 
$\Delta t = 3.90625\cdot 10^{-4}\,\mathrm{ps}$, corresponding to $\tau = 5.007\times 10^{-5}$, $N=51,200$ time steps.

As above, in Figure~\ref{fig:spatial-convergence-bulk},
we plot the final-time errors against the inverse
non-dimensional meshsize \(1/h_\ell\) on logarithmic scales,
together with reference lines corresponding to second-order
convergence in $\bL^2(\Omega)$ and first-order convergence
in $\bH^1(\Omega)$.

\begin{figure}[tbp]
\centering
\begin{minipage}[t]{0.49\linewidth}
  \centering
  \includegraphics[width=\linewidth]{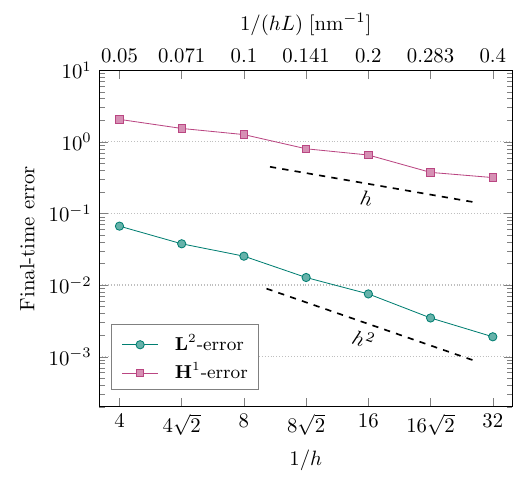}
\end{minipage}
\hfill
\begin{minipage}[t]{0.49\linewidth}
   \centering
   \includegraphics[width=\linewidth]{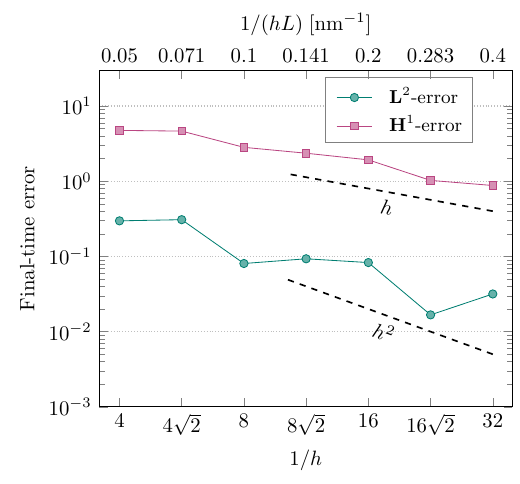}
\end{minipage}
\caption{%
Spatial convergence of the BDF2 scheme for constant initial data
$\mm^0=(0,0,1)$ (left) and skyrmion-like initial data (right). Both dimensional and 
nondimensional mesh sizes are indicated on the horizontal axis.
}
\label{fig:spatial-convergence-bulk}
\end{figure}

Notice that, for this skyrmion-like initial datum~(ii), the errors decrease overall, but
not monotonically, and the expected convergence orders are not observed. 
The main difficulty is that the prescribed initial function 
is discontinuous and jumps across the circle and hence does not
belong to $\mathbf H^1(\Omega)$, violating hypothesis \eqref{enum: punto B1}. 
Therefore, this case is not covered by the convergence analysis of the present work
(nor by the 
\textit{a-priori} analysis in~\cite{part2}).

To give a visual representation of this, Figure~\ref{fig:sharp-initial-sampling} 
illustrates how successive mesh refinements resolve
the transition between the central region with magnetization
$-\ee_3$ and the surrounding region with magnetization $+\ee_3$.
The dashed circle marks the prescribed radius
$R=15\,\mathrm{nm}$.
Although the same nodal prescription is used throughout,
the shape and width of the interpolated transition vary
substantially with the mesh.
These variations suggest that the initial spatial
discretization contributes significantly to the irregular
convergence behaviour.
\begin{figure}
	\includegraphics[width=\linewidth]{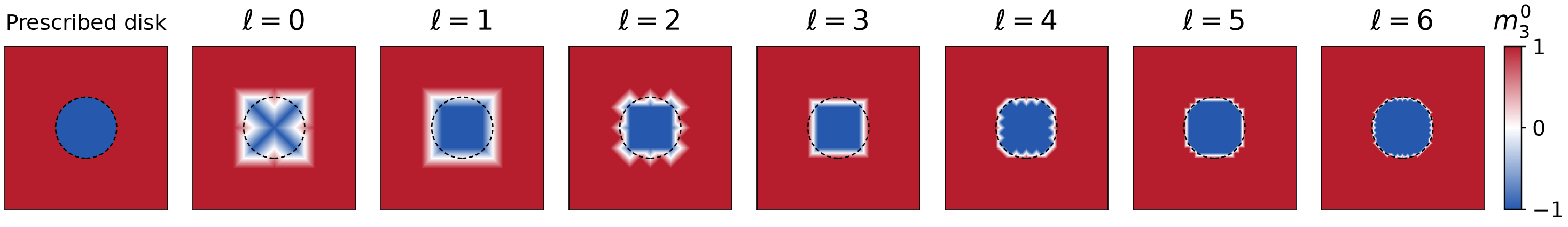}
	\caption{Resolution of the initial transition by successive mesh refinements.}
	\label{fig:sharp-initial-sampling}
\end{figure}
As a complementary diagnostic, we also consider differences 
between numerical solutions on successive meshes, i.e., we consider the differences
\[
d_{X,\ell}
:=\bigl\|\boldsymbol m_{h_\ell,\tau}(T')
-\boldsymbol m_{h_{\ell+1},\tau}(T')\bigr\|_X,
\qquad \ell=0,\ldots,5,
\]
and display the convergence plot in Figure~\ref{fig:spatial-convergence-skyrmion-successive}.
Although some experimental quotients are closer to the desired rates, 
the behaviour remains irregular, 
and these comparisons do not establish improved accuracy or convergence.
\begin{figure}[tbp]
  \centering
  \includegraphics[width=0.5\linewidth]{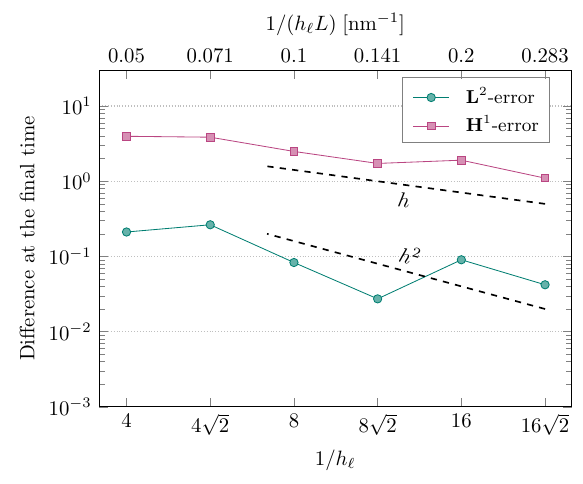}
  \caption{Differences between numerical solutions on successive
  meshes for the skyrmion-like initial state.}
  \label{fig:spatial-convergence-skyrmion-successive}
\end{figure}

\subsubsection{Deviation from the unit-length constraint}\label{section:numerics:constraint}
As in \cite{part1} we also investigate the deviation from the unit-length constraint for both initial data,
and we compare Algorithm~\ref{alg: full discr} 
with the projection-free first-order tangent plane scheme (pfTPS)
applied in every time step. 
We consider the setting of the temporal experiment above, i.e.,
for both initial data~(i) and~(ii): the fixed mesh with $h = 2^{-4}$ (i.e.,
$hL = 5\,\mathrm{nm}$; $1024$ triangles and $545$ vertices), the final time
$T = 20\,\mathrm{ps}$, and the time-step sizes
$\Delta t_\ell = 0.2 \cdot 2^{-\ell}\,\mathrm{ps}$ for $\ell = 0,\dots,6$.

\begin{figure}[ht]
  \centering
  \includegraphics[width=\textwidth]{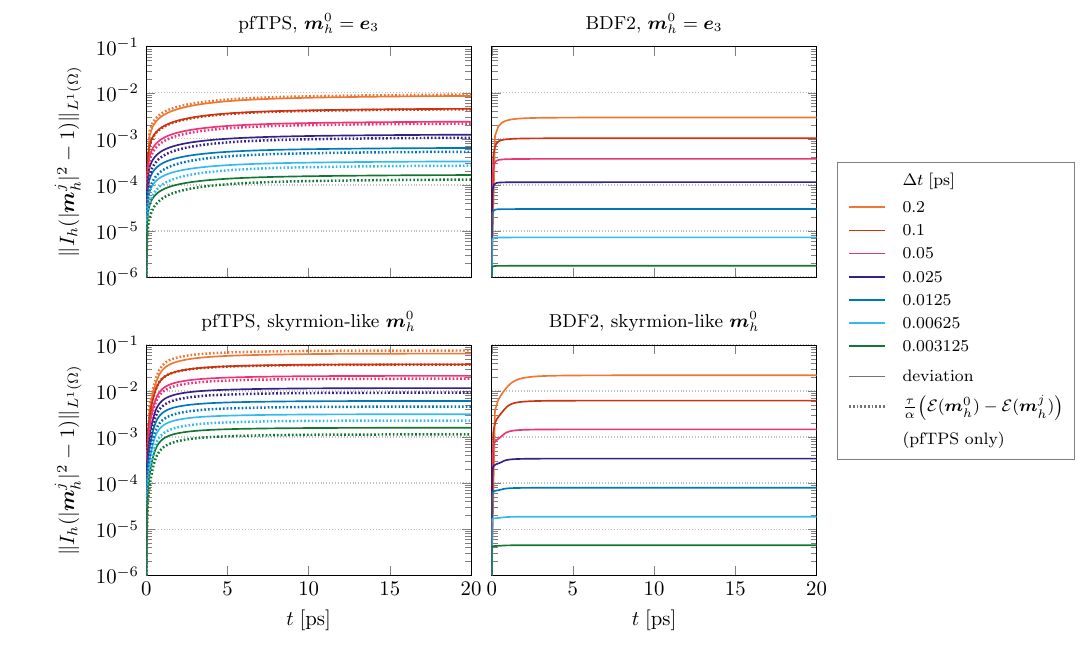}
  \caption{Evolution of $\| I_h(|\mm_h^j|^2 - 1) \|_{L^1(\Omega)}$ for pfTPS
  (left) and BDF2 (right), for the constant initial datum~(i) (top) and the
  skyrmion-like initial datum~(ii) (bottom), on the fixed mesh with
  $hL = 5\,\mathrm{nm}$ and for the time-step sizes
  $\Delta t_\ell = 0.2\cdot 2^{-\ell}\,\mathrm{ps}$, $\ell = 0,\dots,6$.}
  \label{fig:numerics:constraint:evolution}
\end{figure}
 
\begin{figure}[ht]
  \centering
  \includegraphics[width=\textwidth]{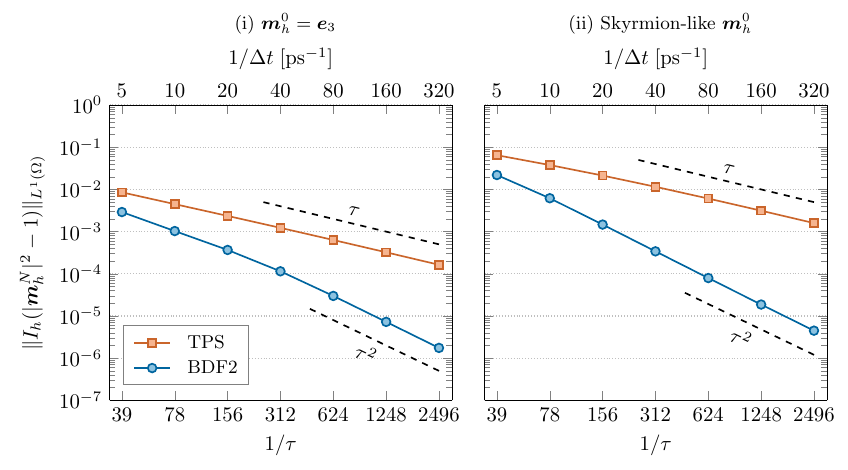}
  \caption{Deviation $\| I_h(|\mm_h^N|^2 - 1) \|_{L^1(\Omega)}$ at
  $T = 20\,\mathrm{ps}$ for pfTPS and BDF2, for the constant initial
  datum~(i) (left) and the skyrmion-like initial datum~(ii) (right), on the
  fixed mesh with $hL = 5\,\mathrm{nm}$.}
  \label{fig:numerics:constraint:convergence}
\end{figure}

\begin{table}[ht]
\centering
\small
\begin{tabular}{@{}rr|ccc|ccc@{}}
\toprule
\multicolumn{2}{c|}{}
& \multicolumn{3}{c|}{$\mm_h^0 = \boldsymbol{e}_3$~(i)}
& \multicolumn{3}{c}{Skyrmion-like initial datum~(ii)}
\\
\cmidrule(lr){3-5}
\cmidrule(lr){6-8}
\multicolumn{2}{c|}{}
& & \multicolumn{2}{c|}{$\tau^2\sum_{j=2}^N\|\d_t^2\mm_h^j\|_{\bL^2(\Omega)}^2$}
& & \multicolumn{2}{c}{$\tau^2\sum_{j=2}^N\|\d_t^2\mm_h^j\|_{\bL^2(\Omega)}^2$}
\\
\cmidrule(lr){4-5}
\cmidrule(lr){7-8}
$\ell$ & $\Delta t_\ell$ [ps]
& $\|\vv_h^0\|_{\bL^2(\Omega)}^2$ & \makebox[1.4cm]{pfTPS} & \makebox[1.4cm]{BDF2}
& $\|\vv_h^0\|_{\bL^2(\Omega)}^2$ & \makebox[1.4cm]{pfTPS} & \makebox[1.4cm]{BDF2}
\\
\midrule
0 & 0.2
  & 1.1429 & 0.5342 & 0.6643
  & 3.9006 & 4.8249 & 9.3519 \\
1 & 0.1
  & 1.4918 & 0.5609 & 0.6892
  & 4.7210 & 4.4280 & 8.7893 \\
2 & 0.05
  & 1.8933 & 0.5943 & 0.7506
  & 5.4957 & 3.6004 & 6.2632 \\
3 & 0.025
  & 2.2753 & 0.5579 & 0.7232
  & 6.1061 & 2.6336 & 3.8357 \\
4 & 0.0125
  & 2.5276 & 0.4200 & 0.5185
  & 6.4871 & 1.7157 & 2.1656 \\
5 & 0.00625
  & 2.6565 & 0.2644 & 0.3024
  & 6.6960 & 1.0075 & 1.1507 \\
6 & 0.003125
  & 2.7170 & 0.1493 & 0.1611
  & 6.8053 & 0.5509 & 0.5916 \\
\bottomrule
\end{tabular}
\caption{Quantities in the regularity condition \eqref{eq: add reg for 2nd order}.
The first time-step, and hence
$\vv_h^0$, is the same for pfTPS and BDF2.}
\label{tab:numerics:constraint:regularity}
\end{table}

Neither scheme involves a nodal projection, and the nodal lengths
$|\mm_h^j(z)|$ never decrease. We recall that, for pfTPS, the nodal orthogonality
$\vv_h^j(z)\cdot\mm_h^j(z) = 0$ gives
\[
  |\mm_h^n(z)|^2 = 1 + \tau^2 \sum_{j=0}^{n-1} |\vv_h^j(z)|^2
  \qquad\text{for all } z \in \mathcal{N}_h,
\]
while for BDF2, the identity
\[
  |\mm_h^n(z)|^2
  = 1 + \frac32\Big(1-\frac{1}{3^n}\Big)\tau^2|\vv_h^0(z)|^2
  + \frac32\,\tau^4\sum_{i=2}^{n}\Big(1-\frac{1}{3^{n+1-i}}\Big)\,|\d_t^2\mm_h^i(z)|^2
\]
of~\cite[Equation~(34)]{part1} holds verbatim, since its proof only
uses the update formula of the BDF2 step and the nodal orthogonality
$\vv_h^j(\zz)\cdot\widehat{\mm}_h^{j+1}(\zz) = 0$.

Figure~\ref{fig:numerics:constraint:evolution} shows the evolution of
$\norm{I_h(|\mm_h^j|^2 - 1)}_{L^1(\Omega)}$. For both initial data, and
especially for small time-step sizes, BDF2 is much more accurate than pfTPS:
for $\Delta t_6 = 3.125\cdot 10^{-3}\,\mathrm{ps}$, its deviation is about
two orders of magnitude smaller. The deviation of pfTPS grows mainly during
the first picoseconds and then remains almost constant. This follows from
the identity for pfTPS: every time step adds $\tau^2|\vv_h^j(z)|^2$ to
$|\mm_h^j(z)|^2$, so that the deviation grows as long as the magnetization is moving, and 
stops growing once the magnetization is (almost) at rest. 
More precisely, summing over all nodes and using the norm
equivalence~\eqref{eq: nodal norm equivalence}, we obtain
\begin{align*}
	\norm{I_h(|\mm_h^n|^2 - 1)}_{L^1(\Omega)}
	\simeq \tau^2 \sum_{j=0}^{n-1} \|\vv_h^j\|_{\bL^2(\Omega)}^2
	\approx \tau \int_0^{t_n} \normO{\partial_t \mm(s)}^2 \, \d s
	\stackrel{\eqref{eq: energy law}}{\approx}
	\frac{\tau}{\alpha}\big(\mathcal{E}(\mm_h^0) - \mathcal{E}(\mm_h^n)\big),
\end{align*}
where the last step follows from the energy law for pfTPS.
Hence, the deviation of pfTPS is essentially $\tau/\alpha$ times the energy
dissipated up to $t_n$: it grows as long as the magnetization dissipates
energy, linearly in time if the energy decays at a constant rate, and it
stops growing once the energy has (almost) reached its minimum. In the
present experiments, without applied field and with $\alpha = 0.3$, the
magnetization relaxes quickly, and most of the energy is dissipated during
the first picoseconds.


This link to the dissipated energy is specific to pfTPS, and it explains why
BDF2 performs much better. By the identity for BDF2, besides the start-up
term with $\vv_h^0$, BDF2 accumulates only the squared second differences
$\tau^4|\d_t^2\mm_h^i(z)|^2$ instead of the squared increments
$\tau^2|\vv_h^j(z)|^2$. Hence, BDF2 is much more accurate only if the time-step size is small
compared to the time in which the magnetization changes: then the second
differences are much smaller than the increments, and the start-up term is
much smaller than the sum of all increments of pfTPS. The deviation of BDF2 
is thus created almost entirely in the first time
steps and stays far below the dotted lines (all panels of
Figure~\ref{fig:numerics:constraint:evolution} share the same $y$-axes). This
requires the second differences to be much smaller than the increments,
i.e., the time-step size to be small compared to the time in which the
magnetization changes. This is the case for $\Delta t_6$, where the start-up
term accounts for more than $90\%$ of the final deviation. For
$\Delta t_0 = 0.2\,\mathrm{ps}$, instead, the first $2\,\mathrm{ps}$, in
which the magnetization changes quickly, contain only $10$ time steps: the
start-up term accounts for only $60\%$ for~(i) and $26\%$ for~(ii).
This is why, for $\Delta t_0$, the deviations of pfTPS and BDF2 are of the same
order of magnitude.
 
Figure~\ref{fig:numerics:constraint:convergence} shows the convergence order of the 
deviation $\| I_h(|\mm_h^N|^2 - 1) \|_{L^1(\Omega)}$ at the final time as a
function of the time-step size. We observe
first-order convergence for pfTPS, and second-order convergence for BDF2.
Moreover, Table~\ref{tab:numerics:constraint:regularity} shows that the
quantities in the regularity condition \eqref{eq: add reg for 2nd order}
remain bounded.


\FloatBarrier

\subsubsection{Relaxed configurations for bulk DMI}\label{section:numerics:relaxed-bulk}
Finally, as in \cite{hpprss2019}, we investigate the configurations obtained by relaxation
for both initial data~(i)--(ii), varying $D\in\{1,\ldots,8\}\times10^{-3}\,\mathrm{J\,m^{-2}}$.
For $D=0$, both initial data are exact equilibria of the discrete problem,
since the discrete effective field is parallel to $\ee_3$ at every vertex;
cf.~\cite[Remark~20]{dFIPS2025}.
We first consider bulk DMI, which is covered by our analysis, and then interfacial DMI,
which is not covered by our analysis but is the setting of~\cite{hpprss2019},
so that our results can be compared directly with those reported there.
In both cases, we compute on the disk of~\cite[Section~4.2]{hpprss2019} and on the unit square.

We use the following triangulations:
an unstructured triangulation of the square with $14,802$ elements,
and an unstructured triangulation of the disk with $11,776$ elements.
Moreover, we compare the unstructured triangulation of the square with a structured one
with $16,384$ elements and meshsize $h=1/64$, which
is obtained from a coarse triangulation by uniform
newest-vertex bisection, and is invariant under all rotations and reflections
of the square.
We retain the material parameters and the vanishing external field
specified above on all three triangulations.
Since $L=80\,\mathrm{nm}$ is also the diameter of the disk, the nondimensional
coefficients $\lamex$, $\lDM$ and $c_{\rm ani}$ are the same on the disk and on the square,
and the reversed region of the initial datum~(ii) is centred at the centre of the domain.
The fixed physical time step is $\Delta t=0.1\,\mathrm{ps}$,
i.e., $\tau=0.012818$.
Each computation is continued until the magnetization is relaxed.

\begin{figure}[ht!]
  \centering
  \includegraphics[width=\linewidth]{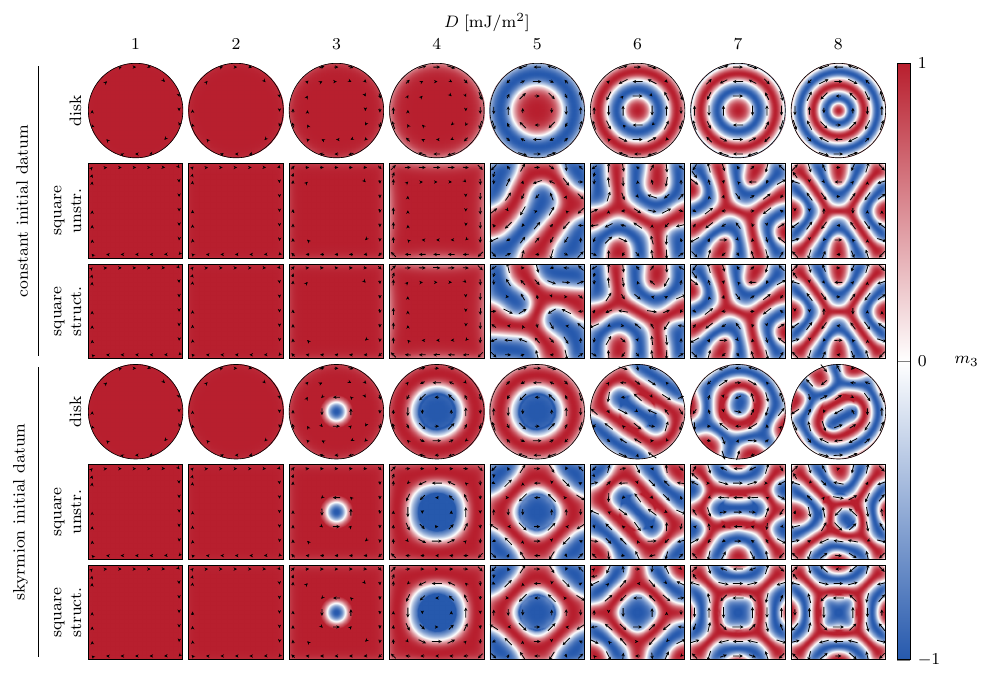}
  \caption{%
    Numerical configurations obtained by relaxation with bulk DMI
    for the initial data~(i)--(ii) and
    $D\in\{1,\ldots,8\}\,\mathrm{mJ\,m^{-2}}$:
    for the initial datum~(i) (first to third row) and~(ii) (fourth to sixth row), each
    computed on the disk, on the square with the unstructured triangulation, and on the
    square with the structured triangulation (from top to bottom).
    Colours represent the out-of-plane component $m_3$,
    while arrows indicate the in-plane magnetization.
  }
  \label{fig:relaxation-bulk}
\end{figure}

Figure~\ref{fig:relaxation-bulk} shows the resulting bulk-DMI configurations.
On the disk (first and fourth row), the uniform initial datum~(i) leads to approximately
uniform states for $D\leq4\,\mathrm{mJ\,m^{-2}}$, with an increasing tilt of the
magnetization next to the boundary, and to radially symmetric configurations with
concentric domains, whose number increases with $D$, for $D\geq5\,\mathrm{mJ\,m^{-2}}$.
For the skyrmion-like initial datum~(ii), the reversed region disappears for
$D\in\{1,2\}\,\mathrm{mJ\,m^{-2}}$, and an isolated skyrmion, whose radius increases with
$D$, is obtained for $D\in\{3,4,5\}\,\mathrm{mJ\,m^{-2}}$; for larger $D$, the reversed
region elongates ($D=6\,\mathrm{mJ\,m^{-2}}$) or is surrounded by further domains
($D\in\{7,8\}\,\mathrm{mJ\,m^{-2}}$).
In agreement with~\cite[Remark~19]{dFIPS2025}, the domain walls exhibit a Bloch-type
structure: near the domain wall, the in-plane magnetization is approximately tangential
to the boundary of the reversed region.
This progression makes physical sense: exchange penalizes rapid spatial variation,
perpendicular anisotropy favours out-of-plane magnetization, while DMI favours twisting,
so that increasing $D$ promotes domain walls and more complicated patterns.

On the square (second, third, fifth and sixth row), the initial datum~(i) leads again to
approximately uniform states for $D\leq4\,\mathrm{mJ\,m^{-2}}$, whereas for
$D\geq5\,\mathrm{mJ\,m^{-2}}$ it leads to stripe domains, which become narrower as $D$
increases.
For~(ii), the reversed region disappears for $D\in\{1,2\}\,\mathrm{mJ\,m^{-2}}$ and
persists as an isolated skyrmion for $D\in\{3,4\}\,\mathrm{mJ\,m^{-2}}$, while more complex
domain patterns, adapted to the corners of the square, appear for $D\geq5\,\mathrm{mJ\,m^{-2}}$.
Thus, the states on the square and on the disk agree for $D\leq4\,\mathrm{mJ\,m^{-2}}$,
whereas for larger $D$ the radially symmetric configurations favoured by the rotational
invariance of the disk are replaced by less symmetric patterns.
The two triangulations of the square lead to the same states for $D\leq4\,\mathrm{mJ\,m^{-2}}$.
For larger $D$, the structured triangulation, which is invariant under the symmetries of the
square, preserves symmetric configurations, such as the fourfold structures obtained
from~(ii), whereas the unstructured triangulation breaks them.
Apart from this, the domain patterns on the two triangulations have the same domain width.

\FloatBarrier
\subsubsection{Relaxed configurations for interfacial DMI}\label{section:numerics:relaxed-interfacial}
We now turn to interfacial DMI.
For a closer comparison with~\cite[Section~4.2]{hpprss2019},
we repeat the experiment in the setting of this cited work,
i.e., with interfacial DMI instead of bulk DMI, although this case
is not strictly covered by our theory. Differently from~\cite{hpprss2019}, we keep the local thin-film 
approximation of the stray field, and we
consider both the disk of~\cite{hpprss2019} and the square. We also keep the same
discretization and material parameters.
Therefore, the bulk-DMI contribution $-\lDM\operatorname{curl}\mm$
to the effective field is replaced by
\[
\lDM
\begin{pmatrix}
\partial_1m_3\\
\partial_2m_3\\
-\partial_1m_1-\partial_2m_2
\end{pmatrix},
\qquad\text{and}\qquad
c_{\rm ani}
=\frac{2K}{\mu_0M_s^2}-1.
\]
Thus, the additional anisotropy correction induced by bulk-DMI
relaxation is absent.

\begin{figure}[ht!]
  \centering
  \includegraphics[width=\linewidth]{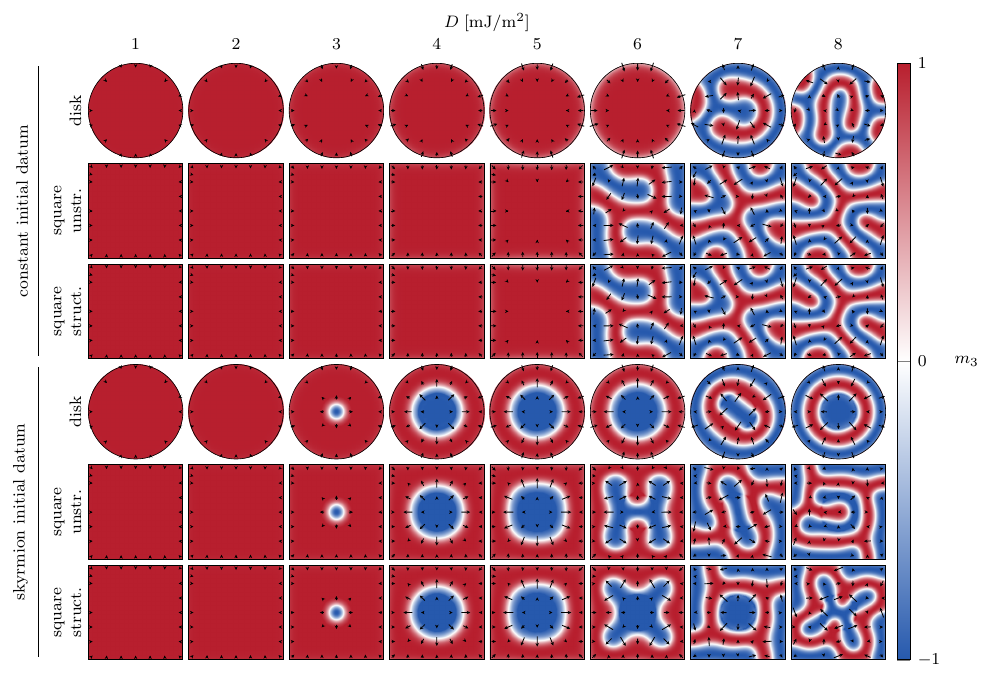}
  \caption{%
    Numerical configurations obtained by relaxation with interfacial
    DMI for the initial data~(i)--(ii) and
    $D\in\{1,\ldots,8\}\,\mathrm{mJ\,m^{-2}}$:
    for the initial datum~(i) (first to third row) and~(ii) (fourth to sixth row), each
    computed on the disk, on the square with the unstructured triangulation, and on the
    square with the structured triangulation (from top to bottom).
    Colours represent the out-of-plane component $m_3$,
    while arrows indicate the in-plane magnetization.
  }
  \label{fig:relaxation-interfacial}
\end{figure}

Figure~\ref{fig:relaxation-interfacial} shows the corresponding
interfacial-DMI configurations, whose domain walls exhibit a N\'eel-type structure, in
agreement with~\cite[Remark~19]{dFIPS2025}.
On the disk, which is the setting of~\cite[Section~4.2]{hpprss2019} up to the thin-film
approximation of the stray field, our results agree with~\cite[Figures~5--6]{hpprss2019}:
the uniform initial datum~(i) leads to approximately uniform states for
$D\leq6\,\mathrm{mJ\,m^{-2}}$ and to multidomain states for $D\in\{7,8\}\,\mathrm{mJ\,m^{-2}}$,
such as the horseshoe-like configuration of~\cite[Figure~5]{hpprss2019}, whereas the
skyrmion-like initial datum~(ii) leads to an isolated N\'eel skyrmion, whose size increases
with $D$, for $D\in\{3,\ldots,6\}\,\mathrm{mJ\,m^{-2}}$ and to multidomain states for
$D\in\{7,8\}\,\mathrm{mJ\,m^{-2}}$.

On the square, the uniform initial datum~(i) leads to approximately uniform
states only for $D\leq5\,\mathrm{mJ\,m^{-2}}$, and the skyrmion-like initial datum~(ii) leads
to an isolated skyrmion only for $D\in\{3,4,5\}\,\mathrm{mJ\,m^{-2}}$, i.e., in both cases up to
$D=5\,\mathrm{mJ\,m^{-2}}$ instead of $D=6\,\mathrm{mJ\,m^{-2}}$ as on the disk; for larger $D$,
more complex domain patterns appear.
As for the bulk-DMI, the corners of the square favour rounded-square and fourfold-like
configurations, and
the symmetric structures obtained from~(ii) on the structured triangulation are broken on
the unstructured one.

Compared with bulk DMI, simple configurations persist over a wider range
of $D$.
The reason is that, for the same $D$, the effective anisotropy is larger for interfacial DMI:
it equals $K_{\rm eff} = K-\mu_0M_s^2/2\approx589\,\mathrm{kJ\,m^{-3}}$ for every $D$, whereas for bulk DMI
it is reduced to $K_{\rm eff} = K-\mu_0M_s^2/2-D^2/(4A)$, see Section~\ref{section:numerics:model}, therefore to
$172\,\mathrm{kJ\,m^{-3}}$ for $D=5\,\mathrm{mJ\,m^{-2}}$, and it is negative for
$D\geq6\,\mathrm{mJ\,m^{-2}}$.
Hence, for bulk DMI, the twisting effect of the DMI prevails already at smaller values of $D$.

In Figure~\ref{fig:relaxation-energy} (left), we compare the energies of the
relaxed states on the disk computed with our BDF2 scheme with those
of~\cite[Figure~7]{hpprss2019}, which were computed with the tangent plane integrator TPS1
for the three-dimensional model with the full stray field (the latter values are read from~\cite[Figure~7]{hpprss2019}).
For the uniform initial datum~(i), the two curves coincide up to the reading accuracy.
Both schemes show the same behaviour: the skyrmion has a higher energy than the
approximately uniform state for $D\leq4\,\mathrm{mJ\,m^{-2}}$ and a lower one for
$D\in\{5,6\}\,\mathrm{mJ\,m^{-2}}$, and the energy drops sharply at $D=7\,\mathrm{mJ\,m^{-2}}$,
where multidomain states appear.
In Figure~\ref{fig:relaxation-energy} (right), we compare the BDF2 energies for interfacial
and bulk DMI: the latter are lower, especially for $D\geq5\,\mathrm{mJ\,m^{-2}}$, in agreement
with the smaller effective anisotropy, as noticed above.

\begin{figure}[ht!]
  \centering
  \includegraphics[width=\linewidth]{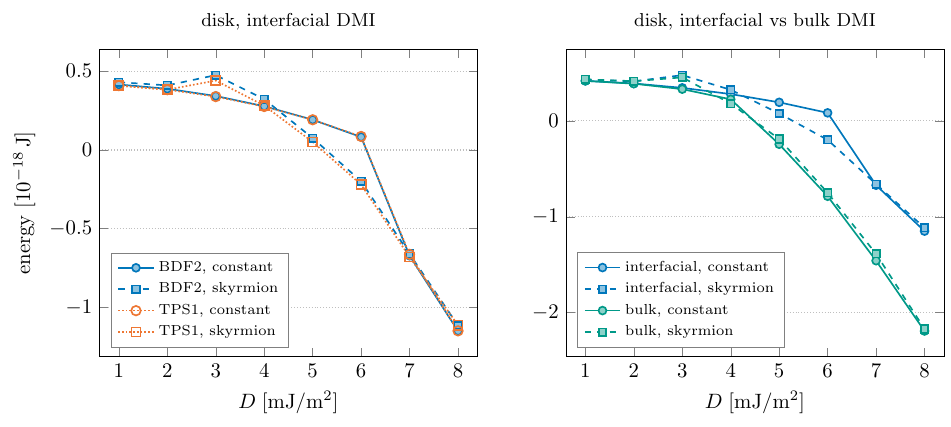}
  \caption{%
    Energies of the relaxed states on the disk for the initial data~(i)--(ii).
    Left: interfacial DMI, BDF2 compared with TPS1 (values read
    from~\cite[Figure~7]{hpprss2019}). Right: BDF2, interfacial and bulk DMI.
  }
  \label{fig:relaxation-energy}
\end{figure}


\subsection{A benchmark problem: relaxation of a bulk-DMI skyrmion}

We consider the two-dimensional skyrmion benchmark problem proposed in~\cite{CortesOrtunoEtAl2018} to assess the ability of the scheme to resolve equilibrium configurations generated by bulk DMI. The computation is performed on the two-dimensional disk $\Omega=
\{\xx\in\mathbb R^2: |\xx|<50\,\mathrm{nm}\}$.
The material parameters are
\[
A=13\,\mathrm{pJ\,m^{-1}},
\qquad
D=3\,\mathrm{mJ\,m^{-2}},
\qquad
M_s=0.86\,\mathrm{MA\,m^{-1}},
\qquad
K_u=0.4\,\mathrm{MJ\,m^{-3}},
\]
as in~\cite{CortesOrtunoEtAl2018}. No applied or stray field is included.
The lower-order contribution in Algorithm~\ref{alg: full discr} is given only by the easy-axis anisotropy $\bpi(\mm)= \kappa_u m_3 \boldsymbol e_3$.
Thus, measuring lengths in nanometers, in our notation,
\[
\lamex^2
=
\frac{2A}{\mu_0 M_s^2(1\,\mathrm{nm})^2}
=
27.9748,
\quad
\lDM
=
\frac{2D}{\mu_0 M_s^2(1\,\mathrm{nm})}
=
6.45572,
\quad
\kappa_u
=
\frac{2K_u}{\mu_0M_s^2}
=
0.86076.
\]

We initialize the calculation with an initial datum having positive out-of-plane magnetization $+\ee_3$ in a disk of radius $18$ nm and negative out-of-plane magnetization $-\ee_3$ outside of it.
We use a triangular mesh consisting of $13824$ elements with $h_{\max}=1.5$ nm.
Since only the relaxed state is of interest, we take the Gilbert damping parameter $\alpha=1$ to accelerate the relaxation. The time step is $\tau=0.025$ in dimensionless time, and the iteration is continued until the relative discrete velocity $\vv_h^j$ is sufficiently small, i.e. $|\Omega|^{-1/2} \normO{\vv_h^j} \approx \normO{\vv_h^j}/\normO{\mm_h^j} \leq 10^{-6}$.

Figure~\ref{fig: bulk DMI skyrmion benchmark} shows the resulting equilibrium
profile. The left panel plots the out-of-plane component $m_3$ along a
diameter of the disk. Following~\cite{CortesOrtunoEtAl2018}, we define the
skyrmion radius $r_{\rm sk}$ as the positive value of $r$ at which
$m_3(r)=0$. Interpolation of the numerical profile gives
\[
r_{\rm sk}=23.4\,\mathrm{nm}.
\]
This is in close agreement with the analytical value
$r_{\rm sk}=22.03\,\mathrm{nm}$ reported in
\cite{CortesOrtunoEtAl2018}; the discrepancy is approximately $6.2\%$.
It is also close to the values $21.87\,\mathrm{nm}$ obtained with OOMMF
and Fidimag and $22.1\,\mathrm{nm}$ obtained with MuMax3 in the same
benchmark. Furthermore, the slight variation of $m_3$ near the outer
boundary is consistent with the characteristic boundary canting
induced by the chiral boundary condition.

The right panel provides a complementary test of the DMI structure.
We evaluate the radial component $m_r= m_1\cos\phi+m_2\sin\phi$ on the circle $r=r_{\rm sk}$. For a bulk-DMI skyrmion, the in-plane magnetization is of Bloch type pointing in the azimuthal direction:
\[
m_r(r_{\rm sk},\phi)=0
\qquad\text{for all }\phi\in[0,2\pi],
\]
The numerical profile remains close to this value, as expected from the theory. We further repeat this experiment with a similar initial datum, but with a smooth transition between the upward and downward pointing magnetizations. In this case, we obtain a relaxed skyrmion radius $r_{\rm sk}=21.72$ nm that is even closer to the theoretical value.

Figure~\ref{fig: bloch relaxation} shows the initial and relaxed magnetization configurations. Starting from the discontinuous initial datum, the LLG dynamics dissipates the micromagnetic energy and drives the configuration toward a stationary state. During the relaxation, the magnetic configuration adjusts its profile and radius to form a Bloch-type skyrmion.

This experiment provides a quantitative validation that the finite element discretization reproduces the characteristic equilibrium configuration of the bulk-DMI model. In particular, the computed skyrmion radius agrees closely with
the established benchmark values, while the nearly vanishing radial component at $r=r_{\rm sk}$ confirms the expected Bloch-type chirality.

\newcommand{\rsk}{23.41}

\begin{figure}[ht]
	\centering
	
	\begin{subfigure}[t]{0.48\textwidth}
		\centering
		\includegraphics{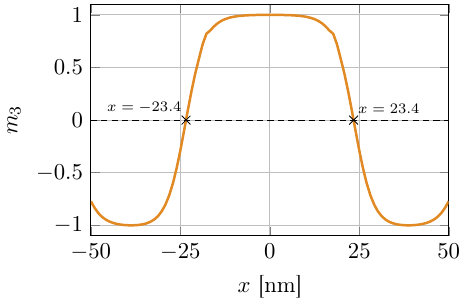}
		
		\caption{Out-of-plane component $m_3$ along diameter.}
		\label{fig: skyrmion mz profile}
	\end{subfigure}
	\hspace{0.02\textwidth}
	\begin{subfigure}[t]{0.48\textwidth}
		\centering
		\includegraphics{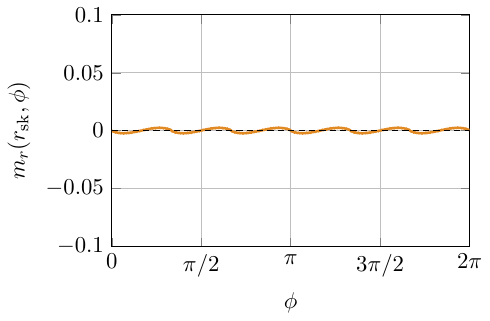}
		\captionsetup{margin={5mm,0mm}}
		\caption{Radial component $m_r$ at $r=r_{\rm sk}$ against $\phi$.}
		\label{fig: skyrmion mr profile}
	\end{subfigure}
	
	\caption{
		Bloch-skyrmion profile for the bulk-DMI relaxation benchmark problem.
		Left: out-of-plane component along the disk diameter, whose zero
		crossings give $r_{\rm sk}=23.4\,\mathrm{nm}$.
		Right: radial component along the circle $r=r_{\rm sk}$; the
		dashed line denotes the theoretical bulk-DMI profile
		$m_r(r_{\rm sk},\phi)=0$.
	}
	\label{fig: bulk DMI skyrmion benchmark}
\end{figure}

\begin{figure}[ht]
	\centering
	\begin{subfigure}[t]{0.43\textwidth}
		\centering
		\includegraphics[width=\textwidth]{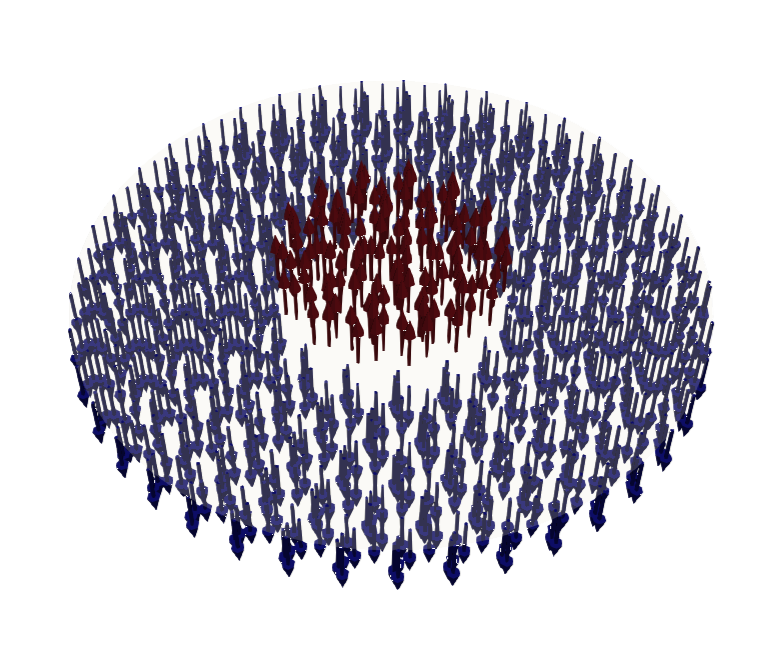}
		\caption{Initial configuration.}
		\label{fig: bloch initial}
	\end{subfigure}
	\hfill
	\begin{subfigure}[t]{0.43\textwidth}
		\centering
		\includegraphics[width=\textwidth]{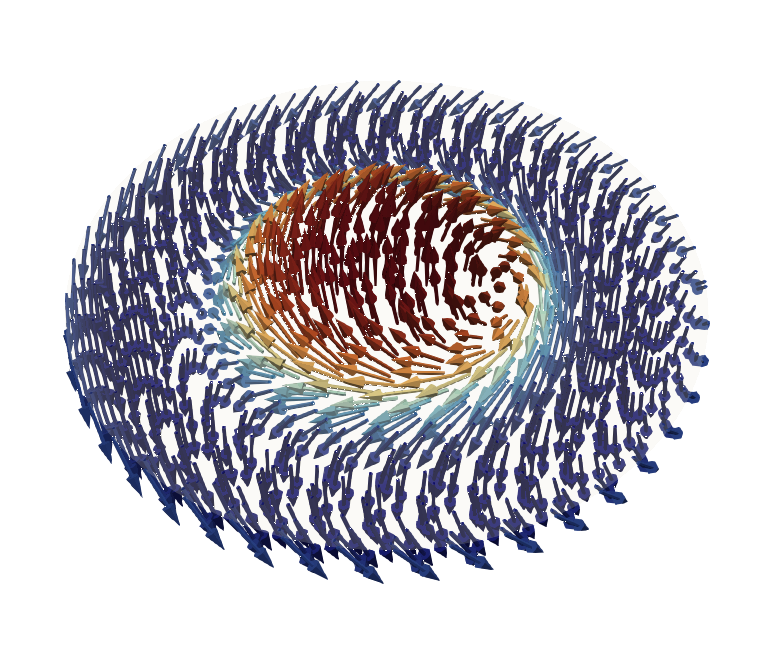}
		\caption{Relaxed configuration.}
		\label{fig: bloch final}
	\end{subfigure}
	\hfill
	\begin{subfigure}[t]{0.09\textwidth}
		\centering
		\includegraphics[width=\textwidth]{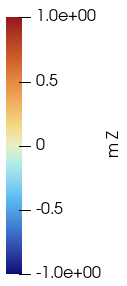}
		\label{fig: bloch legend}
	\end{subfigure}
	\caption{Evolution of a Bloch-type skyrmion under the LLG relaxation.
		The left panel shows the initial configuration, while the right panel
		shows the relaxed state. The color scale represents the out-of-plane
		component $m_3$, while the arrows indicate the direction of the in-plane magnetization.}
	\label{fig: bloch relaxation}
\end{figure}


\subsection{Current-driven steering of a Bloch skyrmion}

Current-driven manipulation of magnetic skyrmions is one of the basic
mechanisms underlying proposed skyrmion-based spintronic devices,
including racetrack memories and logic devices; see, e.g.,
\cite{tomasello2014etal,zhang2020skyrmionelectronics}.
In particular, spin-transfer torque provides a mechanism for transporting
Bloch skyrmions in bulk-DMI materials. Motivated by these
observations, we consider a steering experiment in which a time-dependent
in-plane current repeatedly changes the direction of motion of a single
Bloch skyrmion. We verify that the proposed scheme can stably transport and steer a nontrivial chiral texture over an extended time interval.

We consider the rectangular domain $\Omega=(-250,250)\times(-50,50)\ {\rm nm}^2$ and use the material parameters for MnSi associated with the Bloch-skyrmion spin-transfer-torque configuration in \cite{tomasello2014etal}, namely
\[
M_s=152\,{\rm kA/m},
\qquad
A=0.32\,{\rm pJ/m},
\qquad
D=0.115\,{\rm mJ/m^2},
\]
together with
\[
\alpha=\xi=0.04,
\qquad
P=0.4,
\qquad
\mu_0 H_{\rm ext}=0.25\,{\rm T}.
\]
No crystalline anisotropy is included. We employ the standard local thin-film approximation $\hh_{\rm d}(\mm)=-m_3\ee_3$ for the demagnetizing field. The initial condition for the dynamic simulation is obtained by first relaxing a Bloch-type skyrmion with $\alpha=1$ and zero current.

Lengths are measured in nanometers and time is non-dimensionalized by $t=\gamma_0M_s t_{\rm phys}$. For the above value of $M_s$, one dimensionless time unit corresponds to  approximately $29.76\,{\rm ps}$. We use the time step $\tau=0.05$, which corresponds to approximately $1.49\,{\rm ps}$.

To generate a bounded two-dimensional trajectory, we prescribe a smooth,
spatially uniform electrical current density. To avoid an abrupt switch-on of the Zhang--Li torque and to remain consistent with the assumed temporal continuity of the current density, we introduce a short smooth ramp over $0\le t\le t_{\rm r}$. During an initial ramp
$0\le t\le t_{\rm r}$, we set
\[
\jj(t)
=
\rho(t)
\begin{pmatrix}
	J_x\\
	J_y
\end{pmatrix},
\qquad
\rho(t)
=
3\left(\frac{t}{t_{\rm r}}\right)^2
-
2\left(\frac{t}{t_{\rm r}}\right)^3,
\]
whereas, for $t>t_{\rm r}$,
\[
\jj(t)
=
\begin{pmatrix}
	J_x\cos\bigl(\omega(t-t_{\rm r})\bigr)\\[0.3em]
	J_y\cos\bigl(2\omega(t-t_{\rm r})\bigr)
\end{pmatrix},
\qquad
\omega=\frac{2\pi}{T_{\rm p}}.
\]
Here, we fix
\[
J_x=50\,{\rm MA/cm^2},
\qquad
J_y=30\,{\rm MA/cm^2},
\qquad
t_{\rm r}=6.72144,
\qquad
T_{\rm p}=134.4288.
\]
Two complete periods are simulated after the ramp, giving a total simulation time of
$8.2\,{\rm ns}$ in physical units (or 275.58 in dimensionless time)

The electrical current density is converted to the coefficient $\nuu$ appearing in the Zhang--Li torque according to
\[
\nuu(t)
=
\frac{\mu_B P}
{\gamma_0 e M_s^2 L_0}
\jj(t),
\qquad L_0=1\,{\rm nm},
\]
and at each time level the discrete Zhang--Li torque is evaluated following \eqref{eq: discrete Zhang Li}.

If the skyrmion remains approximately rigid, then for $\alpha=\xi$ its translational velocity is expected to follow the applied current. Hence, the two oscillatory current components above are expected to drive the skyrmion along a figure-eight trajectory. The purpose of this experiment is to examine whether this controlled motion is reproduced by the fully discrete scheme while the skyrmion remains stable.

To quantify the motion, we track the approximate centroid of the $-z$-pointing core, i.e. center of the region where $m_{h,3}<0$. Specifically, with $w_h(t)=\bigl(\max\{-m_{h,3}(t),0\}\bigr)^2$,
we define
\[
X_h(t)
=
\frac{\int_\Omega x_1 \,w_h(t)\d \xx}
{\int_\Omega w_h(t)\d \xx},
\qquad
Y_h(t)
=
\frac{\int_\Omega x_2 \,w_h(t)\d \xx}
{\int_\Omega w_h(t)\d \xx}.
\]
Restricting the weight to the $-z$-pointing core prevents the DMI-induced
boundary canting from contaminating the computed center.

Figures~\ref{fig: bloch steering figures} and~\ref{fig: bloch steering} summarize the experiment. The snapshots show that the skyrmion remains localized while its position changes in response to the applied current. The trajectory of $(X_h,Y_h)$ displays the resulting steering, while the prescribed current components show the corresponding control signal. In particular, the repeated changes in direction do not lead to collapse or expulsion of the skyrmion over the simulated time interval.

\begin{figure}[ht]
	\centering
	\begin{subfigure}[t]{0.48\textwidth}
		\centering
		\includegraphics[width=\linewidth]{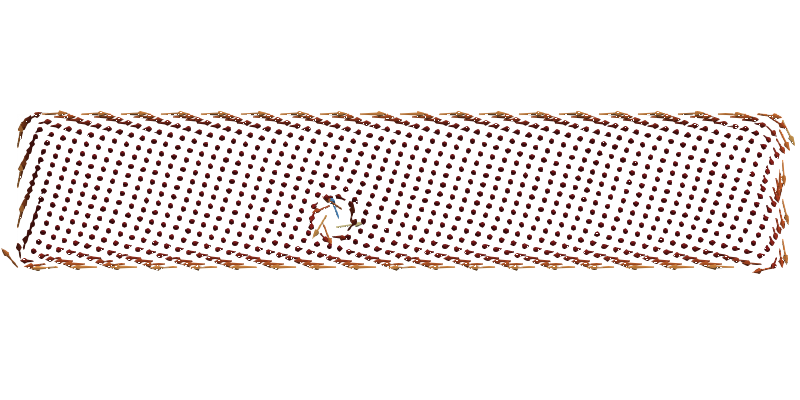}
		\captionsetup{skip=-10pt}
		\caption{$t=0.7\,\mathrm{ns}$}
	\end{subfigure}
	\begin{subfigure}[t]{0.48\textwidth}
		\centering
		\includegraphics[width=\linewidth]{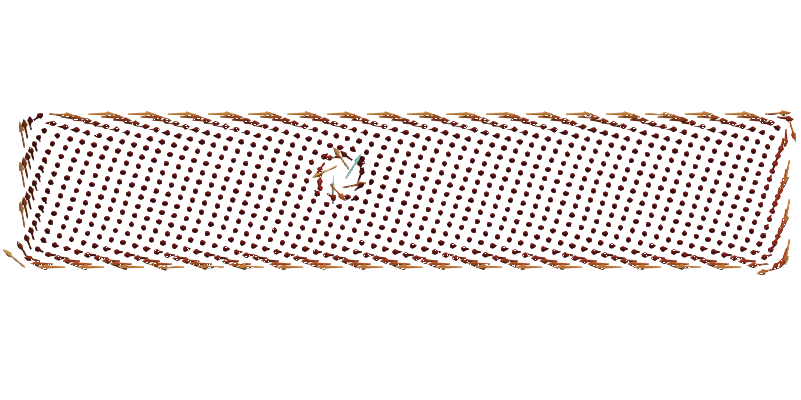}
		\captionsetup{skip=-10pt}
		\caption{$t=1.7\,\mathrm{ns}$}
	\end{subfigure}
	\begin{subfigure}[t]{0.48\textwidth}
		\centering
		\includegraphics[width=\linewidth]{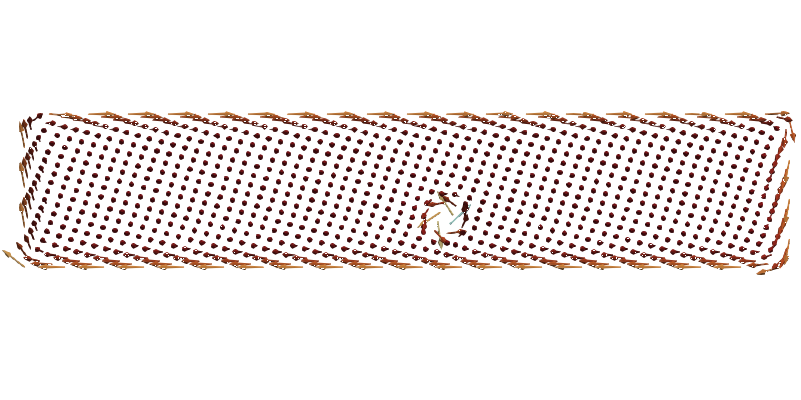}
		\captionsetup{skip=-10pt}
		\caption{$t=2.7\,\mathrm{ns}$}
	\end{subfigure}
	\begin{subfigure}[t]{0.48\textwidth}
		\centering
		\includegraphics[width=\linewidth]{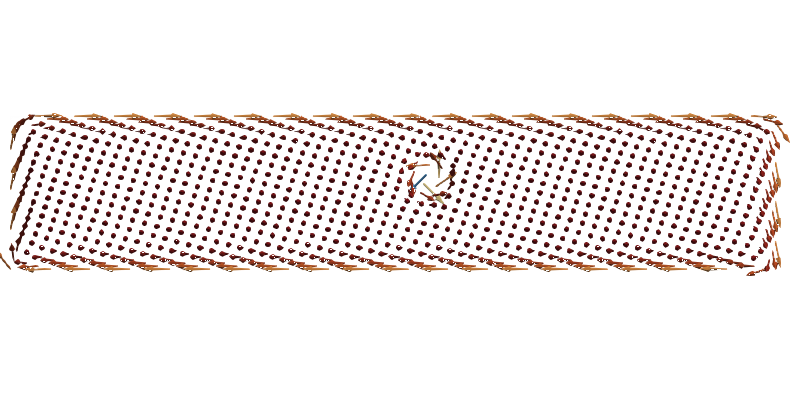}
		\captionsetup{skip=-10pt}
		\caption{$t=3.7\,\mathrm{ns}$}
	\end{subfigure}
	\caption{
	Snapshots of the Bloch skyrmion position at selected times.
}
\label{fig: bloch steering figures}
\end{figure}

	\vspace{1ex}
	
\begin{figure}[ht]
	\centering
	\begin{subfigure}[t]{0.48\textwidth}
		\centering
		\includegraphics{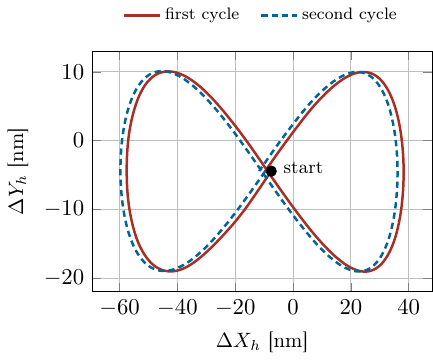}
		\caption{Trajectory of the skyrmion centre.}
	\end{subfigure}
	\hspace{2mm}
	\begin{subfigure}[t]{0.48\textwidth}
		\centering
		\includegraphics{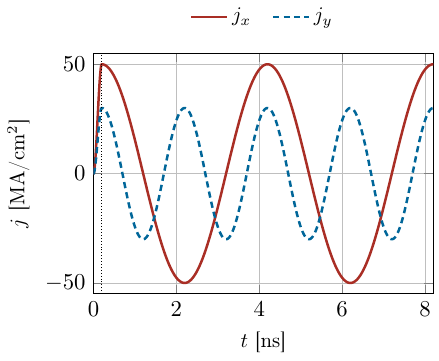}
		\caption{Prescribed current density.}
	\end{subfigure}
	
	\caption{
		Current-driven steering of a Bloch skyrmion.
		The left panel shows the displacement of the skyrmion centre over the first and second cycles, where $\Delta X_h(t)=X_h(t)-X_h(0)$ and $\Delta Y_h(t)=Y_h(t)-Y_h(0)$. The right panel shows the two components of the applied	current density. The dotted vertical line indicates the end of
		the initial smooth ramp.
	}
	\label{fig: bloch steering}
\end{figure}


\appendix
\section{}\label{appendix}
\renewcommand{\thetheorem}{A.\arabic{theorem}}
\setcounter{theorem}{0} 
We recall some auxiliary results that are used in the proofs of the present paper. The detailed proofs of these results can be found in~\cite{part1}.
\begin{lemma}[{Norm equivalence \cite[Lemmas~2.1, 2.2]{abp2024}, \cite[Lemmas 3.3, 3.4]{part1}}]\label{lem: norm equivalence 2}
	For every $n = 1,\dots, N$ it holds 
	\begin{equation}\label{eq: norm equivalence 2}
	 C_{\rm I} \Big( \tau\sum_{j=0}^{n-1}\norm{\vv_h^j}^2_\Omega \Big)^{1/2}
	 \le 
	 \Big( \tau\sum_{j=0}^{n-1}\norm{\d_t\mm_h^{j+1}}^2_\Omega \Big)^{1/2}
	 \le 
	 C_{\rm II} \Big( \tau\sum_{j=0}^{n-1}\norm{\vv_h^j}^2_\Omega \Big)^{1/2}
	\end{equation} 
	and
	\begin{equation}\label{eq: norm equivalence grad 2}
	 C_{\rm I} \Big( \tau\sum_{j=0}^{n-1}\norm{\nabla\vv_h^j}^2_\Omega \Big)^{1/2}
	 \le \Big( \tau\sum_{j=0}^{n-1}\norm[\big]{\nabla\d_t\mm_h^{j+1}}^2_\Omega \Big)^{1/2}
	 \le C_{\rm II} \Big( \tau\sum_{j=0}^{n-1}\norm{\nabla\vv_h^j}^2_\Omega \Big)^{1/2},
	\end{equation} 
	where $C_{\rm I}=\sqrt{1/5}$ and $C_{\rm II} = \sqrt{9/7}$. Moreover, it also holds that 
	\begin{equation}\label{eq: inverse estimate}
 \Big( \tau\sum_{j=2}^n\normO{\d_t^2\mm_h^j}^2 \Big) ^{1/2}
 \le
 \frac{1}{{\tau}}C_{\rm III} \Big( \tau\sum_{j=0}^{n-1}\normO{\vv_h^j}^2 \Big)^{1/2}
\end{equation}
and
\begin{equation}\label{eq: inverse estimate grad}
 \Big( \tau\sum_{j=2}^n\normO[\big]{\nabla\d_t^2\mm_h^j}^2 \Big) ^{1/2}
 \le
 \frac{1}{\tau} C_{\rm III} \Big( \tau\sum_{j=0}^{n-1}\normO{\nabla\vv_h^j}^2 \Big)^{1/2},
\end{equation}
where $C_{\rm III} = \sqrt{18/7}$.
\end{lemma}
\begin{lemma}[Useful identities]\label{lem: identities}
	For all $j=1,\dots,N-1$, the discrete approximations generated by the BDF2 update of Algorithm~\ref{alg: full discr} satisfy:
	\begin{align}
	    \vv_h^0 &= \d_t \mm_h^1,\label{eq: identity0}\\
		2\vv_h^j &= 3\d_t\mm_h^{j+1}-\d_t\mm_h^j;\label{eq: identity1}\\
		2(\vv_h^j-\d_t\mm_h^{j+1}) &= \tau \d_t^2\mm_h^{j+1};\label{eq: identity2}\\
	\mm_h^{j+\frac{1}{2}}-\mm_h^{j-\frac{1}{2}}&= \frac{\tau}{2}[\d _t\mm_h^{j+1}+\d _t\mm_h^j].\label{eq: identity6}
	\end{align}
\end{lemma}

\begin{lemma}[{Discrete Gronwall inequality
    \cite[Lemma~10.5]{thomee2013galerkin}}]
    \label{lem: gronwall}
Let $\left\{a_n\right\}_{n \geq 0},$\\$\left\{b_n\right\}_{n \geq 0}$, and $\left\{w_n\right\}_{n \geq 0}$ be sequences of real numbers such that
$$
0\leq a_n \leq a_{n+1}, \quad b_n \geq 0, \quad \text { and } \quad w_n \leq a_n+\sum_{j=0}^{n-1} b_j w_j \quad \text { for all } n \geq 0 .
$$
Then, it holds that
\begin{equation*}
	w_n \leq a_n \exp \sum_{j=0}^{n-1} b_j \quad \text { for all } n \geq 0.\tag*{\qed}
\end{equation*}
\end{lemma}

\printbibliography

\end{document}